\documentclass[reqno, 10pt]{article}
\usepackage[utf8]{inputenc}
\usepackage[T1]{fontenc}

\usepackage{amsfonts,amsmath,amssymb,mathrsfs,amsthm}
\usepackage[Symbolsmallscale]{upgreek}
\usepackage{bbm}

\usepackage[paperwidth=210mm,
paperheight=297mm,
left=3cm,right=3cm,top=2cm,bottom=2.5cm]{geometry}
\usepackage{authblk}
\usepackage[numbers]{natbib}
\usepackage{graphicx}
\usepackage{color}
\usepackage[colorlinks = true, citecolor = blue]{hyperref}
\usepackage{enumitem} %label option for enumerate \alph* \Alph* \roman* \Roman* or other... [label=$\bullet$], [label=\alph*]
\usepackage{algorithm}
\usepackage{algorithmic}

 \usepackage{ifthen}
 \usepackage{xargs}
\usepackage{aliascnt}
\usepackage{cleveref}
\usepackage{autonum}
\makeatletter
\newtheorem{theorem}{Theorem}
\crefname{theorem}{theorem}{Theorems}
\Crefname{theorem}{Theorem}{Theorems}

\newaliascnt{lemma}{theorem}
\newtheorem{lemma}[lemma]{Lemma}
\aliascntresetthe{lemma}
\crefname{lemma}{lemma}{lemmas}
\Crefname{lemma}{Lemma}{Lemmas}

\newaliascnt{corollary}{theorem}
\newtheorem{corollary}[corollary]{Corollary}
\aliascntresetthe{corollary}
\crefname{corollary}{corollary}{corollaries}
\Crefname{corollary}{Corollary}{Corollaries}

\newaliascnt{proposition}{theorem}
\newtheorem{proposition}[proposition]{Proposition}
\aliascntresetthe{proposition}
\crefname{proposition}{proposition}{propositions}
\Crefname{proposition}{Proposition}{Propositions}

\newaliascnt{remark}{theorem}
\newtheorem{remark}[remark]{Remark}
\aliascntresetthe{remark}
\crefname{remark}{remark}{remarks}
\Crefname{Remark}{Remark}{Remarks}

 \crefname{example}{example}{examples}
 \Crefname{Example}{Example}{Examples}

\crefname{figure}{figure}{figures}
\Crefname{Figure}{Figure}{Figures}

\newtheorem{assumption}{\textbf{H}\hspace{-3pt}}
\Crefname{assumption}{\textbf{H}\hspace{-3pt}}{\textbf{H}\hspace{-3pt}}
\crefname{assumption}{\textbf{H}}{\textbf{H}}

\Crefname{assumptionG}{\textbf{G}\hspace{-3pt}}{\textbf{G}\hspace{-3pt}}
\crefname{assumptionG}{\textbf{G}}{\textbf{G}}

\newtheorem{assumptionA}{\textbf{A}\hspace{-3pt}}
\Crefname{assumptionA}{\textbf{A}\hspace{-3pt}}{\textbf{A}\hspace{-3pt}}
\crefname{assumptionA}{\textbf{A}}{\textbf{A}}

\def\calU{F}
\def\bfX{\mathbf{X}}
\def\bfB{\mathbf{B}}

\def\thetaEB{\bar{\theta}_N}

\def\thetaStar{\theta_{\star}}

\def\tx{\tilde{x}}
\def\dim{d}
\def\dimy{d_y}
\def\dimtheta{d_{\Theta}}

\def\brho{\bar{\rho}}
\def\bA{\bar{A}}
\def\bB{\bar{B}}
\def\hash{\sharp}
\def\ftt{\mathtt{f}}
\def\Rvun{R_{V,1}}
\def\Rvdeux{R_{V,2}}
\def\Rvtrois{R_{V,3}}
\def\Ruun{R_{U,1}}
\def\Rudeux{R_{U,2}}
\def\xvstar{x_{\theta}^{\star}}
\def\xustar{x_{\theta}^{\hash}}
\def\bxvstar{\bar{x}_{\theta}^{\star}}
\def\bxustar{\bar{x}_{\theta}^{\hash}}

\def\Tg{\mathcal{T}_{\gamma, \theta}}

\def\bU{\bar{U}}
\def\bV{\bar{V}}
\def\bH{\bar{H}}
\def\bpi{\bar{\pi}}
\newcommand{\cball}[2]{\overline{\operatorname{B}}(#1,#2)}

\def\Munu{\Mt_{\Theta}}

\def\Lambdabf{\mathbf{\Lambda}}
\def\bLambdabf{\bar{\mathbf{\Lambda}}}
\def\ct{\mathtt{c}}
\def\tpi{\tilde{\pi}}
\def\bkappa{\bar{\kappa}}
\def\ukappa{\underbar{$\kappa$}}
\def\cT{\mathcal{T}}

\def\pow{m}
\def\step{\ceil{1/\gamma}}
\def\bfDd{\mathbf{D}_{\mathrm{d}}}
\newcommand{\tup}[1]{\textup{#1}}
\newcommand{\mtt}{\mathtt{m}}

\newcommand{\prox}{\operatorname{prox}}
\newcommand{\Kker}{\mathrm{K}}
\newcommand{\bKker}{\bar{\mathrm{K}}}
\newcommand{\bRker}{\bar{\mathrm{R}}}
\newcommand{\bSker}{\bar{\mathrm{S}}}
\newcommand{\Rker}{\mathrm{R}}
\newcommand{\Sker}{\mathrm{S}}
\newcommand{\Lt}{\mathtt{L}}
\newcommand{\Mt}{\mathtt{M}}
\newcommand{\Rtheta}{R_{\Theta}}

\def\x{{x}}

\def\y{{y}}

\def\w{{w}}
\def\g{{g}}
\def\f{{f_y}}

\def\L{{L_f}}
\newcommandx{\norm}[2][1=]{\ifthenelse{\equal{#1}{}}{\left\Vert #2 \right\Vert}{\left\Vert #2 \right\Vert^{#1}}}
\newcommandx{\normLigne}[2][1=]{\ifthenelse{\equal{#1}{}}{\Vert #2 \Vert}{\Vert #2\Vert^{#1}}}

\def\msa{\mathsf{A}}

\def\msu{\mathsf{U}}

\def\mcbb{\mathcal{B}}  %%% \mcb est déjà pris
\newcommand{\mcb}[1]{\mathcal{B}(#1)}

\def\mcf{\mathcal{F}}

\def\rset{\mathbb{R}}

\def\cset{\mathbb{C}}

\def\nset{\mathbb{N}}
\def\nsets{\mathbb{N}^*}

\def\rmd{\mathrm{d}}

\def\rme{\mathrm{e}}

\def\rmc{\mathrm{C}}

\newcommand{\R}{\mathbb R}

\def\xstar{x^\star}

\newcommandx{\functionspace}[2][1=+]{\mathbb{F}_{#1}(#2)}
\newcommand{\argmax}{\operatorname*{arg\,max}}
\newcommand{\argmin}{\operatorname*{arg\,min}}

\newcommandx{\VarDeux}[3][3=]{\operatorname{Var}^{#3}_{#1}\left\{#2 \right\}}

\newcommand{\1}{\mathbbm{1}}

\newcommand{\LeftEqNo}{\let\veqno\@@leqno}

\newcommand{\floor}[1]{\left\lfloor #1 \right\rfloor}
\newcommand{\ceil}[1]{\left\lceil #1 \right\rceil}

\newcommand{\N}{\ensuremath{\mathbb{N}}}

\newcommand{\PE}{\mathbb{E}}
\newcommand{\PP}{\mathbb{P}}

\newcommand{\abs}[1]{\left\vert #1 \right\vert}
\newcommand{\absLigne}[1]{\vert #1 \vert}
\newcommand{\tvnorm}[1]{\| #1 \|_{\mathrm{TV}}}

\newcommandx{\Vnorm}[2][1=V]{\| #2 \|_{#1}}
\newcommandx{\VnormEq}[2][1=V]{\left\| #2 \right\|_{#1}}
\newcommand{\parenthese}[1]{\left(#1 \right)}
\newcommand{\parentheseLigne}[1]{(#1 )}
\newcommand{\parentheseDeux}[1]{\left[ #1 \right]}

\newcommand{\defEns}[1]{\left\lbrace #1 \right\rbrace }
\newcommand{\defEnsLigne}[1]{\lbrace #1 \rbrace }

\newcommandx\probaMarkovTilde[2][2=]
{\ifthenelse{\equal{#2}{}}{{\widetilde{\mathbb{P}}_{#1}}}{\widetilde{\mathbb{P}}_{#1}\left[ #2\right]}}

\newcommand{\expe}[1]{\PE \left[ #1 \right]}

\newcommand{\plusinfty}{+\infty}

\def\ie{\textit{i.e.}}
\def\as{\textit{a.s}}

\def\eqsp{\;}
\newcommand{\coint}[1]{\left[#1\right)}
\newcommand{\ocint}[1]{\left(#1\right]}
\newcommand{\ooint}[1]{\left(#1\right)}
\newcommand{\ccint}[1]{\left[#1\right]}

\newcommandx{\weight}[2][2=n]{\omega_{#1,#2}^N}

\newcommand{\ball}[2]{\operatorname{B}(#1,#2)}
\newcommand{\boulefermee}[2]{\overline{\operatorname{B}}(#1,#2)}

\def\TV{\mathrm{TV}}

\def\as{\ensuremath{\text{a.s.}}}

\newcommandx\sequence[3][2=,3=]
{\ifthenelse{\equal{#3}{}}{\ensuremath{\{ #1_{#2}\}}}{\ensuremath{\{ #1_{#2}, \eqsp #2 \in #3 \}}}}
\newcommandx\sequenceD[3][2=,3=]
{\ifthenelse{\equal{#3}{}}{\ensuremath{\{ #1_{#2}\}}}{\ensuremath{( #1)_{ #2 \in #3} }}}

\newcommandx{\sequencen}[2][2=n\in\N]{\ensuremath{\{ #1_n, \eqsp #2 \}}}
\newcommandx\sequenceDouble[4][3=,4=]
{\ifthenelse{\equal{#3}{}}{\ensuremath{\{ (#1_{#3},#2_{#3}) \}}}{\ensuremath{\{  (#1_{#3},#2_{#3}), \eqsp #3 \in #4 \}}}}
\newcommandx{\sequencenDouble}[3][3=n\in\N]{\ensuremath{\{ (#1_{n},#2_{n}), \eqsp #3 \}}}

\newcommand{\wrt}{w.r.t.}

\newcommand{\opnorm}[1]{{\left\vert\kern-0.25ex\left\vert\kern-0.25ex\left\vert #1 
    \right\vert\kern-0.25ex\right\vert\kern-0.25ex\right\vert}}

\def\Id{\operatorname{Id}}

\newcommandx{\CPE}[3][1=]{{\mathbb E}_{#1}\left[#2 \left \vert #3 \right. \right]} %%%% esperance conditionnelle
\newcommandx{\CPVar}[3][1=]{\mathrm{Var}^{#3}_{#1}\left\{ #2 \right\}}
\newcommand{\CPP}[3][]
{\ifthenelse{\equal{#1}{}}{{\mathbb P}\left(\left. #2 \, \right| #3 \right)}{{\mathbb P}_{#1}\left(\left. #2 \, \right | #3 \right)}}

\newcommandx{\osc}[2][1=]{\mathrm{osc}_{#1}(#2)}

\def\Id{\operatorname{Id}}

\def\V{V}

\def\w{w}
\def\y{y}

\def\bgamma{\bar{\gamma}}
\def\bU{\bar{U}}

\def\bX{\bar{X}}

\def\pib{\bar{\pi}}
\def\bpi{\pib}

\def\Mt{\mathtt{M}}
\def\tM{\Mt}

\def\tx{\tilde{x}}

\def\bX{\bar{X}}

\def\bH{\bar{H}}

\newcommand{\ensembleLigne}[2]{\{#1\,:\eqsp #2\}}

\newcommand\coupling[2]{\Gamma(\mu,\nu)}

\def\tpi{\tilde{\pi}}

\newcommand{\complementary}{\mathrm{c}}

\renewcommand{\geq}{\geqslant}
\renewcommand{\leq}{\leqslant}

\def\vareps{\varepsilon}

\def\Phibf{\mathbf{\Phi}}
\def\Psibf{\mathbf{\Psi}}
\def\bPsibf{\bar{\mathbf{\Psi}}}
\def\tPsibf{\tilde{\mathbf{\Psi}}}

\newcommandx{\KL}[2]{\mathrm{KL}\left( #1 | #2 \right)}
\newcommandx{\KLbig}[2]{\mathrm{KL}\left( #1 \middle| #2 \right)}

\title{Maximum likelihood estimation of regularisation parameters in high-dimensional inverse problems: an empirical Bayesian approach\\
	Part II: Theoretical Analysis}
\author[1]{Valentin De Bortoli \footnote{Email: debortoli@cmla.ens-cachan.fr}}
\author[1]{Alain Durmus \footnote{Email: durmus@cmla.ens-cachan.fr }}
\author[2]{Marcelo Pereyra \footnote{Email: m.pereyra@hw.ac.uk} }
\author[2]{Ana F. Vidal \footnote{Email: af69@hw.ac.uk} \newline \indent \  Part of this work has been presented at the 25th IEEE International Conference on Image Processing (ICIP)  \cite{vidal2018maximum} }

\affil[1]{\small CMLA - \'Ecole normale supérieure Paris-Saclay, CNRS, Université Paris-Saclay, 94235 Cachan, France.}
\affil[2]{\small Maxwell Institute for Mathematical Sciences \& School of Mathematical and Computer Sciences, Heriot-Watt University, Edinburgh, EH14 4AS, United Kingdom.}

\begin{document}
\maketitle
\begin{abstract}
  This paper presents a detailed theoretical analysis of the three
  stochastic approximation proximal gradient algorithms proposed in
  our companion paper \cite{vidal:et:al:2019a} to set regularization
  parameters by marginal maximum likelihood estimation. We prove the convergence
  of a more general stochastic approximation scheme that includes the
  three algorithms of \cite{vidal:et:al:2019a} as special
  cases. This includes
  asymptotic and non-asymptotic convergence results with natural and
  easily verifiable conditions, as well as explicit bounds on the
  convergence rates. Importantly, the theory is also general in that it can be applied to other intractable
  optimisation problems. A main novelty of the work is that the stochastic gradient estimates of our scheme
  are constructed from inexact proximal Markov chain Monte Carlo samplers. This allows the use of samplers that
  scale efficiently to large problems and for which we have precise
  theoretical guarantees.
	\end{abstract}
	%
%	\begin{keywords}
%		\emph{Image processing, inverse problems, statistical inference, empirical Bayes, stochastic optimisation, Markov chain Monte Carlo methods, proximal algorithms.}
%	\end{keywords}
	\section{Introduction}
	\label{sec:intro}
        Numerous imaging problems require performing inferences on an unknown image of interest $\x \in \rset^{\dim}$ from some observed data $\y$. Canonical examples include image denoising
        \cite{chouzenoux2015convex,kech2017optimal}, compressive
        sensing \cite{donoho2006compressed,ravishankar2015efficient},
        super-resolution
        \cite{morgenshtern2016super,zhang2018residual}, tomographic
        reconstruction \cite{chung2017motion}, image inpainting
        \cite{galerne2017texture,schonlieb2015partial}, source
        separation
        \cite{bioucas2012hyperspectral,berisha2015deblurring}, fusion
        \cite{simoes2015convex,li2017pixel}, and phase retrieval
        \cite{candes2015phase,iwen2016fast}. Such imaging problems can be formulated in a Bayesian statistical framework, where inferences are derived from the so-called posterior distribution of $\x$ given $\y$, which for the purpose of this paper we specify as follows
        $$
        p(\x|\y,\theta) = p(\y|\x)p(\x|\theta)/p(\y|\theta)
        $$
        where $p(\y|\x) = \exp\{-\f(\x)\}$ with $\f \in \rmc^1(\rset^{\dim}, \rset)$ is the likelihood function, and the prior distribution is $p(\x|\theta) =  \exp\{-\theta^\top g(\x)\}$ with $g: \rset^{\dim} \to \rset^{\dimtheta}$ and $\theta \in \Theta \subset \rset^{\dimtheta}$. The function $\f$ acts as a data-fidelity term, $g$ as a regulariser { that promotes desired structural or regularity properties (e.g., smoothness, piecewise-regularity, or sparsity  \cite{chambolle2016introduction}), and $\theta$ is a regularisation parameter that controls the amount of regularity enforced. Most Bayesian methods in the imaging literature consider models for which $\f$ and $\g$ are convex functions and report as solution the maximum-a-posteriori (MAP) Bayesian estimator 
 \begin{equation}
   \label{eqII:model}
   \text{ argmin }  f_{y, \theta} \eqsp, \text{ where }
   f_{y, \theta}(x) = f_y(x) + \theta^{\top} g(x)  \text{ for any $x \in \rset^d$} \eqsp.
 \end{equation}
  For example, many imaging works consider a linear observation model of the form $\y = A\x + \w$, where $A \in \rset^{\dim} \times \rset^{\dim}$ is some problem-specific linear operator and the noise $\w$ has distribution  $\mathrm{N}(0,\sigma^2\mathbb{I}_\dim)$ with variance $\sigma^2 > 0$. Then, for any
 $x \in \rset^{\dim}$ $f_y(x) = (2 \sigma^2)^{-1} \normLigne{A x - y}^2$. With regards to the prior, a common choice in imaging is to set $\Theta = \mathbb{R}^+$ and $g(x) = \|Bx\|_1$ for some suitable basis or dictionary $B \in \rset^{\dim^\prime} \times \rset^{\dim}$, or  
  $g(x) = \mathrm{TV}(x)$, where $\mathrm{TV}(\x)$ is the isotropic total
 variation pseudo-norm given by $\mathrm{TV}(\x)=\sum_i\sqrt{(\Delta_i^h \x)^2+(\Delta_i^v \x)^2}$
 where $\Delta_i^v$ and $\Delta_i^h$ denote horizontal and vertical first-order local (pixel-wise) difference operators. 
 
Importantly, when $\f$ and $\g$ are convex, problem \eqref{eqII:model} is also convex and can usually be efficiently solved by using modern proximal convex optimisation techniques \cite{chambolle2016introduction}, with remarkable guarantees on the solutions delivered.

Setting  the value of $\theta$ can be notoriously difficult, especially in problems that are ill-posed or ill-conditioned where the regularisation has a dramatic impact on the recovered estimates. 
 We refer to \cite{kaipio2006statistical} and \cite[\Cref{sec:intro}]{vidal:et:al:2019a} for illustrations and a detailed review of the existing methods for setting set $\theta$.

In our companion paper \cite{vidal:et:al:2019a}, we present a new
 method to set regularisation parameters. More precisely, in
\cite{vidal:et:al:2019a}, we adopt an empirical Bayesian approach and
 set $\theta$ by maximum marginal likelihood estimation, \ie 
   \begin{equation}
   \label{eq:def_theta_star_intro_part_II}
   \theta_{\star} \in \argmax_{\theta \in \Theta} \log p(y|\theta)  \eqsp, \text{ where } p(y|\theta ) = \int_{\rset^d}p(y,x | \theta)\rmd x \eqsp, \quad p(y,x | \theta) \propto \exp[-f_{y,\theta}(x)] \eqsp. 
 \end{equation}
 To solve \eqref{eq:def_theta_star_intro_part_II}, we aim at using gradient
 based optimization methods. The gradient of
 $\theta \mapsto \log p(y | \theta)$, can be computed using Fisher's
 identity, see \cite[Proposition A.1]{vidal:et:al:2019a},
 which implies under mild integrability conditions on $f_y$ and $g$, for any $\theta \in \Theta$,
 \begin{equation}
   \nabla_{\theta} \log p(y|\theta) = - \int_{\rset^{\dim}} g(\tilde{x}) p(\tilde{x} | y, \theta) \rmd \tilde{x} + \int_{\rset^{\dim}} g(\tilde{x}) p(\tilde{x} | \theta) \rmd \tilde{x} \eqsp .
 \end{equation}
 It follows that $\theta \mapsto \nabla_{\theta} \log p(y|\theta)$ can
 be written as a sum of two parametric integrals which are untractable
 in most cases. Therefore, we propose to use a stochastic
 approximation (SA) scheme and, in particular, we define three
 different algorithms to solve \eqref{eq:def_theta_star_intro_part_II}
 \cite[Algorithm 3.1, Algorithm 3.2, Algorithm
 3.3]{vidal:et:al:2019a}. These algorithms are extensively
 demonstrated in \cite{vidal:et:al:2019a} through a range of
 applications and comparisons with alternative approaches from the
 state-of-the-art.
 
 In the present paper we theoretically analyse these three SA schemes and
 establish natural and easily verifiable conditions for
 convergence. For generality, rather than presenting
 algorithm-specific analyses, we establish detailed convergence
 results for a more general SA scheme that covers the three algorithms
 of \cite{vidal:et:al:2019a} as specific cases.  Indeed, all these
 methods boil down to defining a sequence $(\theta_n)_{n \in \nset}$
 satisfying a recursion of the form: for any $n \in \nset$,
 \begin{equation}
   \label{eq:rec_intro_part_II}
   \theta_{n+1} = \Pi_{\Theta} \parentheseDeux{\theta_n - \frac{\delta_{n+1}}{m_{n}} \sum_{k=1}^{m_{n}} \defEns{g(X_{k}^n) - g(\bX_{k}^n)}} \eqsp ,
 \end{equation}
 where $\Pi_{\Theta}$ is the projection onto a convex closed set
 $\Theta$, $(X_k^n)_{k \in \{1, \dots, m_n\}}$ and
 $(\bar{X}_k^n)_{k \in \{1, \dots, m_n\}}$ are two independent
 stochastic processes targeting $x \mapsto p(x|y,\theta)$ and
 $x \mapsto p(x|\theta)$ respectively, $(m_n)_{n \in \nset}$ is a
 sequence of batch-sizes and $(\delta_n)_{n \in \nsets}$ is a sequence
 of stepsizes. In this paper, we are interested in establishing the
 convergence of the averaging of $(\theta_n)_{n \in \nset}$ to a
 solution of \eqref{eq:def_theta_star_intro_part_II} in this
 setting. SA has been extensively studied during the past decades
 \cite{robbins1951stochastic,kiefer1952stochastic,polyak:juditsky:1992,boydcandes,metivier:priouret:1984,metivier:priouret:1987,benveniste:metivier:priouret:1990,benaim:1996,tadic:doucet:2017}. Recently,
 quantitative results have been obtained in
 \cite{shamirzhang,bachmoulines2011,rakhlin2011making,atchade2017perturbed,rosasco2019convergence}. In contrast to \cite{atchade2017perturbed}, here we consider the case where
 $(X_k^n)_{k \in \{1, \dots, m_n\}}$ and
 $(\bar{X}_k^n)_{k \in \{1, \dots, m_n\}}$ are \textit{inexact} Markov
 chains which target $x \mapsto p(x|y,\theta)$ and
 $x \mapsto p(x|\theta)$ respectively and are based on some
 generalizations of the Unadjusted Langevin Algorithm (ULA)
 \cite{roberts:tweedie:1996}. In the recent years, ULA has attracted a
 lot of attention since this algorithm exhibits favorable
 high-dimensional convergence properties in the case where the target
 distribution admits a differentiable density, see
 \cite{durmus:moulines:2016,durmus2017nonasymptotic,dalalyan2017theoretical,
   dalalyan2019user}. However, in most imaging models, the penalty
 function $g$ is not differentiable and therefore
 $x \mapsto p(x|y,\theta)$ and $x \mapsto p(x|\theta)$ are not
 differentiable as well. Therefore, we consider proximal Langevin
 samplers which are specifically design to overcome this issue: the
 Moreau-Yoshida Unadjusted Langevin Algorithm (MYULA), see
 \cite{MYULA2016efficient}, and the Proximal Unadjusted Langevin
 Operator (PULA), see \cite{durmus2019analysis}.

 A similar approximation scheme to \eqref{eq:rec_intro_part_II} is
 studied in \cite{atchade2017perturbed}. More precisely \cite[Theorem
 3, Theorem 4]{atchade2017perturbed} are similar to \Cref{thm:cv_pula}
 and \Cref{thm:error_pula}.  Contrarily to that work, here we do not
 require the Markov kernels we use to exactly target
 $x \mapsto p(x|\theta)$ and $x \mapsto p(x|y, \theta)$ but allow some
 bias in the estimation which is accounted for in our convergence
 rates. This relaxation to biased estimates plays a central role in the capacity of the method to scale efficiently to large problems. Moreover, the present paper is also a complement of
 \cite{de2019efficient} which establishes general
 conditions for the convergence of inexact Markovian SA but only apply
 these results to ULA. In this study, we do not consider a general
 Markov kernel but rather specialize the results of
 \cite{de2019efficient} to MYULA and PULA Markov
 kernels. However, to apply results of
 \cite{de2019efficient}, new quantitative
 geometric convergence properties on MYULA and PULA have to be established.

 The remainder of the paper is organized as follows. In
 \Cref{sec:notation}, we recall our notations and conventions. In
 \Cref{sec:stoch-appr-image}, we define the class of optimisation
 problems considered and the SA scheme
 \eqref{eq:rec_intro_part_II}. This setting includes the optimization
 problem presented in \eqref{eq:def_theta_star_intro_part_II} and the
 three specific algorithms introduced in
 \cite{vidal:et:al:2019a}.
 Then, in \Cref{sec:convergence-properties}, we present a detailed
 analysis of the theoretical properties of the proposed
 methodology. First, we show new ergodicity results for the MYULA and
 PULA samplers. In a second part, we provide easily verifiable
 conditions for convergence and quantitative convergence rates for the
 averaging sequences designed from \eqref{eq:rec_intro_part_II}.  The
 proofs of these results are gathered in \Cref{thm:cv_pula_proof}.              

 \section{Notations and conventions}
\label{sec:notation}
We denote by $\ball{0}{R}$ and $\cball{0}{R}$ the open ball, respectively the closed ball, with radius $R$ in $\rset^{\dim}$. Denote by $\mathcal{B}(\rset^{\dim})$ the Borel $\sigma$-field of
$\rset^{\dim}$, $\functionspace[]{\rset^{\dim}}$ the set of all Borel measurable
functions on $\rset^{\dim}$ and for $f \in \functionspace[]{\rset^{\dim}}$,
$\Vnorm[\infty]{f}= \sup_{x \in \rset^{\dim}} \abs{f(x)}$.  For $\mu$ a probability measure
on $(\rset^{\dim}, \mathcal{B}(\rset^{\dim}))$ and $f \in
\functionspace[]{\rset^{\dim}}$ a $\mu$-integrable function, denote by
$\mu(f)$ the integral of $f$ \wrt~$\mu$. For $f \in \functionspace[]{\rset^{\dim}}$, the $V$-norm of $f$ is given by $\Vnorm[V]{f}= \sup_{x \in \rset^{\dim}} |f(x)|/V(x)$. Let $\xi$ be a finite signed measure on $(\rset^{\dim},\mcbb(\rset^{\dim}))$. The $V$-total variation norm of $\xi$ is defined as
\begin{equation}
\Vnorm[V]{\xi} = \sup_{f \in \functionspace[]{\rset^{\dim}}, \Vnorm[V]{f} \leq 1}  \abs{\int_{\rset^{\dim} } f(x) \rmd \xi (x)} \eqsp.
\end{equation}
If $V \equiv 1$, then $\Vnorm[V]{\cdot}$ is the total variation norm on measures  denoted by $\tvnorm{\cdot}$. 

Let $\msu$ be an open set of $\rset^{\dim}$. We denote by
$\rmc^{k}(\msu, \rset^{\dimtheta})$
the set of $\rset^{\dimtheta}$-valued $k$-differentiable functions, respectively
the set of compactly supported $\rset^{\dimtheta}$-valued $k$-differentiable
functions. $\rmc^k(\msu)$  stands $\rmc^k(\msu,\rset)$.  Let $f : \msu \to \rset$, we denote by $\nabla f$,
the gradient of $f$ if it exists. $f$ is said to be 
 $\mtt$-convex with $\mtt\geq 0$ if
for all $x,y \in \rset^{\dim}$ and $t \in \ccint{0,1}$,
\begin{equation}
f(t x + (1-t) y) \leq t f(x)  + (1-t) f(y) -(\mtt/2)t(1-t)  \norm[2]{x-y}  \eqsp.
\end{equation}
Let $(\Omega,\mcf,\PP)$ be a probability space. 
Denote by $\mu
\ll \nu$ if $\mu$ is absolutely continuous \wrt~$\nu$ and $\rmd \mu /
\rmd \nu$ an associated density. Let $\mu,\nu$ be two probability
measures on $(\rset^{\dim}, \mcbb(\rset^{\dim}))$. Define the Kullback-Leibler
divergence of $\mu$ from $\nu$ by
\begin{equation}
  \KL{\mu}{\nu} =
  \begin{cases}
    \int_{\rset^{\dim}} \frac{\rmd \mu}{\rmd \nu}(x) \log \parenthese{\frac{\rmd \mu}{\rmd \nu} (x)} \rmd \nu (x) \eqsp, & \text{if } \mu \ll \nu  \eqsp, \\
\plusinfty & \text{ otherwise} \eqsp.
  \end{cases}
\end{equation}

 \section{Proposed stochastic approximation proximal gradient optimisation methodology}	
 \label{sec:stoch-appr-image}
 \subsection{Problem statement}
 \label{sec:sett-intr-our}
Let $\Theta \subset \rset^{\dimtheta} $ and $f: \ \Theta\to \rset$. We consider the optimisation problem
 \begin{equation}
   \label{eqII:solve}
   \theta_{\star} \in \argmin_{\theta \in \Theta} f(\theta) \eqsp ,
 \end{equation}
in scenarios where it is not possible to evaluate $f$ nor $\nabla f$ because they are computationally intractable. Problem \eqref{eqII:solve} includes the marginal likelihood estimation problem \eqref{eq:def_theta_star_intro_part_II} of our companion paper \cite{vidal:et:al:2019a} as the special case $f = -\log p(y|\cdot)$. We make the following general assumptions on $f$ and $\Theta$, which are in particular verified by the imaging models considered in \cite{vidal:et:al:2019a}.

\begin{assumptionA}
  \label{assum:theta_compact}
  $\Theta$ is a convex compact set and $\Theta \subset \boulefermee{0}{\Rtheta}$ with $\Rtheta>0$.
\end{assumptionA}

\begin{assumptionA}
  \label{assum:f_grad_lip}
  There exist an open set $\msu \subset \rset^{p}$ and $\L \geq 0$ such that $\Theta \subset \msu$, $f \in \rmc^1(\msu,\R)$ and for any $\theta_1, \theta_2 \in \Theta$
  \begin{equation}
    \| \nabla_{\theta} f (\theta_1) - \nabla_{\theta} f (\theta_2) \| \leq \L \| \theta_1 - \theta_2 \| \eqsp .
  \end{equation}
  \end{assumptionA}

\begin{assumptionA}
  \label{assum:grad_expec}
  For any $\theta \in \Theta$, there exist $H_{\theta}, \bH_{\theta}: \ \rset^{\dim} \ \rightarrow \ \rset^{\dimtheta}$ and two probability distributions $\pi_{\theta}, \bpi_{\theta}$ on $(\rset^{\dim}, \mcbb(\rset^{\dim}))$ satisfying for any $\theta \in \Theta$
  \begin{equation}
    \nabla_{\theta} f(\theta) =  \int_{\rset^{\dim}} H_{\theta}(x) \rmd \pi_{\theta}(x) + \int_{\rset^{\dim}} \bH_{\theta}(x) \rmd \bpi_{\theta}(x) \eqsp .
    \end{equation}
  In addition, $(\theta, x) \mapsto H_{\theta}(x)$ and $(\theta, x) \mapsto \bH_{\theta}(x)$ are measurable.
\end{assumptionA}

\begin{remark}
  Note that if $f \in \rmc^2(\Theta)$ then
  \tup{\Cref{assum:f_grad_lip}} is automatically satisfied under
  \tup{\Cref{assum:theta_compact}}, since $\Theta$ is compact.  In
  every model considered in our companion paper
  \cite{vidal:et:al:2019a}, $\theta \mapsto -\log p(y|\theta)$ is
  continuously twice differentiable on each compact using the
  dominated convergence theorem and therefore
  \tup{\Cref{assum:f_grad_lip}} holds under
  \tup{\Cref{assum:theta_compact}}.
\end{remark}

\begin{remark}
  \label{rem:part_II_setting}
  Assumption \tup{\Cref{assum:grad_expec}} is verified in the three
  cases considered in our companion paper \cite[Algorithm 3.1,
  Algorithm 3.2, Algorithm 3.3]{vidal:et:al:2019a}:
\begin{enumerate}[wide, labelwidth=!, labelindent=0pt, label=(\alph*)]
\item if the regulariser $g$ is $\alpha$ positively homogeneous with $\alpha >0$ and $\dimtheta = 1$, corresponding to \cite[Algorithm 3.1]{vidal:et:al:2019a}, then for any $\theta \in \Theta$, $H_{\theta} = g$, $\bH_{\theta} = -\dim/(\alpha\theta)$, $\pi_{\theta}$ is the probability measure with density \wrt~the Lebesgue measure $x \mapsto  p(x |y , \theta)$ and $\bpi_{\theta}$ is any probability measure;
\item if the regulariser $g$ is separably positively homogeneous as in
  \cite[Algorithm 3.2]{vidal:et:al:2019a}, then
  for any $\theta \in \Theta$, $H_{\theta} = g$,
  $\bH_{\theta} = (-\abs{\msa_i}/(\alpha_i\theta^i))_{ i \in \{1,
    \dots, \dimtheta\}}$, $\pi_{\theta}$ is the probability measure
  with density \wrt~the Lebesgue measure $x \mapsto p(x |y , \theta)$
  and $\bpi_{\theta}$ is any probability measure;
\item   \label{rem:part_II_setting_3} if the regulariser $g$ is inhomogeneous, corresponding to
  \cite[Algorithm 3.3]{vidal:et:al:2019a},
  then for any $\theta \in \Theta$, $\bH_{\theta} = -g$, $H_{\theta} = g$,
  $\pi_{\theta}$ and $ \bpi_{\theta}$ are the probability measures associated with the posterior and the prior, with
  density \wrt~the Lebesgue measure $x \mapsto p(x |y , \theta)$ and
  $x \mapsto p(x| \theta)$ respectively.
\end{enumerate}
\end{remark}

We now present in \Cref{algo:general}, the stochastic algorithm we
consider in order to solve \eqref{eqII:solve}. This method encompasses
the schemes introduced in the companion paper \cite[Algorithm 3.1,
Algorithm 3.2, Algorithm 3.3]{vidal:et:al:2019a}.  Starting from
$(X_{0}^0, \bX_{0}^0) \in \rset^{\dim}\times \rset^{\dim}$ and
$\theta_0 \in \Theta$, we define on a probability space
$(\Omega,\mcf,\mathbb{P})$, the sequence
$(\lbrace (X_{k}^n, \bX_{k}^n): k \in \lbrace 0, \dots, m_n \rbrace
\rbrace, \theta_n)_{n \in \N}$ by the following recursion for
$n \in \nset$ and $k \in \lbrace 0, \dots, m_n -1 \rbrace$
  \begin{equation}
    \label{eqII:algo_SOUL}
\begin{aligned}
   (X_{k}^n)_{ k \in \{0,\ldots,m_n\}} & \text{ is a MC with kernel } \Kker_{\gamma_n, \theta_n}    
   \text{ and } X_{0}^n = X_{m_{n-1}}^{n-1} \text{ given } \mathcal{F}_{n-1} \eqsp , \\
   (\bX_{k}^n)_{ k \in \{0,\ldots,m_n\}} & \text{ is a MC with kernel } \bKker_{\gamma_n', \theta_n}    
  \text{ and } \bX_{0}^n = \bX_{m_{n-1}}^{n-1} \text{ given } \mathcal{F}_{n-1} \eqsp , \\   
  \theta_{n+1} &= \Pi_{\Theta}\parentheseDeux{\theta_n - \frac{\delta_{n+1}}{m_{n}} \sum_{k=1}^{m_{n}} \defEns{H_{\theta_n}(X_{k}^n) + \bH_{\theta_n}(\bX_{k}^n)}} \eqsp ,
\end{aligned}
\end{equation}
where $(X_{m_{-1}}^{-1},\bar{X}_{m_{-1}}^{-1}) = (X_0^0,\bar{X}_0^0)$, $\{(\Kker_{\gamma,\theta}, \bKker_{\gamma,\theta}) \, : \, \gamma >0, \theta \in \Theta\}$ is a family of Markov kernels on $\rset^d \times \mcb{\rset^d}$, $(m_n)_{n \in \nset} \in (\nsets)^\nset$, $\delta_n, \gamma_n, \gamma_n' >0$ for any $n \in \nset$, $\Pi_{\Theta}$ is the projection onto $\Theta$ and  $\mathcal{F}_n$ is defined as follows for all $n \in \nset \cup \{-1\}$
\begin{equation}
  \label{eqII:def_F_n}
    \mathcal{F}_n = \sigma \left( \theta_0, \{(X_{k}^{\ell}, \bX_{k}^{\ell})_{k \in \{0,\ldots,m_\ell\}} \, : \, \ell \in \{0, \dots, n\}\} \right) \eqsp ,  \qquad \mcf_{-1} = \sigma(\theta_0, X_{0}^0, \bX_{0}^0) \eqsp .
  \end{equation}

  Define for any $N \in\nset$,
  \begin{equation}
    \thetaEB =  \left. \sum_{n=0}^{N-1} \delta_{n} \theta_n \middle/ \sum_{n=0}^{N-1} \delta_n \right.  \eqsp.
  \end{equation}
  In the sequel, we are interested in the convergence of $(    f(\thetaEB))_{N \in \nset}$ to a minimum of $f$ in the case where the Markov kernels $\{(\Kker_{\gamma,\theta}, \bKker_{\gamma,\theta}) \, : \, \gamma >0 , \theta \in \Theta\}$, used in \Cref{algo:general} are either the ones associated with MYULA or PULA. We now present these two MCMC methods for which some analysis is required in our study of $(    f(\thetaEB))_{N \in \nset}$.
 	\begin{algorithm}
		\caption{General algorithm}
		\label{algo:general}
		\begin{algorithmic}[1]		
                  \STATE Input: initial
                  $\{\theta_0, X_0^0, \bar{X}_0^0\}$,
                  $(\delta_n, \gamma_n, \gamma_n', m_n)_{n \in
                    \nset}$, number of iterations $N$.  \FOR{$n = 0$
                    to $N-1$} 
                \IF{$n>0$} 
                \STATE Set
                  $X_0^n = X_{m_{n-1}}^{n-1}$, 
                  \STATE Set
                  $\bar{X}^n_0 = \bar{X}^{n-1}_{m_{n-1}}$, \ENDIF
                  \FOR{$k=0$ to $m_n-1$} \STATE Sample
                  $X_{k+1}^{n} \sim \Kker_{\gamma_n, \theta_n}(X_k^n,
                  \cdot)$, \STATE Sample
                  $\bar{X}_{k+1}^{n} \sim \bar{\Kker}_{\gamma_n',
                    \theta_n}(\bar{X}_k^n, \cdot)$, \ENDFOR \STATE Set
                  $\theta_{n+1} = \Pi_{\Theta}\parentheseDeux{\theta_n
                    - \frac{\delta_{n+1}}{m_{n}} \sum_{k=1}^{m_{n}}
                    \defEns{H_{\theta_n}(X_{k}^n) +
                      \bH_{\theta_n}(\bX_{k}^n)}}$.  \ENDFOR \STATE
                  Output:
                  $\thetaEB =  \{\sum_{n=0}^{N-1} \delta_n\}^{-1}\sum_{n=0}^{N-1} \delta_{n} \theta_n$.
		\end{algorithmic}
              \end{algorithm}

 \subsection{Choice of MCMC kernels}
        \label{sec:choice-mcmc-kernels}
	Given the high dimensionality involved, it is fundamental to
        carefully choose the families of Markov kernels
        $\ensembleLigne{\Kker_{\gamma, \theta}, \bKker_{\gamma,
            \theta}}{\gamma >0, \theta \in
          \Theta}$ driving \Cref{algo:general}. In the 
        experimental part of this work, see
        \cite[Section 4]{vidal:et:al:2019a}, we use the MYULA Markov
        kernel recently proposed in \cite{MYULA2016efficient}, which
        is a state-of-the-art proximal Markov chain Monte Carlo (MCMC)
        method specifically designed for high-dimensional models that are are log-concave but not smooth. The
        method is derived from the discretisation of an over-damped
        Langevin diffusion, $(\bX_t)_{t \geq 0}$, satisfying the
        following stochastic differential equation
	\begin{equation}
	\label{EQII:PART_II_langevin}
	\rmd \bfX_t = - \nabla_x \calU(\bfX_t) \rmd t + \sqrt{2} \rmd \bfB_t \eqsp ,
	\end{equation}
	where $\calU: \rset^{\dim} \mapsto \rset$ is a continuously
        differentiable potential and $(\bfB_t)_{t \geq 0}$ is a
        standard $\dim$-dimensional Brownian motion. Under mild
        assumptions, this equation has a unique strong solution
        \cite[Chapter 4, Theorem
         2.3]{ikeda1989stochastic}. Accordingly, the law of
        $(X_t)_{t \geq 0}$ converges as $t\rightarrow\infty$ to the
        diffusion's unique invariant distribution, with probability
        density given by $\pi(x) \propto \rme^{ - \calU(\x)}$ for all
        $x \in \rset^{\dim}$ \cite[Theorem
        2.2]{roberts:tweedie:1996}. Hence, to use
        \eqref{EQII:PART_II_langevin} as a Monte Carlo method to
        sample from the posterior $p(x|y,\theta$), we set
        $F(x) = \log p(x|y,\theta)$ and thus specify the desired
        target density. Similarly, to sample from the prior we set
        $F(x) = -\nabla_x \log p(x|\theta)$.
	
	However, sampling directly from \eqref{EQII:PART_II_langevin} is usually not computationally feasible. Instead, we usually 
	resort to a discrete-time Euler-Maruyama approximation of \eqref{EQII:PART_II_langevin} that leads to the following Markov chain $(X_k)_{k \in \nset}$ with $X_0 \in \rset^{\dim}$, given for any $k \in \nset$ by
	\begin{equation}
	\label{eqII:part_ii_euler_maruyama_langevin}
	\textrm{ULA}: X_{k+1} = X_k - \gamma \nabla_{x}\calU(X_k) + \sqrt{2\gamma} Z_{k+1} ,
	\end{equation}
	where $\gamma > 0$ is a discretisation step-size and $(Z_k)_{k \in \N^*}$ is a sequence of i.i.d $\dim$-dimensional zero-mean Gaussian random variables with an identity covariance matrix. This Markov chain is commonly known as the Unadjusted Langevin Algorithm (ULA) \cite{roberts:tweedie:1996}. Under some additional assumptions on $F$, namely Lipschitz continuity of $\nabla_{x}\calU$, the ULA chain inherits the convergence properties of \eqref{EQII:PART_II_langevin} and converges to a stationary distribution that is close to the target $\pi$, with $\gamma$ controlling a trade-off between accuracy and convergence speed \cite{MYULA2016efficient}. 
	
        \begin{remark}
          \label{rem:part_II_ula}
          In this form, the ULA algorithm is limited to distributions where $\calU$ is a Lipschitz continuously differentiable function. However, in the imaging problems of interest this is usually not the case \cite{vidal:et:al:2019a}. For example, to implement any of the algorithms presented in \cite{vidal:et:al:2019a} it is necessary to sample from the posterior distribution $p(\x|\y,\theta)$ (corresponding to $\pi_{\theta}$ in \Cref{sec:sett-intr-our}), which would require setting for any $x \in \rset^d$, $\calU(\x)=\f(\x) + \theta^{\top}g(x)$. Similarly, one of the algorithms also requires sampling from the prior distribution $x\mapsto p(\x|\theta)$ (corresponding to $\bpi_{\theta}$ in \Cref{sec:sett-intr-our}), which requires setting for any $x \in \rset^d$, $\calU(\x)= \theta^{\top}g(x)$. In both cases, if $g$ is not smooth then ULA cannot be directly applied. The MYULA kernel was designed precisely to overcome this limitation. 
        \end{remark}

        \subsubsection{Moreau-Yoshida Unadjusted Langevin Algorithm}
        \label{sec:myula}
	Suppose that the target potential admits a decomposition
        $\calU = V + U$ where $V$ is Lipschitz differentiable and $U$
        is not smooth but convex over $\rset^{\dim}$. In MYULA, the differentiable part is handled via the
        gradient $\nabla_x V$ in a manner akin to ULA, whereas the
        non-differentiable convex part is replaced by a smooth approximation
        $U^\lambda(\x)$ given by the Moreau-Yosida envelope of $U$,
        see \cite[Definition 12.20]{bauschke2017convex}, defined for
        any $x\in \rset^{\dim}$ and $\lambda >0$ by
	\begin{equation}
	\label{eqII:part_ii_MY-envelope}
	U^\lambda(\x)=\underset{\tx \in \rset^{\dim}}{\mathrm{min}}~ \defEns{U(\tx)+(1/2\lambda) \norm{\x-\tx}_{2}^{2}} \eqsp .
      \end{equation}
      Similarly, we define the proximal operator for
        any $x\in \rset^{\dim}$ and $\lambda >0$ by
        \begin{equation}
          \label{eqII:part_ii_def_prox}
        \prox_U^{\lambda}(x) = \argmin_{\tx \in \rset^{\dim}} \defEns{U(\tx)+(1/2\lambda) \norm{\x-\tx}_{2}^{2}} \eqsp .
      \end{equation}
       For any $\lambda > 0$, the Moreau-Yosida envelope $U^{\lambda}$ is continuously differentiable with gradient given for any $x \in \rset^{\dim}$ by
        \begin{equation}
        \label{eqII:part_ii_grad_prox}  \nabla U^{\lambda}(x) = (x - \prox_U^{\lambda}(x)) / \lambda \eqsp,
        \end{equation}
        (see, e.g., \cite[Proposition 16.44]{bauschke2017convex}).
        Using this approximation we obtain the MYULA kernel associated
        with $(X_k)_{k \in \nset}$ given by $X_0 \in \rset^{\dim}$ and
        the following recursion for any $k \in \nset$
	\begin{equation}
	\label{eqII:part_ii_myula-generic}
	\textrm{MYULA}: X_{k+1} = X_k - \gamma \nabla_{x}V(X_k) - \gamma \nabla_{x}U^\lambda(X_k) + \sqrt{2\gamma} Z_{k+1}  \eqsp .
	\end{equation}	
	Returning to the imaging problems of interest, we define the MYULA families of Markov kernels
	$\ensembleLigne{\Rker_{\gamma, \theta}, \bRker_{\gamma,
			\theta}}{\gamma >0, \theta \in
		\Theta}$ that we use in \Cref{algo:general} to target
	$\pi_{\theta}$ and $\bpi_{\theta}$ for $\theta \in \Theta$
	as follows. By \Cref{rem:part_II_ula}, we set $V = \f$ and $U = \theta^{\top} g$, $\bar{V} = 0$
        and $\bar{U} = \theta^{\top}g$. Then, for any $\theta \in \Theta$ and $\gamma >0$,
        $\Rker_{\gamma, \theta}$ associated with $(X_k)_{k \in \nset}$ is
        given by $X_0 \in \rset^{\dim}$ and the following recursion
        for any $k \in \nset$
	\begin{equation}
          \label{EQII:PART_II_ MYULA_explicit_posterior}
          X_{k+1} = X_k - \gamma \nabla_{x}\f(X_k) - \gamma \defEns{X_k - \prox_{\theta^{\top}g}^{\lambda}(X_k)}/\lambda + \sqrt{2\gamma} Z_{k+1} \eqsp .
	\end{equation}
    Similarly, for any $\theta \in \Theta$ and $\gamma' >0$, $\bRker_{\gamma, \theta}$ associated with $(X_k)_{k \in \nset}$ is given by $X_0 \in \rset^{\dim}$ and the following recursion for any $k \in \nset$ 
	\begin{equation}
          \label{EQII:PART_II_ MYULA_explicit_prior}
          \bX_{k+1} = \bX_k - \gamma' \defEns{\bX_k - \prox_{\theta^{\top}g}^{\lambda'}(\bX_k)}/\lambda' + \sqrt{2\gamma} Z_{k+1} \eqsp ,
	\end{equation}		
       where we recall that $\lambda, \lambda^\prime>0$ are the smoothing parameters associated with $\theta^\top g^\lambda$, $\gamma, \gamma^\prime > 0$ are the discretisation steps and $(Z_k)_{k \in \N^*}$ is a sequence of i.i.d $\dim$-dimensional zero-mean Gaussian random variables with an identity covariance matrix.
       
       Notice that other ways of splitting the target potential $F$ can be straightforwardly implemented. For example, instead of a single non-smooth convex term $U$, one might choose a splitting involving several non-smooth terms to simplify the computation of the proximal operators (each term would be replaced by its Moreau-Yosida envelope in \eqref{EQII:PART_II_langevin}). Similarly, although we usually to associate $V, \bar{V}$ and $U,\bar{U}$ to the log-likelihood and the log-prior, some cases might benefit from a different splitting. Moreover, as illustrated in \Cref{sec:pula} below, other discrete approximations of the Langevin diffusion could be considered too.
       
       \subsubsection{Proximal Unadjusted Langevin Algorithm}
       \label{sec:pula}
       As an alternative to MYULA, one could also consider using the Proximal Unadjusted Langevin Algorithm (PULA) introduced in \cite{durmus2019analysis}, which replaces the (forward) gradient step of MYULA by a composition of a backward and forward step. More precisely, PULA defines the Markov chain
       $(X_k)_{k \in \nset}$ starting from  $X_0 \in \rset^{\dim}$ by the
       following recursion: for any $k \in \nset$
	\begin{equation}
	\label{eqII:part_ii_pula-generic}
	\textrm{PULA}: X_{k+1} = \prox_{U}^{\lambda} (X_k) - \gamma \nabla_{x}U(\prox_{U}^{\lambda}(X_k)) + \sqrt{2\gamma} Z_{k+1}  \eqsp .
      \end{equation}
      To highlight the connection with MYULA we note that for any $x \in \rset^{\dim}$ and $\lambda \geq 0$,
      $\nabla U^{\lambda}(x) = (x - \prox_U^{\lambda}(x))/\lambda$ by
      \cite[Proposition 12.30]{bauschke2017convex}. Therefore, if we set
      $\lambda = \gamma$ we obtain that
      \eqref{eqII:part_ii_pula-generic} can be rewritten for any
      $k \in \nset$ a
	\begin{equation}
	X_{k+1} = X_k - \gamma \nabla_x V(X_k)  - \gamma \nabla_{x}U(\prox_{U}^{\lambda}(X_k)) + \sqrt{2\gamma} Z_{k+1}  \eqsp ,
      \end{equation}
      which corresponds to \eqref{eqII:part_ii_myula-generic} with
      $\lambda = \gamma$, except that the term $\nabla_{x}U(X_k)$ in
      \eqref{eqII:part_ii_myula-generic} is replaced by
      $\nabla_{x}U(\prox_{U}^{\lambda}(X_k))$ in
      \eqref{eqII:part_ii_myula-generic}.

      Going back to the imaging problems of interest, to define the PULA families of
      Markov kernels
      $\ensembleLigne{\Sker_{\gamma, \theta}, \bSker_{\gamma,
      		\theta}}{\gamma > 0, \theta \in \Theta}$
      that we use in \Cref{algo:general} to target $\pi_{\theta}$ and
      $\bpi_{\theta}$ for $\theta \in \Theta$ we proceed as follows. We set
 $V = \f$ and
      $U = \theta^{\top} g$, $\bar{V} = 0$ and
      $\bar{U} = \theta^{\top}g$. Then, by
      \Cref{rem:part_II_ula},  for any $\theta \in \Theta$
      and $\gamma >0$, $\Sker_{\gamma, \theta}$ associated with
      $(X_k)_{k \in \nset}$ is given by $X_0 \in \rset^{\dim}$ and the
      following recursion for any $k \in \nset$
	\begin{equation}
          \label{EQII:PART_II_ PULA_explicit_posterior}
          X_{k+1} = \prox_{\theta^{\top}g}^{\lambda}(X_k) - \gamma \nabla_{x}\f(\prox_{\theta^{\top}g}^{\lambda}(X_k))  + \sqrt{2\gamma} Z_{k+1} \eqsp ,
	\end{equation}
 Similarly, for any $\theta \in \Theta$ and $\gamma' >0$, $\bSker_{\gamma, \theta}$ associated with $(X_k)_{k \in \nset}$ is given by $X_0 \in \rset^{\dim}$ and the following recursion for any $k \in \nset$ 
	\begin{equation}
          \label{EQII:PART_II_PULA_explicit_prior}
          \bX_{k+1} = \prox_{\theta^{\top}g}^{\lambda'}(\bX_k) + \sqrt{2\gamma} Z_{k+1} \eqsp .
	\end{equation}		
        Recall that $\lambda, \lambda^\prime>0$ are the
        smoothing parameters associated with $\theta^\top g^\lambda$,
        $\gamma, \gamma^\prime > 0$ are the discretisation steps and
        $(Z_k)_{k \in \N^*}$ is a sequence of i.i.d $\dim$-dimensional
        zero-mean Gaussian random variables with an identity
        covariance matrix. Again, one could use PULA with a different splitting of $F$.

	Finally, we note at this point that the MYULA and PULA kernels \eqref{EQII:PART_II_ MYULA_explicit_posterior}, \eqref{EQII:PART_II_ MYULA_explicit_prior}, \eqref{EQII:PART_II_ PULA_explicit_posterior} and   \eqref{EQII:PART_II_PULA_explicit_prior}, do not target the posterior or prior distributions exactly but rather an approximation of these distributions. This is mainly due to two facts: 1) we are not able to use the exact Langevin diffusion \eqref{EQII:PART_II_langevin}, so we resort to a discrete approximation instead; and 2) we replace the non-differentiable terms with their Moreau-Yosida envelopes. As a result of these approximation errors, \Cref{algo:general} will exhibit some asymptotic estimation bias. This error is controlled by $\lambda, \lambda', \gamma, \gamma'$, and $\delta$, and can be made arbitrarily small at the expense of additional computing time, see \Cref{thm:error_pula} in \Cref{sec:convergence-properties}.

        \section{Analysis of the convergence properties}
\label{sec:convergence-properties}
\subsection{Ergodicity properties of MYULA and PULA}
\label{sec:main-results}

Before establishing our main convergence results about
\Cref{algo:general}, see \Cref{sec:main-results}, we derive
ergodicity properties on the Markov chains given by
\eqref{eqII:part_ii_myula-generic} and
\eqref{eqII:part_ii_pula-generic}.  We consider the following
assumptions on $\pi_{\theta}$ and $\bar{\pi}_{\theta}$. These
assumptions are satisfied for a large class of models in Bayesian
imaging sciences, and in particular by the models considered in our companion paper \cite{vidal:et:al:2019a}.
\begin{assumption}
  \label{assum:potential_regularity}
  For any $\theta \in \Theta$,
  there exist $V_{\theta}, \bV_{\theta}, U_{\theta}, \bU_{\theta}: \ \rset^{\dim} \to \coint{0,+\infty}$ convex functions satisfying the following conditions.
  \begin{enumerate}[label=(\alph*), leftmargin=1cm]
  \item \label{assum:potential_regularity:item:a} For any $\theta \in \Theta$ and $x \in \rset^{\dim}$,
    \begin{equation}\pi_{\theta}(x) \propto \exp \parentheseDeux{-V_{\theta}(x) - U_{\theta}(x)} \eqsp , \qquad \bpi_{\theta}(x) \propto \exp \parentheseDeux{-\bV_{\theta}(x) - \bU_{\theta}(x)} \eqsp, \end{equation}
    and
    \begin{equation}      
      \min\parenthese{\inf_{\theta \in \Theta} \int_{\rset^{\dim}} \exp[-V_{\theta}(\tx) - U_{\theta}(\tx)] \rmd \tx, \eqsp  \inf_{\theta \in \Theta} \int_{\rset^{\dim}} \exp[-\bV_{\theta}(\tx) - \bU_{\theta}(\tx)] \rmd \tx} > 0 \eqsp.
      \label{eqII:mininum_pos}
    \end{equation}
  \item \label{assum:potential_regularity:item:b}For any $\theta \in \Theta$, $V_{\theta}$ and $\bV_{\theta}$ are continuously differentiable and there exists $\Lt \geq 0$ such that for any $\theta \in \Theta$ and $x,y \in \rset^{\dim}$
    \begin{equation}
      \max\parenthese{\norm{\nabla_x V_{\theta}(x) - \nabla_x V_{\theta}(y)}, \normLigne{\nabla_x \bV_{\theta}(x) - \nabla_x \bV_{\theta}(y)}} \leq \Lt \norm{x -y} .
    \end{equation}
    In addition, there exist $\Rvun, \Rvdeux \geq 0$ such that for any $\theta \in \Theta$, there exist $\xvstar, \bxvstar \in \rset^{\dim}$ with $\xvstar \in \argmin_{\rset^{\dim}}  V_{\theta}$, $\bxvstar \in \argmin_{\rset^{\dim}}  \bV_{\theta}$, $\xvstar, \bxvstar \in \cball{0}{\Rvun}$ and $ V_{\theta}(\xvstar), \bV_{\theta}(\bxvstar) \in \cball{0}{\Rvdeux}$.
  \item \label{assum:potential_regularity:item:c}%
There exists $\Mt \geq 0$ such that for any $\theta \in \Theta$ and $x,y \in \rset^{\dim}$
    \begin{equation}
      \max\parenthese{\normLigne{U_{\theta}(x) - U_{\theta}(y)}, \normLigne{\bU_{\theta}(x) - \bU_{\theta}(y)}} \leq \Mt \norm{x-y} \eqsp . 
    \end{equation}
    In addition, there exist $\Ruun, \Rudeux \geq 0$ such that for any $\theta \in \Theta$, there exist $\xustar, \bxustar \in \rset^{\dim}$ with $\xustar, \bxustar \in \cball{0}{\Ruun}$ and $ U_{\theta}(\xustar), \bU_{\theta}(\bxustar) \in \cball{0}{\Rudeux}$.    
  \end{enumerate}
\end{assumption}
Note that \eqref{eqII:mininum_pos} in \Cref{assum:potential_regularity}-\ref{assum:potential_regularity:item:a} is satisfied if $\Theta$ is compact and the functions $\theta \mapsto \int_{\rset^{\dim}} \exp[-V_{\theta}(\tx) - U_{\theta}(\tx)] \rmd \tx$ and $\theta \mapsto \int_{\rset^{\dim}} \exp[-\bV_{\theta}(\tx) - \bU_{\theta}(\tx)] \rmd \tx$ are continuous. This latter condition can be then easily verified using the Lebesgue dominated convergence theorem and some assumptions on $\{ V_{\theta}, \bV_{\theta}, U_{\theta}, \bU_{\theta} \,  : \, \theta \in \Theta\}$.
Note that if there 
exists $V : \rset^{\dim} \to \coint{0,+\infty}$ such that for any $\theta \in \Theta$, $V_{\theta} = V$ and there exists $x^{\star} \in \rset^{\dim}$ with  $x^{\star} \in \argmin_{\rset^{\dim}} V$ then one can choose $\xvstar = x^{\star}$ for any $\theta \in \Theta$ in \Cref{assum:potential_regularity}-\ref{assum:potential_regularity:item:b}. In this case, $\Rvdeux = 0$. Similarly if for any $\theta \in \Theta$, $U_{\theta}(0) = 0$ then one can choose $\xustar = 0$ in \Cref{assum:potential_regularity}-\ref{assum:potential_regularity:item:c} and in this case $\Ruun = \Rudeux = 0$.
These conditions are satisfied by all the models studied in \cite{vidal:et:al:2019a}.

As emphasized in \Cref{sec:sett-intr-our}, we use a stochastic
approximation proximal gradient approach to minimize $f$ and therefore
we need to consider Monte Carlo estimators for
$\nabla_{\theta} f(\theta)$ and $\theta\in \Theta$. These estimators
are derived from Markov chains targeting $\pi_{\theta}$ and
$\bpi_{\theta}$ respectively. We consider two MCMC methodologies to
construct the Markov chains.  A first option, as proposed in
\Cref{sec:myula}, is to use MYULA to sample from $\pi_{\theta}$ and
$\bpi_{\theta}$.  Let $\kappa > 0$ and
$\ensembleLigne{\Rker_{\gamma, \theta}}{\gamma >0, \theta \in \Theta}$
be the family of kernels defined for any $x \in \rset^{\dim}$,
$\gamma >0$, $\theta \in \Theta$ and $\msa \in \mcb{\rset^{\dim}}$ by
\begin{equation}
  \label{eqII:def_R_ker}
  \Rker_{\gamma,\theta}(x,\msa) = (4\uppi \gamma)^{-d/2} \int_{\msa} \exp\parenthese{\left. \norm[2]{y-x + \gamma \nabla_x  V_{\theta}(x) + \kappa^{-1}\defEns{x - \prox_{U_{\theta}}^{\gamma \kappa}(x)}} \middle/ (4\gamma) \right.} \rmd y \eqsp.
\end{equation}
Note that \eqref{eqII:def_R_ker} is the Markov kernel associated with
the recursion \eqref{eqII:part_ii_myula-generic} with
$U \leftarrow U_{\theta}$, $V \leftarrow V_{\theta}$ and
$\lambda \leftarrow \kappa \gamma$. For any $\gamma, \kappa > 0$ and
$\theta \in \Theta$ corresponds to
$\Rker_{\gamma, \kappa \gamma, \theta}$ in
\cite{vidal:et:al:2019a}. Consider also the family of Markov kernels
$\ensembleLigne{\bRker_{\gamma, \theta}}{\gamma >0, \theta \in
  \Theta}$ such that for any $\gamma >0$ and $\theta \in \Theta$,
$\bRker_{\gamma, \theta}$ is the Markov kernel defined by
\eqref{eqII:def_R_ker} but with $\bU_{\theta}$ and $\bV_{\theta}$ in
place of $U_{\theta}$ and $V_{\theta}$ respectively. The coefficient
$\kappa$ is related to $\lambda$ in \eqref{EQII:PART_II_
  MYULA_explicit_posterior} by $\kappa = \lambda / \gamma$.

Moreover, although our companion paper \cite{vidal:et:al:2019a} only considers the MYULA kernel, the theoretical results we present in this paper also hold if the algorithms are implemented using PULA \cite{durmus2019analysis}. Define the family
$\ensembleLigne{\Sker_{\gamma, \theta}}{\gamma >0,
  \theta \in \Theta}$, for any $x \in \rset^{\dim}$,
$\gamma >0$, $\theta \in \Theta$ and
$\msa \in \mcb{\rset^{\dim}}$ by
\begin{equation}
  \label{eqII:def_S_ker}
  \Sker_{\gamma,\theta}(x,\msa) = (4\uppi \gamma)^{-d/2} \int_{\msa} \exp\parenthese{\left. \norm[2]{y-\prox_{U_{\theta}}^{\gamma \kappa}(x) + \gamma \nabla_x V_{\theta}( \prox_{U_{\theta}}^{\gamma\kappa}(x) )} \middle/ (4\gamma) \right.} \rmd y \eqsp.
\end{equation}
Note that \eqref{eqII:def_R_ker} is the Markov kernel associated with
the recursion \eqref{eqII:part_ii_pula-generic} with
$U \leftarrow U_{\theta}$, $V \leftarrow V_{\theta}$ and
$\lambda \leftarrow \kappa \gamma$.  Consider also the family of
Markov kernels
$\ensembleLigne{\bSker_{\gamma, \theta}}{\gamma >0, \theta \in
  \Theta}$ such that for any $\gamma >0$ and $\theta \in \Theta$,
$\bSker_{\gamma, \theta}$ is the Markov kernel defined by the
recursion \eqref{eqII:def_S_ker} but with $\bU_{\theta}$ and
$\bV_{\theta}$ in place of $U_{\theta}$ and $V_{\theta}$ respectively.
We use the results derived in \cite{de2019efficient} to analyse the
sequence given by \eqref{eqII:algo_SOUL} with
$\{(\Kker_{\gamma,\theta}, \bKker_{\gamma,\theta}) \, : \, \gamma \in
\ocint{0,\bgamma}, \theta \in \Theta\}=\{(\Rker_{\gamma,\theta},
\bRker_{\gamma,\theta}) \, : \, \gamma \in \ocint{0,\bgamma}, \theta
\in \Theta\}$ or
$\{(\Sker_{\gamma,\theta}, \bSker_{\gamma,\theta}) \, : \, \gamma \in
\ocint{0,\bgamma}, \theta \in \Theta\}$. To this end, we impose that
for any $\gamma \in \ocint{0,\bgamma}$ and $\theta \in \Theta$, the
kernels $\Kker_{\gamma, \theta}$ and $\bKker_{\gamma, \theta}$ admit
an invariant probability distribution, denoted by
$\pi_{ \gamma, \theta}$ and $\bpi_{ \gamma, \theta}$ respectively
which are approximations of $\pi_{\theta}$ and $\bpi_{\theta}$ defined
in \Cref{assum:grad_expec}, and geometrically converge towards
them. More precisely, we show in \Cref{thm:ergo_cv_myula_text} and
\Cref{thm:ergo_cv_pula_text} below, that MYULA and PULA satisfy these
conditions if at least one of the following assumptions is verified:
\begin{assumption}
  \label{assum:potential_drift_1}
There exists $\mtt >0$ such that for any $\theta \in \Theta$, $V_{\theta}$ and $\bV_{\theta}$ are $\mtt$-convex.
\end{assumption}

\begin{assumption}
  \label{assum:potential_drift_2}
  There exist $\eta >0$ and $\ct \geq 0$ such that for any $\theta \in \Theta$ and $x \in \rset^{\dim}$, $\min(U_{\theta}(x), \bU_{\theta}(x)) \geq \eta \norm{x} - \ct$.
\end{assumption}
Note that if for any $\theta \in \Theta$, $U_{\theta}$ is convex on
$\rset^{\dim}$ and
$\sup_{\theta \in \Theta} (\int_{\rset^{\dim}} \exp [-U_{\theta}(\tx)]
\rmd \tx) <+\infty$, then \Cref{assum:potential_drift_2} is
automatically satisfied, as an immediate extension of \cite[Lemma 2.2
(b)]{bakry2008simple}. In \cite{vidal:et:al:2019a},
\Cref{assum:potential_drift_2} is satisfied as soon as the prior
distribution $x \mapsto p(x |\theta)$ is log-concave and proper for
any $\theta \in \Theta$.  In \cite{vidal:et:al:2019a}, if the prior
$x \mapsto p(x|\theta)$ is improper for some $\theta \in \Theta$ then
we require \Cref{assum:potential_drift_1} to be satisfied, \ie \ for
any $\y \in \cset^{\dimy}$, there exists $\mtt >0$ such that for any
$\theta \in \Theta$, $x \mapsto p(x|y, \theta)$ is $\mtt$-log-concave.
Finally, we believe that \Cref{assum:potential_drift_2} could be
relaxed to the following condition: there exist $\eta >0$ and
$\ct \geq 0$ such that for any $\theta \in \Theta$ and
$x \in \rset^{\dim}$,
$\min(U_{\theta}(x) + V_{\theta}(x), \bU_{\theta}(x) +
\bV_{\theta}(x)) \geq \eta \norm{x} - \ct$.  In particular, this
latter condition holds in the case where
$x \mapsto p(x|\theta) = \exp[-\theta^{\top}\TV(x)]$ and
$\sup_{\theta \in \Theta} (\int_{\rset^{\dim}} \exp [-U_{\theta}(\tx)
+ V_{\theta}(\tx)] \rmd \tx) <+\infty$.

Consider for any $\pow \in \nsets$ and $\alpha >0$, the two functions
$W_{\pow}$ and $W_{\alpha}$ given for any $x \in \rset^d$ by
\begin{equation}
  \label{eqII:def_W_pow_alph}
  W_{\pow}(x) = 1 + \norm{x}^{2\pow} \eqsp, \qquad W_{\alpha} = \exp\parentheseDeux{\alpha\sqrt{1 + \norm{x}^2}} \eqsp.
\end{equation}

\begin{theorem}
  \label{thm:ergo_cv_myula_text}
  Assume \tup{\Cref{assum:potential_regularity}} and
  \tup{\Cref{assum:potential_drift_1}} or
  \tup{\Cref{assum:potential_drift_2}}.  Let
  $\bkappa > 1 \geq \ukappa > 1/2$,
  $\bgamma < \min\defEnsLigne{(2 - 1/\kappa)/\Lt, 2 / (\mtt + \Lt)}$
  if \tup{\Cref{assum:potential_drift_1}} holds and
  $\bgamma < \min \defEnsLigne{(2 - 1 / \kappa) / \Lt, \eta / (2 \Mt
    \Lt)}$ if \tup{\Cref{assum:potential_drift_2}} holds.  Then for
  any $a \in \ocint{0,1}$, there exist $\bA_{2, a} \geq 0$ and
  $\rho_{a} \in \ooint{0,1}$ such that for any $\theta \in \Theta$,
  $\kappa \in \ccint{\ukappa, \bkappa}$,
  $\gamma \in \ocint{0,
    \bgamma}$, 
  $\Rker_{\gamma, \theta}$ and $\bRker_{\gamma, \theta}$ admit
  invariant probability measures $\pi_{\gamma, \theta}$, respectively
  $\bpi_{\gamma, \theta}$. In addition, for any
  $x, y \in \rset^{\dim}$ and $n \in \nset$ we have
\begin{equation}
  \begin{aligned}
        \max \parenthese{ \Vnorm[W^a]{\updelta_x \Rker_{\gamma, \theta}^n - \pi_{\gamma, \theta}}, \Vnorm[W^a]{\updelta_x \bRker_{\gamma, \theta}^n - \bpi_{\gamma, \theta}} } &\leq \bA_{2,  a} \brho_{ a}^{\gamma n} W^a(x) \eqsp ,  \\ 
    \max \parenthese{ \Vnorm[W^a]{\updelta_x \Rker_{\gamma, \theta}^n - \updelta_y \Rker_{\gamma, \theta}^n}, \Vnorm[W^a]{\updelta_x \bRker_{\gamma, \theta}^n - \updelta_y \bRker_{\gamma, \theta}^n}}  &\leq \bA_{2,  a} \brho_{ a}^{\gamma n}  \defEns{W^a(x) + W^a(y)} \eqsp , 
    \end{aligned}
    \end{equation}
       with $W  =W_m$ and  $\pow \in \nsets$ if \tup{\Cref{assum:potential_drift_1}} holds and $W = W_{\alpha}$ with $\alpha < \min(\ukappa \eta /4, \eta /8)$ if \tup{\Cref{assum:potential_drift_2}} holds.
\end{theorem}

\begin{proof}
  The proof is postponed to \Cref{sec:myula_cv}.
\end{proof}

\begin{theorem}
  \label{thm:ergo_cv_pula_text}
  Assume \tup{\Cref{assum:potential_regularity}} and
  \tup{\Cref{assum:potential_drift_1}} or
  \tup{\Cref{assum:potential_drift_2}}.  Let Let
  $\bkappa > 1 \geq \ukappa > 1/2$, $\bgamma < 2 / (\mtt + \Lt)$ if
  \tup{\Cref{assum:potential_drift_1}} holds and $\bgamma < 2 / \Lt$
  if \tup{\Cref{assum:potential_drift_2}} holds.  Then for any
  $a \in \ocint{0,1}$, there exist $A_{2, a} \geq 0$ and
  $\rho_{a} \in \ooint{0,1}$ such that for any $\theta \in \Theta$,
  $\kappa \in \ccint{\ukappa, \bkappa}$,
  $\gamma \in \ocint{0,
    \bgamma}$, 
  $\Sker_{\gamma, \theta}$ and $\bSker_{\gamma, \theta}$ admit an
  invariant probability measure $\pi_{\gamma, \theta}$ and
  $\bpi_{\gamma, \theta}$ respectively. In addition, for any
  $x, y \in \rset^{\dim}$ and $n \in \nset$ we have
\begin{equation}
  \begin{aligned}
        \max \parenthese{ \Vnorm[W^a]{\updelta_x \Sker_{\gamma, \theta}^n - \pi_{\gamma, \theta}}, \Vnorm[W^a]{\updelta_x \bSker_{\gamma, \theta}^n - \bpi_{\gamma, \theta}} } &\leq A_{2,  a} \rho_{ a}^{\gamma n} W^a(x) \eqsp ,  \\ 
    \max \parenthese{ \Vnorm[W^a]{\updelta_x \Sker_{\gamma, \theta}^n - \updelta_y \Sker_{\gamma, \theta}^n}, \Vnorm[W^a]{\updelta_x \bSker_{\gamma, \theta}^n - \updelta_y \bSker_{\gamma, \theta}^n}}  &\leq A_{2,  a} \rho_{ a}^{\gamma n}  \defEns{W^a(x) + W^a(y)} \eqsp , 
    \end{aligned}
    \end{equation}
       with $W  =W_m$ and  $\pow \in \nsets$ if \tup{\Cref{assum:potential_drift_1}} holds and $W = W_{\alpha}$ with $\alpha < \ukappa \eta /4$ if \tup{\Cref{assum:potential_drift_2}} holds.
\end{theorem}

\begin{proof}
  The proof is postponed to \Cref{sec:pula_cv}.
\end{proof}

\subsection{Main results}
 \label{sec:main_results}

We now state our main results regarding the convergence of the sequence defined by \eqref{eqII:algo_SOUL} under the following additional regularity
assumption.
\begin{assumption}
  \label{assum:theta_reg}
  There exist $\Munu \geq 0$ and $\ftt_{\Theta} \in \rmc(\rset_+,\rset_+)$ such that for any $\theta_1, \theta_2 \in \Theta$, $x \in \rset^{\dim}$, 
  \begin{equation}
   \begin{aligned}
   &
   \max\parenthese{\normLigne{\nabla_x V_{\theta_1}(x) - \nabla_x V_{\theta_2}(x)}, \normLigne{\nabla_x \bV_{\theta_1}(x) - \nabla_x \bV_{\theta_2}(x)}} \leq \Munu \norm{\theta_1 - \theta_2} (1 + \norm{x}) \eqsp , \\
   &\max\parenthese{\normLigne{\nabla_x U_{\theta_1}^{\upkappa}(x) - \nabla_x U_{\theta_2}^{\upkappa}(x)}, \normLigne{\nabla_x \bU_{\theta_1}^{\upkappa}(x) - \nabla_x \bU_{\theta_2}^{\upkappa}(x)}} \leq \ftt_{\Theta}(\upkappa) \norm{\theta_1 - \theta_2} (1 + \norm{x}) \eqsp .
   \end{aligned}
  \end{equation}
\end{assumption}
In \Cref{thm:cv_pula}, we give sufficient conditions on the parameters of the algorithm under which the sequence $(\theta_n)_{n \in \nset}$ converges \as, and we give explicit convergence rates in \Cref{thm:error_pula}.
\begin{theorem}
  \label{thm:cv_pula}
  Assume \tup{\Cref{assum:theta_compact}},
  \tup{\Cref{assum:f_grad_lip}}, \tup{\Cref{assum:grad_expec}} and
  that $f$ is convex. Let $\kappa \in \ccint{\ukappa, \bkappa}$ with
  $\bkappa \geq 1 \geq \ukappa > 1/2$. Assume
  \tup{\Cref{assum:potential_regularity}} and one of the following
  conditions:
\begin{enumerate}[label=(\alph*), leftmargin=1cm]
\item   \label{ass:cv_pula_a} \tup{\Cref{assum:potential_drift_1}} holds, $\bgamma < \min(2/(\mtt + \Lt), (2-1/\ukappa)/\Lt, \Lt^{-1})$ and there exists $\pow \in \nsets$ and $C_{\pow} \geq 0$ such that for any $\theta \in \Theta$ and $x \in \rset^{\dim}$,  $ \normLigne{H_{\theta}(x)} \leq  C_{\pow} W^{1/4}_{\pow}(x)$ and $ \normLigne{\bH_{\theta}(x)} \leq  C_{\pow} W^{1/4}_{\pow}(x)$.
\item \label{ass:cv_pula_b} \tup{\Cref{assum:potential_drift_2}} holds, $\bgamma < \min((2 - 1/\ukappa)/\Lt, \eta / (2\Mt\Lt), \Lt^{-1})$ and there exists $0 < \alpha < \eta/4$, $C_{\alpha} \geq 0$ such that for any $\theta \in \Theta$ and $x \in \rset^{\dim}$,  $ \normLigne{H_{\theta}(x)} \leq  C_{\alpha} W^{1/4}_{\alpha}(x)$ and $ \normLigne{\bH_{\theta}(x)} \leq  C_{\alpha} W^{1/4}_{\alpha}(x)$.
\end{enumerate}
Let $(\gamma_n)_{n \in \nset}$,
$(\delta_n)_{n \in \nset}$ be sequences of non-increasing positive real numbers
and $(m_n)_{n \in \nset}$ be a sequence of non-decreasing positive integers satisfying
$\delta_0 < 1 / \L$ and
$\gamma_0 < \bgamma$. Let
$(\lbrace (X_{k}^n, \bX_{k}^n): k \in \lbrace 0, \dots, m_n \rbrace \rbrace,
\theta_n)_{n \in \N}$ be given by \eqref{eqII:algo_SOUL}.  In addition, assume that
$\sum_{n=0}^{+\infty} \delta_{n+1} = +\infty$,
$\sum_{n=0}^{+\infty} \delta_{n+1} \gamma_n^{1/2} < +\infty$ and that one of the
following conditions holds:
  \begin{enumerate}[leftmargin=1cm, label=(\arabic*)]
  \item \label{item:item_1}   $\sum_{n=0}^{+\infty} \delta_{n+1} / (m_n \gamma_n) < +\infty \eqsp ;$
  \item \label{item:item_2} $m_n = m_0 \in \nsets$ for all $n \in \nset$, $\sup_{n \in \nset} \abs{\delta_{n+1} - \delta_n} \delta_n^{-2} < +\infty$, \tup{\Cref{assum:theta_reg}} holds and we have $\sum_{n=0}^{+\infty} \delta_{n+1}^2\gamma_n^{-2}  < +\infty$,  
    $\sum_{n=0}^{+\infty} \delta_{n+1} \gamma_{n+1}^{-3} (\gamma_n - \gamma_{n+1}) < +\infty \eqsp .$
     \end{enumerate}
Then $(\theta_n)_{n \in \N}$ 
    converges \as~to some $\thetaStar \in \argmin_{\Theta} f$. Furthermore, \as~there exists $C\geq0$ such that for any $n \in \nsets$
    \begin{equation}      
    \defEns{\left. \sum_{k=1}^n \delta_k f(\theta_k) \middle/ \sum_{k=1}^n \delta_k \right. } - \min_{\Theta} f \leq \left. C \middle/\left( \sum_{k=1}^n \delta_k \right) \right.  \eqsp.
  \end{equation}  
\end{theorem}

\begin{proof}
  The proof is postponed to \Cref{sec:proof_thm_cv}.
\end{proof}
These results are similar to the
ones identified in \cite[Theorem 1, Theorem 5, Theorem
6]{de2019efficient} for the Stochastic
Optimization with Unadjusted Langevin (SOUL) algorithm. Note that in
SOUL the potential is assumed to be differentiable and the sampler is
given by ULA, whereas in \Cref{thm:cv_pula}, the results are stated
for PULA and MYULA samplers. 

Although rigorously establishing convexity of $f$ is usually not
possible for imaging models, we expect that in many cases, for any of
its minimizer $\theta_{\star}$, $f$ is convex in some neighborhood
of $\theta_{\star}$. For example, this is the case if its Hessian
is definite positive around this point.

Assume that $\delta_n \sim n^{-a}$, $\gamma_n \sim n^{-b}$ and $m_n \sim n^{-c}$ with $a,b, c \geq 0$.
We now distinguish two cases depending on if for all $n \in \nset$, $m_n = m_0 \in \nsets$ (fixed batch size) or not (increasing size).
\begin{enumerate}[wide, labelwidth=!, labelindent=0pt, label=\arabic*)]
\item In the increasing batch size case, \Cref{thm:cv_pula} ensures that $(\theta_n)_{n \in \nset}$ converges if the following inequalities are satisfied
  \begin{equation}
    \label{eqII:condition_sum_expo_1}
    a + b/2 > 1 \eqsp , \qquad a - b + c > 1 \eqsp , \qquad a \leq 1 \eqsp .
  \end{equation}
Note in particular that $c > 0$, \ie \ the number of Markov chain iterates required to compute the estimator of the gradient increases at each step. However, for any $a \in \ccint{0,1}$ there exist $b,c > 0$ such that \eqref{eqII:condition_sum_expo_1} is satisfied. In the special setting where $a = 0$ then for any $\vareps_2 > \vareps_1 > 0$ such that $b = 2 + \vareps_1$ and $c = 3 + \vareps_2$ satisfy the results of \eqref{eqII:condition_sum_expo_1} hold.
\item In the fixed batch size case, which implies that $c=0$, \Cref{thm:cv_pula} ensures that $(\theta_n)_{n \in \nset}$ converges if the following inequalities are satisfied
  \begin{equation}    
    a + b/2 > 1 \eqsp , \qquad 2(a - b) > 1 \eqsp , \qquad a + b + 1 - 2 b > 1 \qquad a \leq 1 \eqsp ,
  \end{equation}
  which can be rewritten as
  \begin{equation}
    \label{eqII:condition_sum_expo_2}
    b \in \ooint{2(1-a), \min(a - 1/2, a/2)} \eqsp , \qquad a \in \ccint{0,1} \eqsp .
  \end{equation}
  The interval $\ooint{2(a-1), \min(a - 1/2, a/2)}$ is then not empty if and only if $a \in \ocint{5/6,1}$.
\end{enumerate}

\begin{theorem}
  \label{thm:error_pula}
  Assume \tup{\Cref{assum:theta_compact}},
  \tup{\Cref{assum:f_grad_lip}}, \tup{\Cref{assum:grad_expec}} and
  that $f$ is convex. Let $\kappa \in \ccint{\ukappa, \bkappa}$ with
  $\bkappa \geq 1 \geq \ukappa > 1/2$. Assume
  \tup{\Cref{assum:potential_regularity}} and that the condition
  \ref{ass:cv_pula_a} or \ref{ass:cv_pula_b} in \Cref{thm:cv_pula} is
  satisfied. Let $(\gamma_n)_{n \in \nset}$,
  $(\delta_n)_{n \in \nset}$ be sequences of non-increasing positive
  real numbers and $(m_n)_{n \in \nset}$ be a sequence of
  non-decreasing positive integers satisfying $\delta_0 < 1 / \L$ and
  $\gamma_0 < \bgamma$. Let
  $(\lbrace (X_{k}^n, \bX_{k}^n): k \in \lbrace 0, \dots, m_n \rbrace
  \rbrace, \theta_n)_{n \in \N}$ be given by \eqref{eqII:algo_SOUL}
    \begin{equation}
    \expe{  \defEns{\left. \sum_{k=1}^n \delta_k f(\theta_k) \middle/ \sum_{k=1}^n \delta_k \right. } - \min_{\Theta} f  }\leq  \left. E_n \middle/  \left( \sum_{k=1}^n \delta_k \right) \right. \eqsp ,
  \end{equation}
where
  \begin{enumerate}[label=(\alph*), leftmargin=1cm]
  \item 
    \begin{equation}
     \label{eqII:ineq_increased} 
      E_n = C_1 \defEns{1 + \sum_{k=0}^{n-1} \delta_{k+1} \gamma_k^{1/2} + \sum_{k=0}^{n-1} \delta_{k+1} / (m_k \gamma_k) + \sum_{k=0}^{n-1} \delta_{k+1}^2 / (m_k \gamma_k)^2} \eqsp . \end{equation}
\item or if $m_n = m_0$ for all $n \in \nset$, $\sup_{n \in \nset} \abs{\delta_{n+1} - \delta_n} \delta_n^{-2} < +\infty$ and \tup{\Cref{assum:theta_reg}} holds
  \begin{equation}
        \label{eqII:ineq_fixed}
      E_n = C_2 \left\lbrace1 + \sum_{k=0}^{n-1} \delta_{k+1} \gamma_k^{1/2}  + \sum_{k=0}^{n-1} \delta_{k+1}^2/\gamma_{k} 
        + \sum_{k=0}^{n-1} \delta_{k+1} \gamma_{k+1}^{-3} (\gamma_k - \gamma_{k+1})\right\rbrace \eqsp .
    \end{equation}
  \end{enumerate}
\end{theorem}

\begin{proof}
  The proof is postponed to \Cref{thm:error_pula_proof}.
\end{proof}

First, note that if the stepsize is fixed and recalling that $\kappa = \lambda / \gamma$ then the condition $\gamma < (2 - 1/\kappa) / \Lt$ can be rewritten as $\gamma < 2 / (\Lt + \lambda^{-1})$. Assume that $ (\delta_n)_{ n \in \N}$ is non-increasing, $\lim_{n \to +\infty} \delta_n = 0$, $\lim_{n \to +\infty} m_n = +\infty$ and $\gamma_n = \gamma_0 > 0$ for all $n \in \nset$. In addition, assume that $\sum_{n \in \nsets} \delta_n = +\infty$ then, by  \cite[Problem 80, Part I]{polya1998problem}, it holds that
\begin{equation}
  \label{eqII:polya_stuff}
  \begin{cases}
    &\lim_{n \to +\infty} \parentheseDeux{\left . \left(\sum_{k=1}^n \delta_{k} / m_k\right) \middle/ \left(\sum_{k=1}^n \delta_k\right) \right.} = \lim_{n \to +\infty} 1/m_n = 0  \eqsp ; \\ &
    \lim_{n \to +\infty} \parentheseDeux{\left. \left(\sum_{k=1}^n \delta_{k}^2\right) \middle / \left(\sum_{k=1}^n \delta_k\right) \right.}= \lim_{n \to +\infty} \delta_n = 0 \eqsp  .
    \end{cases}
\end{equation}
Therefore, using \eqref{eqII:ineq_increased} we obtain that
 \begin{equation}
 \limsup_{n \to \plusinfty} \expe{ \defEns{\left . \sum_{k=1}^n \delta_k f(\theta_k) \middle/ \sum_{k=1}^n \delta_k \right.} - \min f} \leq  C_1 \sqrt{\gamma_0}  \eqsp .
\end{equation}
Similarly, if the stepsize is fixed and the number of Markov chain iterates is fixed, \ie \ for all $n \in \nset$, $\gamma_n=\gamma_0$ and $m_n = m_0$ with $\gamma_0 >0$ and $m_0 \in \nsets$, combining \eqref{eqII:ineq_fixed} and \eqref{eqII:polya_stuff} we obtain that 
\begin{equation}
  \label{eqII:borne_fix}
  \limsup_{n \to \plusinfty} \expe{ \defEns{\left . \sum_{k=1}^n \delta_k f(\theta_k) \middle/ \sum_{k=1}^n \delta_k \right.} - \min f} \leq  C_2\sqrt{\gamma_0}  \eqsp .
\end{equation}

\section{Proof of the main results}
\label{thm:cv_pula_proof}
In this section, we gather the proofs of
\Cref{sec:convergence-properties}.  First, in
\Cref{sec:technical-lemmas} we derive some useful technical lemmas.
In \Cref{sec:myula_cv}, we prove \Cref{thm:ergo_cv_myula_text},
using minorisation and Foster-Lyapunov drift conditions.  Similarly,
we prove \Cref{thm:ergo_cv_pula_text} in \Cref{sec:pula_cv}.  Next,
we show \Cref{thm:cv_pula} by applying \cite[Theorem 1, Theorem
3]{de2019efficient} and \Cref{thm:error_pula} by applying
\cite[Theorem 2, Theorem 4]{de2019efficient}, which boils down to
verifying that \cite[H1, H2]{de2019efficient} are satisfied. In
\Cref{sec:check-citeh1d-pula}, we show that \cite[H1,
H2]{de2019efficient} hold if the sequence is given by
\eqref{eqII:algo_SOUL} where
$\{(\Kker_{\gamma,\theta}, \bKker_{\gamma,\theta}) \, : \, \gamma \in
\ocint{0,\bgamma}, \theta \in \Theta\} = \{(\Rker_{\gamma,\theta},
\bRker_{\gamma,\theta}) \, : \, \gamma \in \ocint{0,\bgamma}, \theta
\in \Theta\}$ defined in \eqref{eqII:def_S_ker}, \ie~we consider PULA
as a sampling scheme in the optimization algorithm. In
\Cref{sec:check-citeh1d-myula} we check that \cite[H1,
H2]{de2019efficient} are satisfied when
$\{(\Kker_{\gamma,\theta}, \bKker_{\gamma,\theta}) \, : \, \gamma \in
\ocint{0,\bgamma}, \theta \in \Theta\} = \{(\Sker_{\gamma,\theta},
\bSker_{\gamma,\theta}) \, : \, \gamma \in \ocint{0,\bgamma}, \theta
\in \Theta\}$ defined in \eqref{eqII:def_R_ker}, \ie~when considering
MYULA as a sampling scheme. Finally, we prove \Cref{thm:cv_pula} in
\Cref{sec:proof_thm_cv} and \Cref{thm:error_pula} in
\Cref{thm:error_pula_proof}.

\subsection{Technical lemmas}
\label{sec:technical-lemmas}
We say that a Markov kernel $\Rker$ on $\rset^{\dim}\times \mcb{\rset^{\dim}}$ satisfies a discrete Foster-Lyapunov drift condition
$\bfDd(W,\lambda,b)$ if there exist $\lambda \in (0,1)$, $b\geq0$ and a measurable function $W: \rset^{\dim} \to \coint{1,+\infty}$ such that for all $x \in \rset^{\dim}$
\begin{equation}
  \label{eqII:discrete_drift}
  \Rker W(x) \leq \lambda W(x) + b \eqsp.
\end{equation}
We will use the following result.
\begin{lemma}
  \label{lemma:majo_step}
  Let $\Rker$ be a Markov kernel on  $\rset^{\dim}\times \mcb{\rset^{\dim}}$ which satisfies $\bfDd(W,\lambda^{\gamma},b\gamma)$ with $\lambda \in (0,1)$, $b\geq0$, $\gamma >0$ and a measurable function $W: \rset^{\dim} \to \coint{1,+\infty}$.
  Then, we have for any $x \in \rset^{\dim}$
  \begin{equation}
    \Rker^{\step}W(x) \leq (1 + b \log^{-1}(1/\lambda) \lambda^{-\bgamma}) W(x) \eqsp .
  \end{equation}
\end{lemma}

\begin{proof}
Using \cite[Lemma 9]{de2019efficient} we have for any $x \in \rset^{\dim}$
\begin{equation}
  \Rker^{\step}W(x) \leq \parenthese{\lambda^{\gamma \step} + b \gamma \sum_{k=0}^{\step - 1} \lambda^{\gamma k} } W(x) \leq (1 + b \log^{-1}(1/\lambda) \lambda^{-\bgamma}) W(x) \eqsp .
\end{equation}

\end{proof}

We continue this section by giving some results on proximal
operators. Some of them are well-known but their proof is given for
completeness.

\begin{lemma}
  \label{lemma:borne_prox}
  Let $\upkappa >0$ and $U : \ \rset^{\dim} \to \rset$ convex. Assume that $U$ is $M$-Lipschitz with $M \geq 0$, then $U^{\upkappa}$ is $M$-Lipschitz and for any $x \in \rset^{\dim}$, $\norm{x - \prox_{U}^{\upkappa}(x)} \leq  \upkappa M$. 
\end{lemma}

\begin{proof}
  Let $\upkappa >0$. We have for any $x,y \in \rset^{\dim}$ by  \eqref{eqII:part_ii_MY-envelope} and \eqref{eqII:part_ii_def_prox} 
  \begin{align}
    &U^{\upkappa}(x) - U^{\upkappa}(y) \\ &= \norm{x - \prox_U^{\upkappa}(x)}^2/(2\upkappa) + U(\prox_U^{\upkappa}(x)) -\norm{y - \prox_U^{\upkappa}(y)}^2/(2\upkappa) - U(\prox_U^{\upkappa}(y)) \\
    &\leq \normLigne{y - \prox_U^{\upkappa}(y)}^2/(2\upkappa) + U(x-y+\prox_U^{\upkappa}(y)) - \normLigne{y - \prox_U^{\upkappa}(y)}^2/(2\upkappa)- U(\prox_U^{\upkappa}(y)) \\
\nonumber
    &\leq M \norm{x -y} \eqsp.
  \end{align}
  Hence, $U^{\upkappa}$ is $M$-Lipschitz. Since by \cite[Proposition 12.30]{bauschke2017convex}, 
  $U^{\upkappa}$ is continuously differentiable we have for any $x \in \rset^{\dim}$,
  $\norm{\nabla U^{\upkappa}(x)} \leq M$. Combining this result with the fact that for any $x \in \rset^{\dim}$, $ \nabla U^{\upkappa}(x) = (x - \prox_{U}^{\upkappa}(x))/ \upkappa$ by  \cite[Proposition 12.30]{bauschke2017convex} concludes the proof.
\end{proof}

\begin{lemma}
  \label{lemma:majo_norm_moins}
Let $U: \rset^{\dim} \to \coint{0,+\infty}$ be a convex and $M$-Lipschitz function with $M \geq 0$ 
. Then for any $\upkappa >0$ and $z, z' \in \rset^{\dim}$,
  \begin{equation}
    \label{eqII:majo_norm_moins}
    \langle \prox_{U}^{\upkappa}(z) - z, z\rangle  \leq -\upkappa U(z) + \upkappa^2 M^2 + \upkappa \defEns{U(z') + M \norm{z'}}  \eqsp .
  \end{equation}
\end{lemma}

\begin{proof}
  $\upkappa >0$ and $z, z' \in \rset^{\dim}$.
    Since $(z - \prox_{U}^{\upkappa}(z))/\upkappa \in \partial U(\prox_{U}^{\upkappa}(z))$ \cite[Proposition 16.44]{bauschke2017convex}, we have
    \begin{align}
     \upkappa \defEns{U(z') - U(\prox_U^{\upkappa}(z))} &\geq \langle z - \prox_{U}^{\upkappa}(z), z' - \prox_{U}^{\upkappa}(z)\rangle \\ &\geq \langle z - \prox_{U}^{\upkappa}(z), z' - z\rangle  + \norm{z - \prox_{U}^{\upkappa}(z)}^2 \\  &\geq  \langle z - \prox_{U}^{\upkappa}(z), z' - z\rangle \eqsp .
    \end{align}
     Combining this result, the fact that $U$ is $M$-Lipschitz and \Cref{lemma:borne_prox} we get that
    \begin{align}
      \langle \prox_{U}^{\upkappa}(z) - z, z\rangle &\leq \upkappa U(z') - \upkappa U(z) + \upkappa M \norm{z - \prox_U^{\upkappa}(z)} + \norm{z'} \norm{z - \prox_U^{\upkappa}(z)}  \\
      &\leq -\upkappa U(z) + \upkappa^2 M^2 + \upkappa \defEns{U(z') + M \norm{z'}}  \eqsp ,
    \end{align}
    which concludes the proof
\end{proof}

\begin{lemma}
  \label{lemma:control_prox_gamma}
  Let $\upkappa_1, \upkappa_2 >0$ and $U: \ \rset^{\dim} \to \rset$ convex and lower semi-continuous. For any $x \in \rset^{\dim}$ we have
  \begin{equation}
    \norm{\prox_U^{\upkappa_1}(x) - \prox_U^{\upkappa_2}(x)}^2 \leq 2 (\upkappa_1 - \upkappa_2)(U(\prox_U^{\upkappa_2}(x)) - U(\prox_U^{\upkappa_1}(x))) \eqsp .
  \end{equation}
  If in addition, $U$ is $\tM$-Lipschitz with $M \geq 0$ then
  \begin{equation}
    \norm{\prox_U^{\upkappa_1}(x) - \prox_U^{\upkappa_2}(x)} \leq 2 M \abs{\upkappa_1 - \upkappa_2} \eqsp .
  \end{equation}  
\end{lemma}

\begin{proof}
  Let $x \in \rset^d$. By definition of $\prox_U^{\upkappa_1}(x)$ we
  have
  \begin{align}
    \label{eqII:ineq_1_prox}
    2 \upkappa_1 U(\prox_U^{\upkappa_1}(x)) + \norm{x - \prox_U^{\upkappa_1}(x)}^2 \leq 2 \upkappa_1 U(\prox_U^{\upkappa_2}(x)) + \norm{x - \prox_U^{\upkappa_2}(x)}^2 \eqsp .
  \end{align}
  Combining this result and the fact that $(x - \prox_U^{\upkappa_2}(x))/\upkappa_2 \in \partial U(\prox_U^{\upkappa_2}(x))$  we have
  \begin{align}
    &\norm{\prox_U^{\upkappa_1}(x) - \prox_U^{\upkappa_1}(x)}^2 \\
    &\qquad \leq 2\upkappa_1\defEns{U(\prox_U^{\upkappa_2}(x)) - U(\prox_U^{\upkappa_1}(x))} + 2 \langle x - \prox_U^{\upkappa_2}(x), \prox_U^{\upkappa_1}(x) - \prox_U^{\upkappa_2}(x) \rangle \\
                                                               &\qquad \leq 2\upkappa_1\defEns{U(\prox_U^{\upkappa_2}(x)) - U(\prox_U^{\upkappa_1}(x))} + 2\upkappa_2\defEns{U(\prox_U^{\upkappa_1}(x)) - U(\prox_U^{\upkappa_2}(x))} \\
    &\qquad \leq 2 (\upkappa_1 - \upkappa_2) (U(\prox_U^{\upkappa_2}(x)) - U(\prox_U^{\upkappa_1}(x))) \eqsp ,
  \end{align}
  which concludes the proof.
\end{proof}

\begin{lemma}
  \label{lemma:contrac_diff}
  Let $V : \ \rset^{\dim} \to \rset$ $\mtt$-convex and continuously
  differentiable with $\mtt \geq 0$. Assume that there exists $M > 0$
  such that for any $x,y \in \rset^{\dim}$
  \begin{equation}
    \label{eqII:grad_u_lip}
  \norm{\nabla V(x) - \nabla V(y)} \leq M \norm{x - y} \eqsp .
\end{equation}
Assume that there exists $x^{\star} \in \argmin_{\rset^{\dim}} V$,
then for any $\gamma \in \ocint{0, \bgamma}$ with $\bgamma < 2 / (M+ \mtt)$ and $x \in \rset^{\dim}$
\begin{equation}
  \norm{x - \gamma \nabla V(x)}^2 \leq (1 - \gamma \varpi) \norm{x}^2 + \gamma \defEnsLigne{(2/(\mtt+M) - \bgamma)^{-1} + 4 \varpi} \norm{x^{\star}}^2 \eqsp ,
\end{equation}
with $\varpi = \mtt M / (\mtt + M)$.
\end{lemma}

\begin{proof}
  Let $x \in \rset^{\dim}$, $\gamma \in \ocint{0, \bgamma}$ and $\bgamma < 2 / (\mtt + M)$. Using \cite[Theorem 2.1.11]{nesterov2013introductory} and the fact that for any $a,b, \vareps >0$, $\vareps a^2 + b^2 / \vareps \geq 2 ab$ we have
    \begin{align}
      &\norm{x - \gamma \nabla V(x)}^2  \\ & \quad \leq \norm{x}^2 -2 \gamma \langle \nabla V(x) - \nabla V(x^{\star}), x - x^{\star} \rangle + \gamma \bgamma  \norm{\nabla V(x) - \nabla V(x^{\star})}^2 \\
                                       & \qquad + 2 \gamma \norm{x^{\star}}  \norm{\nabla V(x) - \nabla V(x^{\star})} \\
&\quad \leq \norm{x}^2 -2 \gamma \varpi\norm{x - x^{\star}}^2  - \gamma \parentheseLigne{2/(\mtt + M)-\bgamma}  \norm{\nabla V(x) - \nabla V(x^{\star})}^2 \\
                                       & \qquad + 2 \gamma \norm{x^{\star}}  \norm{\nabla V(x) - \nabla V(x^{\star})} \\
&\quad \leq \norm{x}^2 -2 \gamma \varpi\norm{x - x^{\star}}^2  - \gamma \parentheseLigne{2/(\mtt + M)-\bgamma}  \norm{\nabla V(x) - \nabla V(x^{\star})}^2 \\
                                       & \qquad + \gamma \parentheseLigne{2/(\mtt + M)-\bgamma}  \norm{\nabla V(x) - \nabla V(x^{\star})}^2 + \gamma / \parentheseLigne{2/(\mtt + M)-\bgamma} \norm{x^{\star}}^2  \\
                                       &\quad \leq (1 - 2 \gamma \varpi) \norm{x}^2 + 4\gamma \varpi \norm{x^{\star}} \norm{x} + \gamma / \parentheseLigne{2/(\mtt + M)-\bgamma} \norm{x^{\star}}^2 \\
      &\quad \leq (1 - \gamma \varpi) \norm{x}^2 + \gamma \defEns{(2/(\mtt+M) - \bgamma)^{-1} + 4 \varpi} \norm{x^{\star}}^2 \eqsp .
    \end{align}
    ~
  \end{proof}

  \begin{lemma}
    \label{lemma:drift_det_pula_1}
    Assume \tup{\Cref{assum:potential_regularity}} and \tup{\Cref{assum:potential_drift_1}}. Then for any $\kappa >0$, $\theta \in \Theta$, $\gamma \in \ocint{0, \bgamma}$ with $\bgamma < 2 / (\mtt + \Lt)$ and $x \in \rset^{\dim}$, we have
    \begin{multline}
      \norm{\prox_{U_{\theta}}^{\gamma\kappa}(x)  -\gamma \nabla_x V_{\theta}(\prox_{U_{\theta}}^{\gamma\kappa}(x))}^2 \\ \leq
      (1- \gamma \varpi/2) \norm{x}^2 + \gamma \left[ \bgamma \kappa^2 \Mt^2 + \defEns{(2/(\mtt + \Lt) - \bgamma)^{-1} + 4 \varpi} \Rvun^2 \right. \left.+ 2  \kappa^2 \Mt^2 \varpi^{-1} \right] \eqsp , \end{multline}
  \end{lemma}
with $\varpi = \mtt \Lt / (\mtt + \Lt)$.
  \begin{proof}
    Let $\kappa >0$, $\theta \in \Theta$, $\gamma \in \ocint{0, \bgamma}$ and $x \in \rset^{\dim}$. Using \Cref{assum:potential_regularity}, \Cref{assum:potential_drift_1}, \Cref{lemma:borne_prox},  \Cref{lemma:contrac_diff}, the Cauchy-Schwarz inequality and that for any $\alpha, \beta \geq 0$, $\max_{t \in \rset}(-\alpha t^2 + 2\beta t) = \beta^2/\alpha$, 
  we have
    \begin{align}
       &\norm{\prox_{U_{\theta}}^{\gamma\kappa}(x)  -\gamma \nabla_x V_{\theta}(\prox_{U_{\theta}}^{\gamma\kappa}(x))}^2  \\
    & \quad \leq (1- \gamma \varpi) \norm{\prox_{U_{\theta}}^{\gamma\kappa}(x)}^2 + \gamma \defEns{(2/(\mtt + \Lt) - \bgamma)^{-1} + 4 \varpi} \norm{\xvstar}^2  \\
    & \quad \leq (1- \gamma \varpi) \norm{x - \prox_{U_{\theta}}^{\gamma\kappa}(x) -x}^2 + \gamma \defEns{(2/(\mtt + \Lt) - \bgamma)^{-1} + 4 \varpi} \Rvun^2  \\
    & \quad \leq (1- \gamma \varpi) \norm{x}^2 + \gamma^2 \kappa^2 \Mt^2 + 2 \gamma \kappa \Mt \norm{x} + \gamma \defEns{(2/(\mtt + \Lt) - \bgamma)^{-1} + 4 \varpi} \Rvun^2  \\
    & \quad \leq (1- \gamma \varpi/2) \norm{x}^2 + \gamma^2 \kappa^2 \Mt^2  
                                                                                   + \gamma \defEns{(2/(\mtt + \Lt) - \bgamma)^{-1} + 4 \varpi} \Rvun^2  + 2 \gamma \kappa \Mt \norm{x} - \gamma \varpi \norm{x}^2/2 \\
    & \quad \leq (1- \gamma \varpi/2) \norm{x}^2 
                                                         + \gamma \bgamma \kappa^2 \Mt^2 + \gamma \defEns{(2/(\mtt + \Lt) - \bgamma)^{-1} + 4 \varpi} \Rvun^2  + 2 \gamma \kappa^2 \Mt^2 \varpi^{-1} \eqsp .
    \end{align}
  \end{proof}

    \begin{lemma}
    \label{lemma:drift_det_pula_2}
    Assume \tup{\Cref{assum:potential_regularity}} and \tup{\Cref{assum:potential_drift_2}}. Then for any $\kappa >0$, $\theta \in \Theta$, $\gamma \in \ocint{0, \bgamma}$ with $\bgamma < 2 / \Lt$ and $x \in \rset^{\dim}$, we have
    \begin{align}
      \norm{\prox_{U_{\theta}}^{\gamma\kappa}(x)  -\gamma \nabla_x V_{\theta}(\prox_{U_{\theta}}^{\gamma\kappa}(x))}^2 &\leq \norm{x}^2 + \gamma \left[3 \bgamma \kappa^2 \Mt^2 + 2 \kappa \ct + 2 \kappa(\Rudeux + \Mt \Ruun) \right . \\ & \left . \qquad + (2/L - \bgamma)^{-1}\Rvun^2 - 2\kappa \eta \norm{x} \right] \eqsp .
  \end{align}
\end{lemma}

\begin{proof}
   Let $\kappa >0$, $\theta \in \Theta$, $\gamma \in \ocint{0, \bgamma}$ and $x \in \rset^{\dim}$.
    Using \Cref{assum:potential_regularity}, \Cref{assum:potential_drift_2}, \Cref{lemma:borne_prox} and \Cref{lemma:majo_norm_moins} and \Cref{lemma:contrac_diff} we have 
  \begin{align}
    &\norm{\prox_{U_{\theta}}^{\gamma\kappa}(x)  -\gamma \nabla_x V_{\theta}(\prox_{U_{\theta}}^{\gamma\kappa}(x))}^2 \leq \normLigne{\prox_{U_{\theta}}^{\gamma\kappa}(x)}^2 + \gamma / (2/ \Lt - \bgamma) \Rvun^2  \\
    &\quad \leq \norm{x}^2 + \gamma^2 \kappa^2 \Mt^2 + 2  \langle \prox_{U_{\theta}}^{\gamma\kappa}(x) - x, x\rangle +  \gamma / (2/ \Lt - \bgamma) \Rvun^2 \\
    &\quad \leq \norm{x}^2 + 3 \gamma^2 \kappa^2 \Mt^2  -2 \gamma \kappa U(x) + 2 \gamma \kappa (U(\xustar) + \Mt \normLigne{\xustar}) +  \gamma / (2/ \Lt - \bgamma) \Rvun^2 \\    
    &\quad \leq \norm{x}^2 + 3 \gamma^2 \kappa^2 \Mt^2  -2 \gamma \kappa \eta \norm{x} + 2 \gamma \kappa \ct \\ & \qquad + 2 \gamma \kappa (U(\xustar) + \Mt \normLigne{\xustar}) +  \gamma / (2/ \Lt - \bgamma) \Rvun^2 \\
    &\quad \leq \norm{x}^2 + \gamma \parentheseDeux{3 \bgamma \kappa^2 \Mt^2 + 2 \kappa \ct + 2 \kappa(\Rudeux + \Mt \Ruun) + (2/L - \bgamma)^{-1}\Rvun^2  - 2\kappa \eta \norm{x} } \eqsp .
  \end{align}
\end{proof}

\begin{lemma}
  \label{lemma:drift_det_myula_1}
      Assume \tup{\Cref{assum:potential_regularity}} and \tup{\Cref{assum:potential_drift_1}}. Then for any $\kappa >0$, $\theta \in \Theta$, $\gamma \in \ocint{0, \bgamma}$ with $\bgamma < 2 / (\mtt + \Lt)$ and $x \in \rset^{\dim}$, we have
      \begin{multline}
        \norm{x - \gamma \nabla_x V_{\theta}(x) - \gamma \nabla_x
          U_{\theta}^{\gamma \kappa}(x)}^2 \leq (1 - \gamma \varpi/2) \norm{x}^2
        \\ + \gamma \defEns{(2 /(\mtt + \Lt) - \bgamma)^{-1} + 4\varpi} \Rvun^2 +
        2 \gamma^2 \Mt \Lt \Rvun + \gamma^2 \Mt^2 + 2\gamma \Mt^2(1 + \bgamma
        \Lt)^2 \varpi^{-1} \eqsp , \end{multline}
      with
      $\varpi = \mtt \Lt / (2\mtt + 2\Lt)$.
\end{lemma}

\begin{proof}
  Let $\kappa >0$, $\theta \in \Theta$, $\gamma \in \ocint{0, \bgamma}$ and $x \in \rset^{\dim}$.
  Using \Cref{assum:potential_regularity}, \Cref{assum:potential_drift_1}, \Cref{lemma:borne_prox}, \Cref{lemma:contrac_diff} and that for any $\alpha, \beta \geq 0$, $\max(-\alpha t^2 + 2\beta t) = \beta^2/\alpha$  we have 
  \begin{align}
    &\norm{x - \gamma \nabla_x V_{\theta}(x) - \gamma \nabla_x U_{\theta}^{\gamma \kappa}(x)}^2 \\
    & \qquad \leq \norm[2]{x-\gamma \nabla_x V_{\theta}(x)} + 2 \gamma \Mt \norm{x-\gamma \{\nabla_x V_{\theta}(x) - \nabla_x V_{\theta}(\xstar_{\theta})\}} + \gamma^2 \Mt^2 \\
      & \qquad \leq (1 - \gamma \varpi) \norm{x}^2 + \gamma \defEns{(2 /(\mtt + \Lt) - \bgamma)^{-1} + 4 \varpi} \norm{\xvstar}^2 \\ & \qquad \qquad + 2\gamma \Mt\norm{x} + 2 \gamma^2 \Mt \norm{\nabla_x V_{\theta}(x) - \nabla_x V_{\theta}(\xstar_{\theta})}  + \gamma^2 \Mt^2 \\
    & \qquad \leq (1 - \gamma \varpi) \norm{x}^2 + \gamma \defEns{(2 /(\mtt + \Lt) - \bgamma)^{-1} + 4 \varpi} \norm{\xvstar}^2 \\ & \qquad \qquad + 2\gamma \Mt \norm{x} + 2 \gamma^2 \Mt \Lt \norm{x} + 2 \gamma^2 \Mt \Lt \norm{\xvstar} + \gamma^2 \Mt^2 \\
    & \qquad \leq (1 - \gamma \varpi/2) \norm{x}^2 + \gamma \defEns{(2 /(\mtt + \Lt) - \bgamma)^{-1} + 4 \varpi} \Rvun^2  \\ & \qquad \qquad  + 2 \gamma^2 \Mt \Lt \Rvun + \gamma^2 \Mt^2 + 2\gamma \Mt (1 + \bgamma \Lt) \norm{x} - \gamma \varpi \norm{x}^2/2 \\
      &\qquad \leq (1 - \gamma \varpi/2) \norm{x}^2 + \gamma \defEns{(2 /(\mtt + \Lt) - \bgamma)^{-1} + 4 \varpi} \Rvun^2  \\ & \qquad \qquad  + 2 \gamma^2 \Mt \Lt \Rvun + \gamma^2 \Mt^2 + 2\gamma \Mt^2(1 + \bgamma \Lt)^2 \varpi^{-1}  \eqsp .
  \end{align}
  \end{proof}

      \begin{lemma}
    \label{lemma:drift_det_myula_2}
    Assume \tup{\Cref{assum:potential_regularity}} and
    \tup{\Cref{assum:potential_drift_2}}. Then for any $\kappa >0$,
    $\theta \in \Theta$, $x \in \rset^{\dim}$ and
    $\gamma \in \ocint{0, \bgamma}$ with
    $\bgamma < \min(2/\Lt,\eta/(2\Mt\Lt))$, we have
    \begin{align}
      &\norm{x - \gamma \nabla_x V_{\theta}(x) - \gamma \nabla_x U_{\theta}^{\gamma \kappa}(x)}^2  \\
      & \quad \leq \norm{x}^2 + \gamma \parentheseDeux{(2/L - \bgamma)^{-1}\Rvun^2 + 3 \bgamma \Mt^2 + 2\ct + 2 (\Mt \Ruun + \Rudeux) + 2 \bgamma \Mt \Lt \Rvdeux  - \eta \norm{x}} \eqsp .
  \end{align}
\end{lemma}

\begin{proof}
  Let $\kappa >0$, $\theta \in \Theta$, $\gamma \in \ocint{0, \bgamma}$ and $x \in \rset^{\dim}$.
    Using \Cref{assum:potential_regularity}, \Cref{assum:potential_drift_2},   \eqref{eqII:part_ii_MY-envelope},  \Cref{lemma:borne_prox} and \Cref{lemma:majo_norm_moins}  we have
  \begin{align}
    &\norm{x - \gamma \nabla_x V_{\theta}(x) - \gamma \nabla_x U_{\theta}^{\gamma \kappa}(x)}^2  \\
    &\leq \norm{x - \gamma \nabla_x V_{\theta}(x)}^2 - 2 \gamma  \langle x - \gamma \nabla_x V_{\theta}(x), \nabla_x U_{\theta}^{\gamma \kappa}(x) \rangle + \gamma^2 \Mt^2  \\
    &\leq \norm{x - \gamma \nabla_x V_{\theta}(x)}^2 - 2 \kappa^{-1} \langle x - \gamma \nabla_x V_{\theta}(x), x- \prox_{ U_{\theta}}^{\gamma \kappa}(x) \rangle + \gamma^2 \Mt^2  \\
    &\leq \norm{x - \gamma \nabla_x V_{\theta}(x)}^2 - 2 \kappa^{-1} \langle x , x- \prox_{ U_{\theta}}^{\gamma \kappa}(x) \rangle + 2 \kappa^{-1} \gamma \norm{\nabla_x V_{\theta}(x)}\normLigne{ x- \prox_{ U_{\theta}}^{\gamma \kappa}(x)} + \gamma^2 \Mt^2  \\
    &\leq \norm{x - \gamma \nabla_x V_{\theta}(x)}^2 + 3 \gamma^2 \Mt^2 - 2 \gamma \eta \norm{x} + 2 \gamma \ct + 2 \gamma \parentheseLigne{\Mt \normLigne{\xustar} + U(\xustar)} + 2 \gamma \bgamma \Mt \norm{ \nabla_x V_{\theta}(x)} \\
    &\leq \norm{x - \gamma \nabla_x V_{\theta}(x)}^2 + 3 \gamma \bgamma \Mt^2 - 2 \gamma \eta \norm{x} \\ & \qquad + 2 \gamma \ct + 2 \gamma \parentheseLigne{\Mt \Ruun + \Rudeux} + 2 \gamma \bgamma \Mt \Lt \norm{x} + 2\gamma \bgamma \Mt \Lt \norm{\xvstar}  \\
    &\leq \norm{x - \gamma \nabla_x V_{\theta}(x)}^2 + 3 \gamma \bgamma  \Mt^2 -  \gamma \eta \norm{x} + 2 \gamma \ct + 2 \gamma \parentheseLigne{\Mt \Ruun + \Rudeux} + 2\gamma \bgamma \Mt \Lt \norm{\xvstar} \eqsp, 
  \end{align}
where we have used for the last inequality that $\bgamma < \eta/(2\Mt\Lt)$. Then, we can conclude using \Cref{assum:potential_regularity} and \Cref{lemma:contrac_diff} that
 \begin{align}
    &\norm{x - \gamma \nabla_x V_{\theta}(x) - \gamma \nabla_x U_{\theta}^{\gamma \kappa}(x)}^2  \\
   &\leq \norm{x}^2 + \gamma/(2/L - \bgamma)\Rvun^2 + 3 \gamma \bgamma \Mt^2 -  \gamma \eta \norm{x} + 2 \gamma \ct + 2 \gamma \parentheseLigne{\Mt \Ruun + \Rudeux} + 2\gamma \bgamma \Mt \Lt \Rvun  \\
    &\leq \norm{x}^2 + \gamma \parentheseDeux{(2/L - \bgamma)^{-1}\Rvun^2 + 3 \bgamma \Mt^2 + 2\ct + 2 (\Mt \Ruun + \Rudeux) + 2 \bgamma \Mt \Lt \Rvdeux  - \eta \norm{x}} \eqsp .
  \end{align}
\end{proof}

For $\upsilon \in \rset^{\dim}$ and $\upsigma >0$, denote
$\Upsilon_{\upsilon, \upsigma}$ the $d$-dimensional Gaussian
distribution with mean $\upsilon$ and covariance matrix
$\upsigma^2 \Id$.

\begin{lemma}
  \label{lemma:kl_gauss}
  For any $\upsigma_1, \upsigma_2 >0$ and $\upsilon_1, \upsilon_2 \in \rset^{\dim}$, we have
  \begin{align}
    \KL{\Upsilon_{\upsilon_1, \upsigma_1 \Id}}{\Upsilon_{\upsilon_2, \upsigma_2 \Id}}&= \norm{\upsilon_1 - \upsilon_2}^2/(2\upsigma_2^2) + (d/2)\defEns{-\log(\upsigma_1^2/\upsigma_2^2) - 1 + \upsigma_1^2/\upsigma_2^2} \eqsp .
  \end{align}
  In addition, if $\upsigma_1 \geq \upsigma_2$
  \begin{equation}
    \KL{\Upsilon_{\upsilon_1, \upsigma_1 \Id}}{\Upsilon_{\upsilon_2, \upsigma_2 \Id}} \leq \norm{\upsilon_1 - \upsilon_2}^2/(2\upsigma_2^2) + (d/2)(1 -\upsigma_1^2 / \upsigma_2^2)^2\eqsp .
  \end{equation}
\end{lemma}

\begin{proof}
  Let $X$ be a $d$-dimensional Gaussian random variable with mean
  $\upsilon_1$ and covariance matrix $\upsigma_1^2 \Id$.
  We have that
  \begin{align}
    &\KL{\Upsilon_{\upsilon_1, \upsigma_1 \Id}}{\Upsilon_{\upsilon_2, \upsigma_2 \Id}}= \expe{\log \defEns{(\upsigma_2^2/\upsigma_1^2)^{d/2}\exp\parentheseDeux{-\norm{X - \upsilon_1}^2/(2\upsigma_1^2) + \norm{X - \upsilon_2}^2/(2\upsigma_2^2)}}} \\
    & \qquad = -(d/2)\log(\upsigma_1^2/\upsigma_2^2) + \expe{-\norm{X - \upsilon_1}^2/(2\upsigma_1^2) + \norm{X - \upsilon_2}^2/(2\upsigma_2^2)} \\
    & \qquad = -(d/2)\log(\upsigma_1^2/\upsigma_2^2) + (1/2)(\upsigma_2^{-2} - \upsigma_1^{-2})\expe{-\norm{X - \upsilon_1}^2} + \norm{\upsilon_1^2 - \upsilon_2^2}/(2\upsigma_2^2) \\
    & \qquad = -(d/2)\log(\upsigma_1^2/\upsigma_2^2) + (d/2)(\upsigma_1^2/\upsigma_2^{2} - 1) + \norm{\upsilon_1^2 - \upsilon_2^2}/(2\upsigma_2^2) \\
    & \qquad = \norm{\upsilon_1 - \upsilon_2}^2/(2\upsigma_2^2) + (d/2)\defEns{-\log(\upsigma_1^2/\upsigma_2^2) - 1 + \upsigma_1^2/\upsigma_2^2} \eqsp .
  \end{align}
  In the case where $\upsigma_1 \geq \upsigma_2$, let
  $s = \upsigma_1^2 / \upsigma_2^2 - 1$. Since $s \geq 0$ we have
  $ \log(1 + s) \geq s - s^2$.  Therefore, we get that
\begin{equation}
  -\log(\upsigma_1^2/\upsigma_2^2) - 1 + \upsigma_1^2/\upsigma_2^2 = -\log(1+s) + s \leq s^2 \eqsp ,
\end{equation}
which concludes the proof.
\end{proof}

\subsection{Proof of \Cref{thm:ergo_cv_myula_text}}
\label{sec:myula_cv}

We show that under \Cref{assum:potential_drift_1} or
\Cref{assum:potential_drift_2}, Foster-Lyapunov drifts hold for MYULA
in \Cref{lemma:drift_myula_1} and \Cref{lemma:drift_myula_2}. Combining
these Foster-Lyapunov drifts with an appropriate minorisation
condition \Cref{lemma:minorization_myula}, we obtain the geometric
ergodicity of the underlying Markov chain in \Cref{thm:ergo_cv_myula}.

\begin{lemma}
  \label{lemma:drift_myula_1}
  Assume \tup{\Cref{assum:potential_regularity}} and
  \tup{\Cref{assum:potential_drift_1}}. Then for
  any 
  $\theta \in \Theta$, $\kappa \in \ccint{\ukappa, \bkappa}$ and
  $\gamma \in \ocint{0, \bgamma}$ with $\bkappa \geq 1 \geq \ukappa > 1/2$, $\bgamma < 2 / (\mtt +
  \Lt)$
  , $\Rker_{\gamma, \theta}$ and $\bRker_{\gamma, \theta}$ satisfy
  $\bfDd(W_1, \lambda_2^{\gamma}, b_2\gamma)$ with
  \begin{equation}
    \begin{aligned}
       &\lambda_2 = \exp\parentheseDeux{-\varpi/2} \eqsp ,  \\
       &b_2 =  \defEns{(2 /(\mtt + \Lt) - \bgamma)^{-1} + 4 \varpi} \Rvun^2    + 2 \bgamma \Mt \Lt \Rvun + \bgamma \Mt^2 + 2 d + 2 \Mt^2(1 + \bgamma \Lt)^2 \varpi^{-1} + \varpi /2  \eqsp , \\
       &\varpi = \mtt \Lt / (\mtt + \Lt) \eqsp ,
    \end{aligned}
  \end{equation}
  where for any $x \in \rset^{\dim}$, $W_2(x) = 1 + \normLigne{x}^2$.
  In addition, for any $\pow \in \nsets$, there exist
  $\lambda_m \in \ooint{0,1}$, $b_m \geq 0$ such that for any
  $\theta \in \Theta$, $\kappa \in \ccint{\ukappa, \bkappa}$,
  $\gamma \in \ocint{0, \bgamma}$ with $\bkappa \geq 1 \geq \ukappa > 1/2$,
  $\bgamma < 2 / (\mtt + \Lt)$, $\Rker_{\gamma, \theta}$ and
  $\bRker_{\gamma, \theta}$ satisfy
  $\bfDd(W_m, \lambda_m^{\gamma}, b_m\gamma)$, where $W_m$ is given in
  \eqref{eqII:def_W_pow_alph}.
\end{lemma}

\begin{proof}
  We show the property for $\Rker_{\gamma,\theta}$ only as the proof
  for $\bRker_{\gamma, \theta}$ is identical.  Let
  $\theta \in \Theta$, $\kappa \in \ccint{\ukappa, \bkappa}$,
  $\gamma \in \ocint{0, \bgamma}$ and $x \in \rset^{\dim}$. Let $Z$ be
  a $\dim$-dimensional Gaussian random variable with zero mean and
  identity covariance matrix. Using \Cref{lemma:drift_det_myula_1} we
  have
  \begin{align}
    &\int_{\rset^{\dim}} \norm{y}^2 \Rker_{\gamma, \theta}(x, \rmd y) = \expe{\norm{x - \gamma \nabla_x V_{\theta}(x) - \gamma \nabla_x U_{\theta}^{\gamma \kappa}(x) + \sqrt{2 \gamma} Z}^2} \\
    &\quad = \norm{x - \gamma \nabla_x V_{\theta}(x) - \gamma \nabla_x U_{\theta}^{\gamma \kappa}(x)} + 2 \gamma d \\ 
    &\quad \leq (1 - \gamma \varpi/2) \norm{x}^2 + \gamma \left[ \defEns{(2 /(\mtt + \Lt) - \bgamma)^{-1} + 4 \varpi} \Rvun^2  \right . \\ &\qquad \qquad   \left . + 2 \bgamma \Mt \Lt \Rvun + \bgamma \Mt^2 + 2 d + 2 \Mt^2(1 + \bgamma \Lt)^2 \varpi^{-1} \right] \eqsp .                         
  \end{align}
  Therefore, we get
  \begin{multline}
    \int_{\rset^{\dim}} (1+\norm{y}^2) \Rker_{\gamma, \theta}(x, \rmd y) \leq (1 - \gamma \varpi/2) (1 + \norm{x}^2) + \gamma \left[ \defEns{(2 /(\mtt + \Lt) - \bgamma)^{-1} + 4 \varpi} \Rvun^2  \right . \\   \left . + 2 \bgamma \Mt \Lt \Rvun + \bgamma \Mt^2 + 2 d + 2 \Mt^2(1 + \bgamma \Lt)^2 \varpi^{-1} + \varpi /2 \right] \eqsp ,
  \end{multline}
  which concludes the first part of the proof. Let
  $\Tg(x) = x - \gamma \nabla_x V_{\theta}(x) - \gamma \nabla_x
  U_{\theta}^{\gamma \kappa}(x)$. In the sequel, for any
  $k \in \{1, \dots, m \}$, $b , \tilde{b}_k \geq 0$ and
  $\lambda, \tilde{\lambda}_k \in \coint{0,1}$ are constants
  independent of $\gamma$ which may take different values at each
  appearance. Note that using \Cref{lemma:drift_det_myula_1}, for any
  $k \in \{1, \dots, 2m\}$ there exist
  $\tilde{\lambda}_k \in \ooint{0,1}$ and $\tilde{b}_k \geq 0$ such
  that
  \begin{align}
    \label{eqII:1}
    \norm[k]{\Tg(x)} &\leq \{\tilde{\lambda}_k^{\gamma} \norm{x} + \gamma \tilde{b}_k\}^{k} \\ & \leq \tilde{\lambda}_k^{\gamma k} \norm{x}^{k} + \gamma 2^{k} \max(\tilde{b}_k, 1)^{k} \max(\bgamma, 1)^{2 k - 1} \defEns{1  + \norm{x}^{k-1}} \\ & \leq \tilde{\lambda}_k^{\gamma} \norm[k]{x} + \tilde{b}_k \gamma  \defEns{1+\norm[k-1]{x}} \leq (1 + \norm{x}^{k}) (1 + \tilde{b}_k \gamma) \eqsp. 
  \end{align}
Therefore, combining \eqref{eqII:1} and the Cauchy-Schwarz inequality we obtain
  \begin{align}
    &\int_{\rset^{\dim}} (1 + \norm{y}^2) \Rker_{\gamma, \theta}(x , \rmd y)  = 1 + \expe{(\norm{\Tg(x)}^2 + 2 \sqrt{2\gamma}  \langle \Tg(x), Z \rangle + 2 \gamma \norm{Z}^2)^m} \\
                         & \quad  =1+ \sum_{k=0}^m \sum_{\ell =0}^{k} {m \choose k} {k \choose \ell}  \norm{\Tg(x)}^{2(m-k)} 2^{(3k-\ell)/2}  \gamma^{(k+\ell)/2} \expe{\langle \Tg(x), Z \rangle^{k -\ell}\norm{Z}^{2\ell}} \\
    &\quad \leq  1 + \norm{\Tg(x)}^{2m} \\
    & \quad \quad 
      + 2^{3m/2} \sum_{k=1}^m \sum_{\ell =0}^{k} {m \choose k} {k \choose \ell}  \norm{\Tg(x)}^{2(m-k)}  \gamma^{(k+\ell)/2} \expe{\langle \Tg(x), Z \rangle^{k -\ell}\norm{Z}^{2\ell}} \1_{\{(1,0)\}^{\complementary}}(k,\ell) \\
    &\quad \leq  1 + \norm{\Tg(x)}^{2m} \\
    & \quad \quad 
      + \gamma 2^{3m/2} \sum_{k=1}^m \sum_{\ell =0}^{k} {m \choose k} {k \choose \ell}  \norm{\Tg(x)}^{2m - k-\ell}  \bgamma^{(k+\ell)/2 - 1} \expe{\norm{Z}^{k + \ell}} \1_{\{(1,0)\}^{\complementary}}(k,\ell) \\
    &\quad \leq 1 + \lambda_{2m}^{\gamma} \norm{x}^{2 m} + b_{2m} \gamma \defEns{1 + \norm{x}^{2 m - 1}} \\ & \quad \quad + \gamma 2^{3m/2}  2^{2m} \max(\bgamma, 1)^{2m}  \sup_{k \in \{1, \dots, m\}} \defEns{(1 + \tilde{b}_k \bgamma) \expe{\norm{Z}^k}}  (1 + \norm{x}^{2m - 1}) \\ 
                         & \quad \leq 1 + \lambda^{\gamma} \norm{x}^{2m} + \gamma b (1 + \norm{x}^{2m-1}) \\ & \quad \leq \lambda^{\gamma/2} (1 +  \norm{x}^{2m}) + \gamma b(1 + \norm{x}^{2 m -1}) + \lambda^{\gamma} (1 +  \norm{x}^{2m}) - \lambda^{\gamma/2} (1 +  \norm{x}^{2m}) \eqsp .
  \end{align}
  Using that $\lambda^{\gamma} - \lambda^{\gamma /2} \leq -\log(1/\lambda) \gamma  \lambda^{\gamma/2} / 2$, concludes the proof.
\end{proof}

\begin{lemma}
  \label{lemma:drift_myula_2}
  Assume \tup{\Cref{assum:potential_regularity}} and
  \tup{\Cref{assum:potential_drift_2}}. Then for
  any 
  $\theta \in \Theta$, $\kappa \in \ccint{\ukappa, \bkappa}$ and
  $\gamma \in \ocint{0, \bgamma}$ with $\bkappa \geq 1 \geq \ukappa > 1/2$, 
  $\bgamma < \min(2 / \Lt, \eta / (2 \Mt \Lt))$,
  $\Rker_{\gamma, \theta}$ and $\bRker_{\gamma, \theta}$ satisfy
  $\bfDd(W, \lambda^{\gamma}, b\gamma)$ with
   \begin{equation}
     \label{eqII:def_const_myula}
    \begin{aligned}
      &\lambda = \rme^{-\alpha^2} \eqsp , \\
      &b_e = (4/L - 2\bgamma)^{-1}\Rvun^2 + (3/2) \bgamma \Mt^2 + \ct +  \Mt \Ruun + \Rudeux +  \bgamma \Mt \Lt \Rvdeux +d +2 \alpha \eqsp , \\
      &b= \alpha b_e \rme^{\alpha \bgamma b_e}W(R) \eqsp , \\ 
      &W = W_{\alpha}  \eqsp ,  
      \qquad \alpha < \eta/8 \eqsp , \\
      &R_{\eta} = \max\left(2b_e/(\eta - 8\alpha) , 1\right) \eqsp ,
    \end{aligned}
  \end{equation}
  where $W_{\alpha}$ is given in \eqref{eqII:def_W_pow_alph}.
\end{lemma}

\begin{proof}
  We show the property for $\Rker_{\gamma,\theta}$ only as the proof
  for $\bRker_{\gamma, \theta}$ is identical.  Let
  $\theta \in \Theta$, $\kappa \in \ccint{\ukappa, \bkappa}$
  $\gamma \in \ocint{0, \bgamma}$, $x \in \rset^{\dim}$ and $Z$ be a
  $\dim$-dimensional Gaussian random variable with zero mean and
  identity covariance matrix.  Using \Cref{lemma:drift_det_myula_2} we
  have
  \begin{align}
    &\int_{\rset^{\dim}} \norm{y}^2 \Rker_{\gamma, \theta}(x, \rmd y) = \norm{x - \gamma \nabla_x V_{\theta}(x) - \gamma \nabla_x U_{\theta}^{\gamma \kappa}}^2 + 2 \gamma d \\
    &\leq \norm{x}^2 + \gamma \parentheseDeux{(2/L - \bgamma)^{-1}\Rvun^2 + 3 \bgamma \Mt^2 + 2\ct + 2 (\Mt \Ruun + \Rudeux) + 2 \bgamma \Mt \Lt \Rvdeux + 2d - \eta \norm{x}} \eqsp .
  \end{align}
Using the log-Sobolev inequality \cite[Proposition 5.4.1]{bakry:gentil:ledoux:2014} and Jensen's inequality we get that
\begin{align}
  \label{eqII:myula_log_sob}
    \Rker_{\gamma, \theta} W(x) &\leq \exp\parentheseDeux{\alpha \Rker_{\gamma, \theta} \phi(x) + \alpha^2 \gamma} \\
    &\leq \exp \parentheseDeux{\alpha \parenthese{1 + \int_{\rset^{\dim}} \norm{y}^2 \Rker_{\gamma, \theta}(x, \rmd y)}^{1/2} + \alpha^2 \gamma} \eqsp .
  \end{align}
We now distinguish two cases:
\begin{enumerate}[label=(\alph*), wide, labelwidth=!, labelindent=0pt]
\item If $\norm{x } \geq R_{\eta}$, recalling that $R_{\eta}$ is given in \eqref{eqII:def_const_myula}, then 
  \begin{equation} (2/L - \bgamma)^{-1}\Rvun^2 + 3 \bgamma \Mt^2 + 2\ct + 2 (\Mt \Ruun + \Rudeux) + 2 \bgamma \Mt \Lt \Rvdeux + 2d - \eta \norm{x} \leq - 8\alpha \norm{x} \eqsp .\end{equation} In this case using that $\phi^{-1}(x) \norm{x } \geq 1/2$ and that for any $t \geq 0$, $\sqrt{1 + t} \leq 1 + t/2$ we have
  \begin{align}
    &\parenthese{1 + \int_{\rset^{\dim}} \norm{y}^2 \Rker_{\gamma, \theta}(x, \rmd y)}^{1/2} - \phi(x) \leq   \\
    &\quad \leq \left . \gamma \phi^{-1}(x) \parenthese{(2/L - \bgamma)^{-1}\Rvun^2 + 3 \bgamma \Mt^2 + 2\ct + 2 (\Mt \Ruun + \Rudeux) + 2 \bgamma \Mt \Lt \Rvdeux + 2 d- \eta \norm{x}} \middle / 2  \right . \\
    &\quad \leq  - 4 \alpha \gamma \phi^{-1}(x)\norm{x} \leq - 2 \alpha \gamma \eqsp .
  \end{align}
  Hence,
  \begin{equation}
    \Rker_{\gamma, \theta}W(x) \leq \parentheseDeux{\alpha \parenthese{1 + \int_{\rset^{\dim}} \norm{y}^2 \Rker_{\gamma, \theta}(x, \rmd y)}^{1/2} + \alpha^2 \gamma} \leq \rme^{-\alpha^2 \gamma} W(x) \eqsp .
  \end{equation}
\item If $\norm{x} \leq R_{\eta}$ then using that for any $t \geq 0$, $\sqrt{1 + t} \leq 1 + t/2$ we have
  \begin{multline}
    \parenthese{1 + \int_{\rset^{\dim}} \norm{y}^2 \Rker_{\gamma, \theta}(x, \rmd y)}^{1/2} - \phi(x ) \\ \leq \gamma ((4/L - 2\bgamma)^{-1}\Rvun^2 + (3/2) \bgamma \Mt^2 + \ct +  \Mt \Ruun + \Rudeux +  \bgamma \Mt \Lt \Rvdeux +d )\eqsp .
  \end{multline}
  Therefore, using \eqref{eqII:myula_log_sob}, we get
  \begin{multline}
    \Rker_{\gamma, \theta} W(x) \\ \leq \exp\parentheseDeux{\alpha \gamma \defEns{(4/L - 2\bgamma)^{-1}\Rvun^2 + (3/2) \bgamma \Mt^2 + \ct +  \Mt \Ruun + \Rudeux +  \bgamma \Mt \Lt \Rvdeux +d + \alpha}} W(x) \eqsp .
  \end{multline}
  Since for all $a \geq b$, $\rme^{a} - \rme^{b} \leq (a-b) \rme^{a}$ we obtain that
  \begin{equation}
    \Rker_{\gamma, \theta}W(x) \leq \lambda^{\gamma}W(x) + \gamma \alpha b_e \rme^{\alpha \bgamma b_e}W(R_{\eta}) \eqsp ,
  \end{equation}
which concludes the proof.
\end{enumerate}
\end{proof}

\begin{lemma}
  \label{lemma:minorization_myula}
  Assume \tup{\Cref{assum:potential_regularity}}. For any
  $\kappa \in \ccint{\ukappa, \bkappa}$, $\theta \in \Theta$,
  $\gamma \in \ocint{0, \bgamma}$ with
  $\bkappa \geq 1 \geq \ukappa > 1/2$, $\bgamma < (2 - 1/\ukappa)/\Lt$
  and $x, y \in \rset^{\dim}$
  \begin{equation}
    \max \parenthese{\tvnorm{\updelta_x \Rker_{\gamma, \theta}^{\step} - \updelta_y \Rker_{\gamma, \theta}^{\step}}, \tvnorm{\updelta_x \bRker_{\gamma, \theta}^{\step} - \updelta_y \bRker_{\gamma, \theta}^{\step}}} \leq 1 - 2 \Phibf\defEns{- \norm{x -y}/(2\sqrt{2})} \eqsp ,
  \end{equation}
where $\Phibf$ is the cumulative distribution function of the standard normal distribution on $\rset$.
\end{lemma}

\begin{proof}
  We only show that for any $\theta \in \Theta$,
  $\kappa \in \ccint{\ukappa, \bkappa}$,
  $\gamma \in \ocint{0, \bgamma}$ with
  $\bkappa \geq 1 \geq \ukappa > 1/2$, $\bgamma < (2 - 1/\kappa)/\Lt$
  and $x, y \in \rset^{\dim}$, we have
  $\tvnorm{\updelta_x \Rker_{\gamma, \theta}^{\step} - \updelta_y
    \Rker_{\gamma, \theta}^{\step}} \leq 1 - 2 \Phibf\defEns{- \norm{x
      -y}/(2\sqrt{2})}$ as the proof of for $\bRker_{\gamma, \theta}$
  is similar.  Let 
  $\kappa \in \ccint{\ukappa, \bkappa}$, $\theta \in \Theta$,
  $\gamma \in \ocint{0, \bgamma}$.  We have that
  $x \mapsto V_{\theta}(x) + U_{\theta}^{\gamma \kappa}(x)$ is convex,
  continuously differentiable and satisfies for any
  $x,y \in \rset^{\dim}$
\begin{equation}
  \norm{\nabla_x V_{\theta}(x) + \nabla_x U_{\theta}^{\gamma \kappa}(x) - \nabla_x V_{\theta}(y) - \nabla_x U_{\theta}^{\gamma \kappa}(y)} \leq \defEnsLigne{\Lt + 1/(\gamma\kappa)} \norm{x-y} \eqsp ,
\end{equation}
Combining this result with \cite[Theorem 2.1.5, Equation (2.1.8)]{nesterov2013introductory} and the fact that $\gamma \leq 2/\defEnsLigne{\Lt + 1/(\gamma\kappa)}$ since $\bgamma \leq (2 - 1/\kappa)/\Lt$, we have for any $x,y \in \rset^{\dim}$ 
\begin{equation}
  \norm{x - \gamma \nabla_x V_{\theta}(x) - \gamma\nabla_x U_{\theta}^{\gamma \kappa}(x) - y +\gamma\nabla_x V_{\theta}(y) + \gamma\nabla_x U_{\theta}^{\gamma \kappa}(y)} \leq \norm{x-y} \eqsp .
\end{equation}
The proof is then an application of \cite[Proposition
3b]{debortoli2019convergence} with $\ell \leftarrow 1$, for any
$x \in \rset^d$,
$\Tg(x) \leftarrow x- \gamma \nabla_x V_{\theta}(x) - \gamma \nabla_x
\nabla U_{\theta}^{\gamma \kappa}(x)$ and $\Pi \leftarrow \Id$.
\end{proof}

\begin{theorem}
  \label{thm:ergo_cv_myula}
  Assume \tup{\Cref{assum:potential_regularity}} and
  \tup{\Cref{assum:potential_drift_1}} or
  \tup{\Cref{assum:potential_drift_2}}.  Let
  $\bkappa \geq 1 \geq \ukappa > 1/2$,
  $\bgamma < \min\defEnsLigne{(2 - 1/\ukappa)/\Lt, 2 / (\mtt + \Lt)}$
  if \tup{\Cref{assum:potential_drift_1}} holds and
  $\bgamma < \min \defEnsLigne{(2 - 1 / \ukappa) / \Lt, \eta / (2 \Mt
    \Lt)}$ if \tup{\Cref{assum:potential_drift_2}} holds.  Then for
  any $a \in \ocint{0,1}$, there exist $A_{2, a} \geq 0$ and
  $\rho_{a} \in \ooint{0,1}$ such that for any $\theta \in \Theta$,
  $\kappa \in \ccint{\ukappa, \bkappa}$,
  $\gamma \in \ocint{0, \bgamma}$, $\Rker_{\gamma, \theta}$ and
  $\bRker_{\gamma, \theta}$ admit invariant probability measures
  $\pi_{\gamma, \theta}$, respectively $\bpi_{\gamma, \theta}$, and
  for any $x, y \in \rset^{\dim}$ and $n \in \nset$ we have
\begin{equation}
  \begin{aligned}
        \max \parenthese{ \Vnorm[W^a]{\updelta_x \Rker_{\gamma, \theta}^n - \pi_{\gamma, \theta}}, \Vnorm[W^a]{\updelta_x \bRker_{\gamma, \theta}^n - \bpi_{\gamma, \theta}} } &\leq A_{2,  a} \rho_{ a}^{\gamma n} W^a(x) \eqsp ,  \\ 
    \max \parenthese{ \Vnorm[W^a]{\updelta_x \Rker_{\gamma, \theta}^n - \updelta_y \Rker_{\gamma, \theta}^n}, \Vnorm[W^a]{\updelta_x \bRker_{\gamma, \theta}^n - \updelta_y \bRker_{\gamma, \theta}^n}}  &\leq A_{2,  a} \rho_{ a}^{\gamma n}  \defEns{W^a(x) + W^a(y)} \eqsp , 
    \end{aligned}
    \end{equation}
    with $W =W_m$ and $\pow \in \nsets$ if
    \tup{\Cref{assum:potential_drift_1}} holds and $W = W_{\alpha}$
    with $\alpha < \min(\ukappa \eta /4, \eta /8)$ if
    \tup{\Cref{assum:potential_drift_2}} holds, see
    \eqref{eqII:def_W_pow_alph}.
\end{theorem}

\begin{proof}
    We only show that for any $a \in \ocint{0,1}$, there exist
  $A_{2, a} \geq 0$ and $\rho_{a} \in \ooint{0,1}$ such that for any
  $\theta \in \Theta$, $\kappa \in \ccint{\ukappa, \bkappa}$ and
  $\gamma \in \ocint{0, \bgamma}$ we have
  $\Vnorm[W^a]{\updelta_x \Rker_{\gamma, \theta}^n - \pi_{\gamma,
      \theta}} \leq A_{2, a} \rho_{ a}^{\gamma n} W^a(x)$ and
  $\Vnorm[W^a]{\updelta_x \Rker_{\gamma, \theta}^n - \updelta_y
    \Rker_{\gamma, \theta}^n} \leq A_{2, a} \rho_{ a}^{\gamma n}
  \defEns{W^a(x) + W^a(y)}$, since the proof for
  $\bRker_{\gamma,\theta}$ is similar .  Let $a \in \ccint{0,1}$.
  First, using Jensen's inequality and \Cref{lemma:drift_myula_1} if
  \Cref{assum:potential_drift_1} holds or \Cref{lemma:drift_myula_2} if
  \Cref{assum:potential_drift_2} holds, we get that there exist
  $\lambda_{ a}$ and $b_{a}$ such that for any $\theta \in \Theta$,
  $\kappa \in \ccint{\ukappa, \bkappa}$,
  $\gamma \in \ocint{0, \bgamma}$, $\Rker_{\gamma, \theta}$ and
  $\bRker_{\gamma, \theta}$ satisfy
  $\bfDd(W^a, \lambda_{a}^{\gamma}, b_{a}\gamma)$.  Combining
  \cite[Theorem 6]{debortoli2019convergence},
  \Cref{lemma:minorization_myula} and
  $\bfDd(W^a, \lambda_{ a}^{\gamma}, b_{a} \gamma)$, we get that there
  exist $\bA_{2, a} \geq 0$ and $\rho_{a} \in \ooint{0,1}$ such that
  for any $\theta \in \Theta$, $\kappa \in \ccint{\ukappa, \bkappa}$,
  $\gamma \in \ocint{0, \bgamma}$, $x, y \in \rset^{\dim}$ and
  $n \in \nset$, $\Rker_{\gamma, \theta}$ and
  $\bRker_{\gamma, \theta}$ admit invariant probability measures
  $\pi_{\gamma, \theta}$ and $\bpi_{\gamma, \theta}$ respectively and
\begin{equation}
  \label{eqII:contrac}
      \max \defEns{ \Vnorm[W^a]{\updelta_x \Rker_{\gamma, \theta}^n - \updelta_y \Rker_{\gamma, \theta}^n}, \Vnorm[W^a]{\updelta_x \bRker_{\gamma, \theta}^n - \updelta_y \bRker_{\gamma, \theta}^n}} \leq \bA_{2, a} \rho_{a}^{\gamma n} \defEns{W^a(x) + W^a(y)} \eqsp .
    \end{equation}
    Using that for any 
    $\theta \in \Theta$, $\kappa \in \ccint{\ukappa, \bkappa}$ and
    $\gamma \in \ocint{0,
      \bgamma}$ 
    , $\Rker_{\gamma, \theta}$ and $\bRker_{\gamma, \theta}$ satisfy
    $\bfDd(W^a, \lambda_{a}^{\gamma}, b_{ a}\gamma)$ and \cite[Lemma
    S2]{de2019efficient} we have
    \begin{equation}
      \label{eqII:majo_pi}
      \pi_{\gamma, \theta} (W^a) \leq b_{a} \gamma / (1 - \lambda_{a}^{\gamma}) \leq b_{a} \lambda_{a}^{-\bgamma} / \log(1/\lambda_{a}) \eqsp .
    \end{equation}
    Hence, combining \eqref{eqII:contrac} and \eqref{eqII:majo_pi}, we
    have for any 
    $\theta \in \Theta$, $\kappa \in \ccint{\ukappa, \bkappa}$,
    $\gamma \in \ocint{0,
      \bgamma}$ 
    and $n \in \nset$
    \begin{equation}
      \max \defEns{ \Vnorm[W]{\updelta_x \Rker_{\gamma, \theta}^n - \pi_{\gamma, \theta}}, \Vnorm[W]{\updelta_x \bRker_{\gamma, \theta}^n - \bpi_{\gamma, \theta}}}  \leq \bA_{2,a} \rho_{ a}^{\gamma n} (1 + b_{a} \lambda_{a}^{-\bgamma} / \log(1/\lambda_{a}))W^a(x) \eqsp .
    \end{equation}
    We conclude upon letting $A_{2, a} = \bA_{2, a} (1 + b_{a} \lambda_{ a}^{-\bgamma} / \log(1/\lambda_{a}))$.
  \end{proof}

\subsection{Proof of \Cref{thm:ergo_cv_pula_text}}
\label{sec:pula_cv}
We show that under \Cref{assum:potential_drift_1} or
\Cref{assum:potential_drift_2}, Foster-Lyapunov drifts hold for PULA
in \Cref{lemma:drift_pula_1} and \Cref{lemma:drift_pula_2}. Combining
these Foster-Lyapunov drifts with an appropriate minorisation
condition \Cref{lemma:minorization_pula}, we obtain the geometric
ergodicity of the underlying Markov chain in \Cref{thm:ergo_cv_pula}.
\begin{lemma}
  \label{lemma:drift_pula_1}
  Assume \tup{\Cref{assum:potential_regularity}} and
  \tup{\Cref{assum:potential_drift_1}}. Then for any
  $\theta \in \Theta$, $\kappa \in \ccint{\ukappa, \bkappa}$ and
  $\gamma \in \ocint{0, \bgamma}$ with
  $\bkappa \geq 1 \geq \ukappa > 1/2$ and
  $\bgamma < 2 / (\mtt + \Lt)$, $\Sker_{\gamma, \theta}$ and
  $\bSker_{\gamma, \theta}$ satisfy
  $\bfDd(W_1, \lambda^{\gamma}_2, b_2\gamma)$ with
  \begin{equation}
    \begin{aligned}
       &\lambda_2 = \exp\parentheseDeux{-\varpi/2} \eqsp ,  \\
       &b_2 = \bgamma \bkappa^2 \Mt^2   + \defEns{(2/(\mtt + \Lt) - \bgamma)^{-1} + 4 \varpi} \Rvdeux^2 + 2 d + 2 \bkappa^2 \Mt^2\varpi^{-1} + \varpi / 2   \eqsp , \\
       &\varpi = \mtt \Lt / (\mtt + \Lt) \eqsp ,
    \end{aligned}
  \end{equation}
  where for any $x \in \rset^{\dim}$, $W_1(x) = 1 + \normLigne{x}^2$.
  In addition, for any $\pow \in \nsets$, there exist
  $\lambda_{\pow} \in \ooint{0,1}$, $b_{\pow} \geq 0$ such that for
  any $\theta \in \Theta$, $\kappa \in \ccint{\ukappa, \bkappa}$ and
  $\gamma \in \ocint{0, \bgamma}$ with $\bkappa \geq 1 \geq \ukappa > 1/2$ and
  $\bgamma < 2/(\mtt +
  \Lt)$, 
  $\Sker_{\gamma, \theta}$ and $\bSker_{\gamma, \theta}$ satisfy
  $\bfDd(W_{\pow}, \lambda^{\gamma}_{\pow}, b_{\pow}\gamma)$, where
  $W_m$ is given in \eqref{eqII:def_W_pow_alph}.
\end{lemma}

\begin{proof}
  We show the property for $\Sker_{\gamma,\theta}$ only as the proof
  for $\bSker_{\gamma, \theta}$ is identical.  Let
  $\theta \in \Theta$, $\kappa \in \ccint{\ukappa, \bkappa}$,
  $\gamma \in \ocint{0, \bgamma}$ and $x \in \rset^{\dim}$.  Let $Z$
  be a $\dim$-dimensional Gaussian random variable with zero mean and
  identity covariance matrix. Using \Cref{lemma:drift_det_pula_1} we
  have
  \begin{multline}
    \int_{\rset^{\dim}} \norm{y}^2 \Sker_{\gamma, \theta}(x, \rmd y) = \expe{\norm{\prox_{U_{\theta}}^{\gamma\kappa}(x)  -\gamma \nabla_x V_{\theta}(\prox_{U_{\theta}}^{\gamma\kappa}(x)) + \sqrt{2 \gamma} Z}^2}\\
    \leq
      (1- \gamma \varpi/2) \norm{x}^2 + \gamma \left[ \bgamma \kappa^2 \Mt^2 + \defEns{(2/(\mtt + \Lt) - \bgamma)^{-1} + 4 \varpi} \Rvun^2 \right.  \ \left.+ 2  \kappa^2 \Mt^2\varpi^{-1} \right] + 2 \gamma d \eqsp .
  \end{multline}
  Therefore, we get
  \begin{multline}
    \int_{\rset^{\dim}} (1 + \norm{y}^2) \Sker_{\gamma, \theta}(x, \rmd y) \leq (1- \gamma \varpi / 2) (1 + \norm{x}^2)   + \gamma \left[ \bgamma \kappa^2 \Mt^2  \right . \\ \left . + \defEns{(2/(\mtt + \Lt) - \bgamma)^{-1} + 4 \varpi} \Rvun^2 + 2 d + 2 \kappa^2 \Mt^2\varpi^{-1} + \varpi/2 \right]\eqsp ,
  \end{multline}
  which concludes the first part of the proof using that for any
  $t \geq 0$, $1 -t \leq \rme^{-t}$.  The proof of the result for
  $W = W_{\pow}$ with $\pow \in \nsets$ is a straightforward
  adaptation of the one of \Cref{lemma:drift_myula_1} and is left to
  the reader.
\end{proof}

\begin{lemma}
  \label{lemma:drift_pula_2}
  Assume \tup{\Cref{assum:potential_regularity}} and
  \tup{\Cref{assum:potential_drift_2}}. Then for
  any 
  $\theta \in \Theta$, $\kappa \in \ccint{\ukappa, \bkappa}$ and
  $\gamma \in \ocint{0, \bgamma}$ with $\bkappa \geq 1 \geq \ukappa > 1/2$ and
  $\bgamma < 2 / \Lt
  $, $\Sker_{\gamma, \theta}$ and $\bSker_{\gamma, \theta}$ satisfy
  $\bfDd(W, \lambda^{\gamma}, b\gamma)$ with
  \begin{equation}
    \begin{aligned}
      &\lambda = \rme^{-\alpha^2} \eqsp , \\
      &b_e = (3/2)\bgamma \bkappa^2 \Mt^2 + \bkappa \ct + \bkappa(\Rudeux + \Mt \Ruun) + (4/L - 2\bgamma)^{-1}\Rvun^2 + d + 2 \alpha  \\ 
      &b = \alpha b_e \rme^{\alpha \bgamma b_e} W(R) \eqsp ,\\
      &W = W_{\alpha} \eqsp ,  \qquad 0 < \alpha < \ukappa \eta/4 \eqsp , \\
      &R_{\eta} = \max \parenthese{
        b_e / (\ukappa\eta  - 4 \alpha), 1} \eqsp ,
    \end{aligned}
  \end{equation}
  and where $W_{\alpha}$ is given in \eqref{eqII:def_W_pow_alph}.
\end{lemma}

\begin{proof}We show the property for $\Sker_{\gamma,\theta}$ only as
  the proof for $\bSker_{\gamma, \theta}$ is identical.  Let
  $\theta \in \Theta$, $\kappa \in \ccint{\ukappa, \bkappa}$,
  $\gamma \in \ocint{0, \bgamma}$, $x \in \rset^{\dim}$, and $Z$ be a
  $\dim$-dimensional Gaussian random variable with zero mean and
  identity covariance matrix.  Using \Cref{lemma:drift_det_pula_2} we
  have
  \begin{align}
    &\int_{\rset^{\dim}} \norm{y}^2 \Sker_{\gamma, \theta}(x ,\rmd y)  \leq \norm{\prox_{U_{\theta}}^{\gamma\kappa}(x)  -\gamma \nabla_x V_{\theta}(\prox_{U_{\theta}}^{\gamma\kappa}(x))}^2 + 2 \gamma d \\
    & \qquad \leq \norm{x}^2 + \gamma \parentheseDeux{3 \bgamma \kappa^2 \Mt^2 + 2 \kappa \ct + 2 \kappa(\Rudeux + \Mt \Ruun) + (2/L - \bgamma)^{-1}\Rvun^2  + 2d - 2\kappa \eta \norm{x} } \eqsp .
  \end{align}
Using the log-Sobolev inequality \cite[Proposition 5.4.1]{bakry:gentil:ledoux:2014} and Jensen's inequality we get that
\begin{align}
  \label{eqII:ineq_orig}
    \Sker_{\gamma, \theta} W(x) &\leq \exp\parentheseDeux{\alpha \ \Sker_{\gamma, \theta} \phi(x) + \alpha^2 \gamma} \\
    &\leq \exp \parentheseDeux{\alpha \parenthese{1 + \int_{\rset^{\dim}} \norm{y}^2 \Sker_{\gamma, \theta}(x ,\rmd y)}^{1/2} + \alpha^2 \gamma} \eqsp .
  \end{align}
We now distinguish two cases.
\begin{enumerate}[label=(\alph*), wide, labelwidth=!, labelindent=0pt]
\item If $\norm{x} \geq R_{\eta}$ then $\phi^{-1}(x) \norm{x} \geq 1/2$ and
  $  3 \bgamma \kappa^2 \Mt^2 + 2 \kappa \ct + 2 \kappa(\Rudeux + \Mt \Ruun) + (2/L - \bgamma)^{-1}\Rvun^2 + 2d  - 2\kappa \eta \norm{x}  \leq - 8\alpha \norm{x}$. In this case using that for any $t \geq 0$, $\sqrt{1 + t} - 1 \leq t/2$ we get
  \begin{align}
    &\parenthese{1 + \int_{\rset^{\dim}} \norm{y}^2 \Sker_{\gamma, \theta}(x ,\rmd y)}^{1/2} - \phi(x) \\
    & \qquad \quad \leq \gamma \phi^{-1}(x) \parentheseDeux{3 \bgamma \kappa^2 \Mt^2 + 2 \kappa \ct + 2 \kappa(\Rudeux + \Mt \Ruun) + (2/L - \bgamma)^{-1}\Rvun^2 + 2d  - 2\kappa \eta \norm{x}}/2  \\
    & \qquad \quad \leq  - 4 \alpha \gamma \phi^{-1}(x)\norm{x} \leq - 2 \alpha \gamma \eqsp .
  \end{align}
  Hence,
  \begin{equation}
    \Sker_{\gamma, \theta}W(x) \leq \exp\parentheseDeux{\alpha \parenthese{1 + \int_{\rset^{\dim}} \norm{y}^2 \Sker_{\gamma, \theta}(x ,\rmd y)}^{1/2} + \alpha^2 \gamma} \leq \rme^{-\alpha^2 \gamma} W(x) \eqsp .
  \end{equation}
\item If $\norm{x} \leq R_{\eta}$ then using that for any $t \geq 0$, $\sqrt{1 + t} - 1 \leq t/2$
  \begin{multline}
    \parenthese{1 + \int_{\rset^{\dim}} \norm{y}^2 \Sker_{\gamma, \theta}(x ,\rmd y)}^{1/2} - \phi(x) \\ \leq 
    \gamma \parentheseDeux{(3/2)\bgamma \kappa^2 \Mt^2 + \kappa \ct + \kappa(\Rudeux + \Mt \Ruun) + (4/L - 2\bgamma)^{-1}\Rvun^2 + d}  \eqsp .
  \end{multline}
  Therefore we get using \eqref{eqII:ineq_orig}
  \begin{multline}
    \Sker_{\gamma, \theta} W(x) / W(x) \\ \leq \exp\parentheseDeux{
      \alpha \gamma \defEns{(3/2)\bgamma \kappa^2 \Mt^2 + \kappa \ct + \kappa(\Rudeux + \Mt \Ruun) + (4/L - 2\bgamma)^{-1}\Rvun^2 + d + \alpha}
    } \leq \rme^{\alpha b_e \gamma} \eqsp .
  \end{multline}
  Since for all $a \geq b$, $\rme^{a} - \rme^{b} \leq (a-b) \rme^{a}$ we obtain that
  \begin{align}
    \Sker_{\gamma, \theta}W(x) \leq \lambda^{\gamma} W(x) + \gamma \alpha b_e \rme^{\alpha \bgamma b_e}  W(R_{\eta}) \eqsp ,
  \end{align}
which concludes the proof.
\end{enumerate}
\end{proof}

\begin{lemma}
  \label{lemma:minorization_pula}
  Assume \tup{\Cref{assum:potential_regularity}}. For any
  $\theta \in \Theta$, $\kappa \in \ccint{\ukappa, \bkappa}$ and
  $\gamma \in \ocint{0, \bgamma}$ with $\bkappa \geq 1 \geq \ukappa > 1/2$,
  $\bgamma < 2 / \Lt$
  and $x, y \in \rset^{\dim}$
  \begin{equation}
    \max \parenthese{ \tvnorm{\updelta_x \Sker_{\gamma, \theta}^{\step} - \updelta_y \Sker_{\gamma, \theta}^{\step}},  \tvnorm{\updelta_x \bSker_{\gamma, \theta}^{\step} - \updelta_y \bSker_{\gamma, \theta}^{\step}}}  \leq 1 - 2 \Phibf\defEns{- \norm{x -y}/(2 \sqrt{2})} \eqsp ,
  \end{equation}
where $\Phibf$ is the cumulative distribution function of the standard normal distribution on $\rset$.
\end{lemma}

\begin{proof}
  We only show that for any $\theta \in \Theta$,
  $\kappa \in \ccint{\ukappa, \bkappa}$,
  $\gamma \in \ocint{0, \bgamma}$ with $\bgamma < 2/ \Lt$, and
  $x,y \in \rset^{\dim}$,
  $\tvnorm{\updelta_x \Sker_{\gamma, \theta}^{\step} - \updelta_y
    \Sker_{\gamma, \theta}^{\step}} \leq 1 - 2 \Phibf\defEns{- \norm{x
      -y}/(2 \sqrt{2})} $ since the proof for
  $\bSker_{\gamma, \theta}$ is
  similar. 
  Let $\theta \in \Theta$, $\kappa \in \ccint{\ukappa, \bkappa}$,
  $\gamma \in \ocint{0,
    \bgamma}$. 
  Using \cite[Theorem 2.1.5, Equation
  (2.1.8)]{nesterov2013introductory} and that the proximal operator is
  non-expansive \cite[Proposition 12.28]{bauschke2017convex}, we have
  for any $x,y \in \rset^{\dim}$
  \begin{align}
    &\norm{\prox_{U_{\theta}}^{\gamma\kappa}(x) - \prox_{U_{\theta}}^{\gamma \kappa}(y)- \gamma(\nabla_x V_{\theta}(\prox_{U_{\theta}}^{\gamma\kappa}(x)) - \nabla_x V_{\theta}(\prox_{U_{\theta}}^{\gamma\kappa}(y)))} \\ &\leq \norm{\prox_{U_{\theta}}^{\gamma\kappa}(x) - \prox_{U_{\theta}}^{\gamma\kappa}(y)} \leq \norm{x -y} \eqsp .
  \end{align}
The proof is then an application of \cite[Proposition 3b]{debortoli2019convergence} with $\ell \leftarrow 1$, for any $x \in \rset^d$, $\Tg(x) \leftarrow \prox_{U_{\theta}}^{\gamma \kappa}(x)- \gamma\nabla_x V_{\theta}(\prox_{U_{\theta}}^{\gamma\kappa}(x))$ and $\Pi \leftarrow \Id$.
\end{proof}

\begin{theorem}
  \label{thm:ergo_cv_pula}
  Assume \tup{\Cref{assum:potential_regularity}} and
  \tup{\Cref{assum:potential_drift_1}} or
  \tup{\Cref{assum:potential_drift_2}}. Let
  $\bkappa \geq 1 \geq \ukappa > 1/2$.  Let
  $\bgamma < 2 / (\mtt + \Lt)$ if \tup{\Cref{assum:potential_drift_1}}
  holds and $\bgamma < 2 / \Lt$ if
  \tup{\Cref{assum:potential_drift_2}} holds.  Then for any
  $a \in \ocint{0,1}$, there exist $A_{2, a} \geq 0$ and
  $\rho_{a} \in \ooint{0,1}$ such that for any $\theta \in \Theta$,
  $\kappa \in \ccint{\ukappa, \bkappa}$,
  $\gamma \in \ocint{0, \bgamma}$, $\Sker_{\gamma, \theta}$ and
  $\bSker_{\gamma, \theta}$ admit an invariant probability measure
  $\pi_{\gamma, \theta}$ and $\bpi_{\gamma, \theta}$ respectively, and
  for any $x, y \in \rset^{\dim}$ and $n \in \nset$ we have
\begin{equation}
  \begin{aligned}
        \max \parenthese{ \Vnorm[W^a]{\updelta_x \Sker_{\gamma, \theta}^n - \pi_{\gamma, \theta}}, \Vnorm[W^a]{\updelta_x \bSker_{\gamma, \theta}^n - \bpi_{\gamma, \theta}} } &\leq A_{2,  a} \rho_{ a}^{\gamma n} W^a(x) \eqsp ,  \\ 
    \max \parenthese{ \Vnorm[W^a]{\updelta_x \Sker_{\gamma, \theta}^n - \updelta_y \Sker_{\gamma, \theta}^n}, \Vnorm[W^a]{\updelta_x \bSker_{\gamma, \theta}^n - \updelta_y \bSker_{\gamma, \theta}^n}}  &\leq A_{2,  a} \rho_{ a}^{\gamma n}  \defEns{W^a(x) + W^a(y)} \eqsp , 
    \end{aligned}
    \end{equation}
    with $W =W_m$ and $\pow \in \nsets$ if
    \tup{\Cref{assum:potential_drift_1}} holds and $W = W_{\alpha}$
    with $\alpha < \ukappa \eta /4$ if
    \tup{\Cref{assum:potential_drift_2}} holds, see
    \eqref{eqII:def_W_pow_alph}.
\end{theorem}

\begin{proof}  
    The proof is similar to the one of \Cref{thm:ergo_cv_myula}.
  \end{proof}

\subsection{Checking \cite[H1, H2]{de2019efficient} for PULA}
\label{sec:check-citeh1d-pula}
\Cref{lemma:bornitude_pula} implies that
\cite[H1a]{de2019efficient} holds. The geometric
ergodicity proved in \Cref{thm:ergo_cv_pula} implies
\cite[H1b]{de2019efficient}. Then, we show that
the distance between the invariant probability distribution of the
Markov chain and the target distribution is controlled in
\Cref{coro:control_disc_pula} and therefore
\cite[H1c]{de2019efficient} is satisfied. Finally,
we show that \cite[H2]{de2019efficient} is
satisfied in \Cref{prop:error_kernel_pula}.

\begin{lemma}
  \label{lemma:bornitude_pula}
  Assume \tup{\Cref{assum:potential_regularity}},
  \tup{\Cref{assum:potential_drift_1}} or
  \tup{\Cref{assum:potential_drift_2}}, and let
  $(X_{k}^n, \bar{X}_k^n)_{n \in \nset, k \in \{0, \dots, m_n\}}$ be
  given by \eqref{eqII:algo_SOUL} with
  $\{(\Kker_{\gamma,\theta}, \bKker_{\gamma,\theta}) \, : \, \gamma
  \in \ocint{0,\bgamma}, \theta \in \Theta\} =
  \{(\Sker_{\gamma,\theta}, \bSker_{\gamma,\theta}) \, : \, \gamma \in
  \ocint{0,\bgamma}, \theta \in \Theta\}$ and
  $\kappa \in \ccint{\ukappa, \bkappa}$ with
  $\bkappa \geq 1 \geq \ukappa > 1/2$.  Then there exists $A_{1} \geq 1$ such that
  for any $n, p \in \nset$ and $k \in \{0, \dots, m_n\}$
  \begin{equation}
    \begin{aligned}
      &\CPE{\Sker_{\gamma_n, \theta_n}^p W(X_k^n)}{X_0^0} \leq A_{1} W(X_0^0) \eqsp , \\
      &\CPE{\bSker_{\gamma_n, \theta_n}^p W(\bar{X}_k^n)}{\bar{X}_0^0} \leq A_{1} W(\bar{X}_0^0) \eqsp , \\
      &\expe{W(X_0^0)} < +\infty\eqsp , \qquad \expe{W(\bX_0^0)} < +\infty \eqsp ,
      \end{aligned}
  \end{equation}
  with $W = W_{\pow}$ with $\pow \in \nsets$ and
  $\bgamma < 2 / (\mtt + \Lt)$ if \tup{\Cref{assum:potential_drift_1}}
  holds and $W = W_{\alpha}$ with $\alpha < \ukappa \eta /4$ and
  $\bgamma < 2 / \Lt$ if \tup{\Cref{assum:potential_drift_2}} holds,
  see \eqref{eqII:def_W_pow_alph}.
\end{lemma}

\begin{proof}
  Combining \cite[Lemma S15]{de2019efficient} and \Cref{lemma:drift_pula_1} if \Cref{assum:potential_drift_1} holds or \Cref{lemma:drift_pula_2} if \Cref{assum:potential_drift_2} holds conclude the proof.
\end{proof}

  \begin{lemma}
    \label{lemma:pi_majo}
    Assume \tup{\Cref{assum:potential_regularity}} and
    \tup{\Cref{assum:potential_drift_1}} or
    \tup{\Cref{assum:potential_drift_2}}.  We have
    $\sup_{\theta \in \Theta} \defEnsLigne{\pi_{\theta}(W) +
      \bpi_{\theta}(W)} < +\infty$, with $W = W_{\pow}$ with
    $\pow \in \nsets$ if \tup{\Cref{assum:potential_drift_1}} holds
    and $W = W_{\alpha}$ with $\alpha < \eta$ if
    \tup{\Cref{assum:potential_drift_2}} holds, see
    \eqref{eqII:def_W_pow_alph}.
  \end{lemma}

  \begin{proof}
    We only show that $\sup_{\theta} \pi_{\theta}(W) < +\infty$ since
    the proof for $\bpi_{\theta}$ is similar.  Let $\pow \in \nsets$,
    $\alpha < \eta$ and $\theta \in \Theta$ The proof is divided into
    two parts.
    \begin{enumerate}[label=(\alph*), wide, labelwidth=!, labelindent=0pt]
    \item If \Cref{assum:potential_drift_1} holds then using \Cref{assum:potential_regularity}-\ref{assum:potential_regularity:item:b} we have
      \begin{align}
        &\int_{\rset^{\dim}} (1 + \norm{x }^{2\pow}) \exp\parentheseDeux{- U_{\theta}(x) - V_{\theta}(x)} \rmd x \leq \int_{\rset^{\dim}} (1 + \norm{x}^{2\pow}) \exp\parentheseDeux{-V_{\theta}(x)} \rmd x \\
        & \qquad  \leq \int_{\rset^{\dim}} (1 + \norm{x}^{2\pow}) \exp\parentheseDeux{-V_{\theta}(\xvstar)  - \mtt\norm{x - \xvstar}^2/2} \rmd x \\
&  \qquad \leq \exp \parentheseDeux{ \Rvtrois  + \mtt \Rvun^2 /2 } \int_{\rset^{\dim}} (1 + \norm{x}^{2\pow}) \exp \parentheseDeux{ \mtt \Rvun  \norm{x} - \mtt \norm{x}^2/2} \rmd x  \eqsp .
      \end{align}
Hence using \Cref{assum:potential_regularity}-\ref{assum:potential_regularity:item:a} we have
\begin{multline}
  \sup_{\theta \in \Theta} \pi_{\theta}(W) \leq \exp \parentheseDeux{
    \Rvtrois + \mtt \Rvun^2 /2 } \int_{\rset^{\dim}} (1 +
  \norm{x}^{2\pow}) \exp \parentheseDeux{ \mtt \Rvun \norm{x} - \mtt
    \norm{x}^2/2} \rmd x \\ \left . \middle/ \inf_{\theta \in \Theta}
    \defEns{\int_{\rset^{\dim}} \exp\parentheseDeux{- U_{\theta}(x) -
        V_{\theta}(x)} \rmd x} < +\infty \right. \eqsp .
\end{multline}

    \item if \Cref{assum:potential_drift_2} holds then we have
      \begin{align}
        \int_{\rset^{\dim}} \exp\parentheseDeux{\alpha \phi(x)} \exp\parentheseDeux{- U_{\theta}(x) - V_{\theta}(x)} \rmd x &\leq \int_{\rset^{\dim}} \exp\parentheseDeux{\alpha \phi(x)} \exp\parentheseDeux{- U_{\theta}(x)} \rmd x \\ 
&\hspace{-1cm}  \leq\rme^{\ct} \int_{\rset^{\dim}} \exp\parentheseDeux{\alpha(1 + \norm{x})} \exp\parentheseDeux{- \eta \norm{x}} \rmd x \eqsp .
      \end{align}
    \end{enumerate}
Since $\alpha < \eta$ we have using \Cref{assum:potential_regularity}-\ref{assum:potential_regularity:item:a}
\begin{multline}
\sup_{\theta \in \Theta} \pi_{\theta}(W)  \leq \rme^{\ct} \int_{\rset^{\dim}} \exp\parentheseDeux{\alpha(1 + \norm{x})} \exp\parentheseDeux{- \eta \norm{x}} \rmd x \\ \left . \middle/ \inf_{\theta \in \Theta} \defEns{\int_{\rset^{\dim}}  \exp\parentheseDeux{- U_{\theta}(x) - V_{\theta}(x)} \rmd x} < +\infty \right. \eqsp,
\end{multline}
which concludes the proof.
  \end{proof}

  \begin{theorem}
    \label{thm:error_discret_pula_1}
    Assume \tup{\Cref{assum:potential_regularity}} and
    \tup{\Cref{assum:potential_drift_1}} or
    \tup{\Cref{assum:potential_drift_2}}.  Let
    $\bkappa \geq 1 \geq \ukappa > 1/2$. Let $\bgamma < 2 / (\mtt + \Lt)$ if
    \tup{\Cref{assum:potential_drift_1}} holds and $\bgamma < 2 / \Lt$
    if \tup{\Cref{assum:potential_drift_2}} holds.  Then for any
    $\theta \in \Theta$, $\kappa \in \ccint{\ukappa, \bkappa}$ and
    $\gamma \in \ocint{0, \bgamma}$ we have
    \begin{equation}
      \max \parenthese{\Vnorm[W^{1/2}]{\pi_{\gamma, \theta}^{\hash} - \pi_{\theta}}, \Vnorm[W^{1/2}]{\bpi_{\gamma, \theta}^{\hash} - \bpi_{\theta}}} \leq \tPsibf(\gamma) \eqsp ,
    \end{equation}
    where for any $\theta \in \Theta$ and $\gamma \in \ocint{0, \bgamma}$, $\pi_{\gamma, \theta}^{\hash}$, respectively $\bpi_{\gamma, \theta}^{\hash}$, is the invariant probability measure of $\Sker_{\gamma, \theta}$, respectively $\bSker_{\gamma, \theta}$,  given by \eqref{eqII:def_S_ker} and associated with $\kappa = 1$. In addition, for any $\gamma \in \ocint{0, \bgamma}$
    \begin{equation}\tPsibf(\gamma) = \sqrt{2} \defEnsLigne{b
        \lambda^{-\bgamma} / \log(1/\lambda) + \sup_{\theta \in
          \Theta} \pi_{ \theta}(W) + \sup_{\theta \in \Theta}
        \bpi_{\theta}(W)}^{1/2}(\Lt d + \Mt^2)^{1/2}\sqrt{\gamma}
      \eqsp , \end{equation} and where $W = W_{\pow}$ with
    $\pow \in \nsets$ and $\bgamma, \lambda, b$ are given in
    \Cref{lemma:drift_pula_1} if \tup{\Cref{assum:potential_drift_1}}
    holds and $W= W_{\alpha}$ with
    $\alpha < \min(\ukappa \eta /4, \eta)$ and $\bgamma, \lambda, b$ are given
    in \Cref{lemma:drift_pula_2} if
    \tup{\Cref{assum:potential_drift_2}} holds, see
    \eqref{eqII:def_W_pow_alph}.
  \end{theorem}

  \begin{proof}
    We only show that for any $\theta \in \Theta$,
    $\kappa \in \ccint{\ukappa, \bkappa}$ and
    $\gamma \in \ocint{0, \bgamma}$,
    $\Vnorm[W^{1/2}]{\pi_{\gamma, \theta}^{\hash} - \pi_{\theta}} \leq
    \tPsibf(\gamma)$, since the proof of
    $\Vnorm[W^{1/2}]{\tpi_{\gamma, \theta}^{\hash} - \tpi_{\theta}}
    \leq \tPsibf(\gamma)$ is similar.  Let $\theta \in \Theta$,
    $\kappa \in \ccint{\ukappa, \bkappa}$,
    $\gamma \in \ocint{0, \bgamma}$ and $x \in \rset^{\dim}$ Using
    \Cref{thm:ergo_cv_pula} we obtain that
    $(\updelta_x \Sker_{\gamma, \theta}^n)_{n \in \nset}$, with
    $\kappa = 1$, is weakly convergent towards
    $\pi_{\gamma, \theta}^{\hash}$. Using that
    $\mu \mapsto \KL{\mu}{\pi_{\theta}}$ is lower semi-continuous for
    any $\theta \in \Theta$, see \cite[Lemma 1.4.3b]{dupuis1997weak},
    and \cite[Corollary 18]{durmus2019analysis} we get that
    \begin{equation}
      \KL{\pi_{\gamma, \theta}^{\hash}}{\pi_{\theta}} \leq \liminf_{n \to +\infty} \KLbig{n^{-1}\sum_{k=1}^n \updelta_x \Sker_{\gamma, \theta}^k }{\pi_{\theta}} \leq \gamma (\Lt d + \Mt^2) \eqsp .
    \end{equation}
Using a generalized Pinsker inequality, see \cite[Lemma 24]{durmus2017nonasymptotic}, \Cref{lemma:pi_majo} and \Cref{lemma:drift_pula_1} if \Cref{assum:potential_drift_1} holds or \Cref{lemma:drift_pula_2} if \Cref{assum:potential_drift_2} holds, we get that
\begin{align}
  \Vnorm[W^{1/2}]{\pi_{\gamma, \theta}^{\hash} - \pi_{\theta}} &\leq \sqrt{2} (\pi_{\gamma, \theta}^{\hash}(W) + \pi_{\theta}(W))^{1/2} \KL{\pi_{\gamma, \theta}^{\hash}}{\pi_{\theta}}^{1/2}  \\
&\leq \sqrt{2} \defEnsLigne{b \lambda^{-\bgamma} / \log(1/\lambda) + \sup_{\theta \in \Theta} \pi_{\theta}(W)}^{1/2}(\Lt d  + \Mt^2)^{1/2}\gamma^{1/2} \eqsp ,
\end{align}
which concludes the proof.
\end{proof}

\begin{lemma}
  \label{prop:kl_error_pula}
  Assume \tup{\Cref{assum:potential_regularity}} and
  \tup{\Cref{assum:potential_drift_1}} or
  \tup{\Cref{assum:potential_drift_2}}. Let $\bkappa \geq 1 \geq \ukappa > 1/2$.
  Let $\bgamma < 2/ (\mtt + \Lt) $ if
  \tup{\Cref{assum:potential_drift_1}} holds and $\bgamma < 2 / \Lt$
  if \tup{\Cref{assum:potential_drift_2}} holds.  Then there exists
  $\bar{B}_3 \geq 0$ such that for any $\theta \in \Theta$,
  $\gamma\in \ocint{0, \bgamma}$, $x \in \rset^{\dim}$ and
  $\kappa_i \in \ccint{\ukappa, \bkappa}$ with $i \in \{1, 2\}$ we have
  \begin{equation}
    \max \parenthese{\Vnorm[W^{1/2}]{\updelta_x \Sker_{1, \gamma, \theta}^{\step} - \updelta_x \Sker_{2, \gamma,  \theta}^{\step}}, \Vnorm[W^{1/2}]{\updelta_x \bSker_{1, \gamma, \theta}^{\step} - \updelta_x \bSker_{2, \gamma,  \theta}^{\step}} } \leq \bar{B}_3 \gamma \abs{\kappa_1 - \kappa_2}W^{1/2}(x) \eqsp .
  \end{equation}
  where for any $i \in \{1, 2\}$, $\theta \in \Theta$ and
  $\gamma \in \ocint{0, \bgamma}$, $\Sker_{i, \gamma, \theta}$ is
  given by \eqref{eqII:def_S_ker} and associated with
  $\kappa \leftarrow \kappa_i$, and $W = W_{\pow}$ with
  $\pow \in \nsets$ if \tup{\Cref{assum:potential_drift_1}} holds. In
  addition, $W= W_{\alpha}$ with $\alpha < \min(\ukappa \eta /4, \eta)$
  if \tup{\Cref{assum:potential_drift_2}} holds, see
  \eqref{eqII:def_W_pow_alph}.
     \end{lemma}

     \begin{proof}
       We only show that for any $\theta \in \Theta$,
       $\gamma\in \ocint{0, \bgamma}$, $x \in \rset^{\dim}$ and
       $\kappa_i \in \ccint{\ukappa, \bkappa}$ with $i \in \{1, 2\}$
       we have
       $\Vnorm[W^{1/2}]{\updelta_x \Sker_{1, \gamma, \theta}^{\step} -
         \updelta_x \Sker_{2, \gamma, \theta}^{\step}} \leq \bar{B}_3
       \gamma \absLigne{\kappa_1 - \kappa_2}W^{1/2}(x)$ since the
       proof for $\bSker_{1, \gamma, \theta}$ and
       $\bSker_{2, \gamma, \theta}$ is similar.  Let
       $\theta \in \Theta$, $\gamma\in \ocint{0, \bgamma}$,
       $x \in \rset^{\dim}$ and
       $\kappa_i \in \ccint{\ukappa, \bkappa}$ with $i \in \{1,
       2\}$. Using a generalized Pinsker inequality, see \cite[Lemma
       24]{durmus2017nonasymptotic}, we have
       \begin{multline}
         \label{eqII:pinsker_gene}
         \Vnorm[W^{1/2}]{\updelta_x \Sker_{1, \gamma,  \theta}^{\step} - \updelta_x \Sker_{2, \gamma, \theta}^{\step}} \\ \leq \sqrt{2} (\Sker_{1, \gamma, \theta}^{\step}W(x) + \Sker_{2, \gamma, \theta}^{\step} W(x))^{1/2} \KL{\updelta_x \Sker_{1, \gamma, \theta}^{\step}}{\updelta_x \Sker_{2, \gamma, \theta}^{\step}}^{1/2} \eqsp .
       \end{multline}
Using \cite[Lemma 4.1]{kullback1959information} we get that $\KL{\updelta_x \Sker_{1, \gamma, \theta}^{\step}}{\updelta_x \Sker_{2, \gamma, \theta}^{\step}} \leq \KL{\tilde{\mu}_1}{\tilde{\mu}_2}$ where setting  $T = \gamma \step$, $\tilde{\mu}_i$, $i \in \{1, 2\}$, is the probability measure over $\mcb{\rmc([0,T], \rset^{\dim})}$ which is defined for any $\msa \in \mcb{\rmc([0,T], \rset^{\dim})}$ by $\tilde{\mu}_i(A) = \mathbb{P}((X_t^i)_{t \in \ccint{0,T}} \in \msa)$, $i \in \{1, 2\}$ and for any $t \in \ccint{0,T}$ 
\begin{equation}
  \rmd X_t^i = b_i(t, (X_s^i)_{s \in \ccint{0,T}}) \rmd t +  \sqrt{2} \rmd B_t \eqsp , \qquad X_0^i = x \eqsp ,
\end{equation}
with for any $(\omega_s)_{s \in \ccint{0,T}} \in \rmc(\ccint{0,T}, \rset^{\dim})$ and $t \in \ccint{0,T}$
\begin{equation}
  b_i(t, (\omega_s)_{s \in \ccint{0,T}}) = \sum_{p\in \nset} \1_{\coint{p \gamma, (p+1)\gamma}}(t) \cT(\prox_{U_{\theta}}^{\gamma \kappa_i}(\omega_{p\gamma})) \eqsp ,
\end{equation}
where for any $y \in \rset^{\dim}$, $\cT_{\gamma, \theta}(y) = y - \gamma \nabla_x V_{\theta}(y)$.
Since $(X_t^i)_{t \in \ccint{0,T}} \in \rmc(\ccint{0,T}, \rset^{\dim})$, $b_i$ and $b$ are continuous for any $i \in \{1, 2\}$,  \cite[Theorem 7.19]{liptser2001stat} applies and we obtain that $\tilde{\mu}_1 \ll \tilde{\mu}_2$ and
\begin{multline}
  \frac{\rmd \tilde{\mu}_1}{\rmd \tilde{\mu}_2}((X_t^1)_{t \in \ccint{0,T}}) = \exp\left\lbrace(1/4)\int_0^T \norm{b_1(t, (X_s^1)_{s \in \ccint{0,T}}) - b_2(t, (X_s^1)_{s \in \ccint{0,T}}) }^2 \rmd t  \right. \\ \left. + (1/2) \int_0^T \langle b_1(t, (X_s^1)_{s \in \ccint{0,T}}) - b_2(t, (X_s^1)_{s \in \ccint{0,T}}), \rmd X_t^1 \rangle  \right\rbrace \eqsp ,
\end{multline}
where the equality holds almost surely. 
As a consequence we obtain that
\begin{equation}
\label{eqII:kl_ineq}
  \KL{\tilde{\mu}_1}{\tilde{\mu}_2} = (1/4) \expe{\int_0^T \norm{b_1(t, (X_s^1)_{s \in \ccint{0,T}}) - b_2(t, (X_s^1)_{s \in \ccint{0,T}}) }^2 \rmd s } \eqsp .
\end{equation}
In addition, using \Cref{lemma:control_prox_gamma}, we have for any $(\omega_s)_{s \in \ccint{0,T}} \in \rmc(\ccint{0,T}, \rset^{\dim})$ and $t \in \ccint{0,T}$
\begin{align}
  &\norm{b_1(t, (\omega_s)_{s \in \ccint{0,T}}) - b_2(t, (\omega_s)_{s \in \ccint{0,T}}) }^2  = \norm{\cT_{\gamma, \theta}(\prox_{U_{\theta}}^{\gamma \kappa_1}(\omega_{\gamma\floor{t/\gamma}})) - \cT_{\gamma, \theta}(\prox_{U_{\theta}}^{\gamma \kappa_2}(\omega_{\gamma\floor{t/\gamma}}))}^2 \\
& \qquad \leq \norm{\prox_{U_{\theta}}^{\gamma \kappa_1}(\omega_{\gamma\floor{t/\gamma}}) - \prox_{U_{\theta}}^{\gamma \kappa_2}(\omega_{\gamma\floor{t/\gamma}})}^2  \leq 4 \gamma^2 (\kappa_1 - \kappa_2)^2 \Mt^2 \eqsp .
    \label{eqII:diff_source}
\end{align}
Combining this result and \eqref{eqII:kl_ineq} we get that
\begin{equation}
\label{eqII:kl_majo_final}
  \KL{\updelta_x \Sker_{1, \gamma, \theta}^{\step}}{\updelta_x \Sker_{2, \gamma, \theta}^{\step}} \leq (1 + \bgamma) \Mt^2 \gamma^2\abs{\kappa_1 - \kappa_2}^2 \eqsp .
\end{equation}
Combining \eqref{eqII:kl_majo_final} and \eqref{eqII:pinsker_gene} we get that
\begin{multline}
  \Vnorm[W^{1/2}]{\updelta_x \Sker_{1, \gamma, \theta}^{\step} - \updelta_x \Sker_{2, \gamma, \theta}^{\step}} \\ \leq 2^{1/2} (1 + \bgamma)^{1/2} \Mt (\Sker_{1, \gamma, \theta}^{\step}W(x) + \Sker_{2, \gamma, \theta}^{\step}W(x))^{1/2} \gamma \abs{\kappa_1 - \kappa_2}  \eqsp .
\end{multline}
We conclude the proof upon using \Cref{lemma:majo_step}, and
\Cref{lemma:drift_pula_1} if \Cref{assum:potential_drift_1} holds, or
\Cref{lemma:drift_pula_2} if \Cref{assum:potential_drift_2} holds.

\begin{proposition}
\label{prop:error_intra_pula}
Assume \tup{\Cref{assum:potential_regularity}} and
\tup{\Cref{assum:potential_drift_1}} or
\tup{\Cref{assum:potential_drift_2}}. Let $\bkappa \geq 1 \geq \ukappa > 1/2$.
Let $\bgamma < 2 / (\mtt + \Lt)$ if
\tup{\Cref{assum:potential_drift_1}} holds and $\bgamma < 2 / \Lt$ if
\tup{\Cref{assum:potential_drift_2}} holds.  Then there exists
$B_3 \geq 0$ such that for any $\theta \in \Theta$,
$\gamma\in \ocint{0, \bgamma}$ and
$\kappa_i \in \ccint{\ukappa, \bkappa}$ with $i \in \{1, 2\}$ we have
  \begin{equation}
    \max \parenthese{ \Vnorm[W^{1/2}]{\pi_{\gamma, \theta}^1 - \pi_{\gamma, \theta}^2}, \Vnorm[W^{1/2}]{\bpi_{\gamma, \theta}^1 - \bpi_{\gamma, \theta}^2} }\leq B_3 \gamma \abs{\kappa_1- \kappa_2} \eqsp ,
  \end{equation}
  where for any $i \in \{1, 2\}$, $\theta \in \Theta$ and
  $\gamma \in \ocint{0, \bgamma}$, $\pi_{\gamma, \theta}^i$,
  respectively $\bpi_{\gamma, \theta}^i$, is the invariant probability
  measure of $\Sker_{i, \gamma, \theta}$, respectively
  $\bSker_{i, \gamma, \theta}$, given by \eqref{eqII:def_S_ker} and
  associated with $\kappa \leftarrow \kappa_i$. In addition,
  $W= W_{\pow}$ with $\pow \in \nsets$ if
  \tup{\Cref{assum:potential_drift_1}} holds and $W = W_{\alpha}$ with
  $\alpha < \min(\ukappa \eta /4, \eta)$ if
  \tup{\Cref{assum:potential_drift_2}} holds, see
  \eqref{eqII:def_W_pow_alph}.
\end{proposition}
  
\begin{proof}
  We only show that for any $\theta \in \Theta$,
  $\gamma \in \ocint{0, \bgamma}$ and
  $\kappa_i \in \ccint{\ukappa, \bkappa}$ with $i \in \{1,2\}$,
  $\Vnorm[W^{1/2}]{\pi_{\gamma, \theta}^1 - \pi_{\gamma, \theta}^2}
  \leq B_3 \gamma \absLigne{\kappa_2 - \kappa_1}$ since the proof for
  $\bpi_{\gamma, \theta}^1$ and $\bpi_{\gamma, \theta}^2$ are similar.
  Let $\theta \in \Theta$, $\gamma\in \ocint{0, \bgamma}$,
  $x \in \rset^{\dim}$ and $\kappa_i > 1/2$. Using
  \Cref{thm:ergo_cv_pula} we have
  \begin{equation}
    \lim_{n\to +\infty} \Vnorm[W^{1/2}]{\updelta_x \Sker_{1, \gamma, \theta}^{n} - \updelta_x \Sker_{2, \gamma, \theta}^{n}} = \Vnorm[W^{1/2}]{\pi_{1, \gamma, \theta}- \pi_{2, \gamma, \theta}} \eqsp .
  \end{equation}
Let $n = q \step$. Using \Cref{thm:ergo_cv_pula} with $a = 1/2$, that $W^{1/2}(x) \leq W(x)$ for any $x \in \rset^{\dim}$, \Cref{prop:kl_error_pula}, \Cref{lemma:majo_step} and \Cref{lemma:drift_pula_1} if \Cref{assum:potential_drift_1} holds or \Cref{lemma:drift_pula_2} if \Cref{assum:potential_drift_2} holds, we have
\begin{align}
  &\Vnorm[W^{1/2}]{\updelta_x \Sker_{1, \gamma, \theta}^{n} - \updelta_x \Sker_{2, \gamma, \theta}^{n}} \leq \sum_{k=0}^{q-1} \Vnorm[W^{1/2}]{\updelta_x \Sker_{1, \gamma, \theta}^{(k+1)\step}\Sker_{2, \gamma, \theta}^{(q-k-1)\step} - \updelta_x \Sker_{1, \gamma, \theta}^{k\step}\Sker_{2, \gamma, \theta}^{(q-k)\step}} \\
&\leq \sum_{k=0}^{q-1} A_{2,1/2} \rho_{1/2}^{q-k-1} \VnormEq[W^{1/2}]{\updelta_x \Sker_{1, \gamma, \theta}^{k\step}\defEns{\Sker_{1, \gamma, \theta}^{\step} - \Sker_{2, \gamma, \theta}^{\step}}} \\ 
&\leq A_{2,1/2} \sum_{k=0}^{q-1}  \rho_{1/2}^{q-k-1} \bar{B}_3 \gamma \abs{\kappa_1 - \kappa_2} \updelta_x \Sker_{1, \gamma, \theta}^{k\step}W(x) \\
&\leq A_{2, 1/2} \sum_{k=0}^{q-1}  \rho_{1/2}^{q-k-1} \bar{B}_3 \gamma \abs{\kappa_1 - \kappa_2} (1 + b\lambda^{-\bgamma} / \log(1/\lambda))W(x) \\
&\leq A_{2, 1/2} \bar{B}_3 (1 + b\lambda^{-\bgamma} / \log(1/\lambda))/ (1 - \rho_{1/2}) \abs{\kappa_1 - \kappa_2} \gamma W(x) \eqsp ,
\end{align}
which concludes the proof with $B_3 = 2A_{2, 1/2} \bar{B}_3 (1 + b\lambda^{-\bgamma} / \log(1/\lambda))/ (1 - \rho_{1/2})\kappa$ upon setting $x = 0$.
\end{proof}

\begin{corollary}
  \label{coro:control_disc_pula}
    Assume \tup{\Cref{assum:potential_regularity}} and \tup{\Cref{assum:potential_drift_1}} or \tup{\Cref{assum:potential_drift_2}}. Let $\bkappa \geq 1 \geq \ukappa > 1/2$. 
  Let $\bgamma <  2 / (\mtt + \Lt) $ if \tup{\Cref{assum:potential_drift_1}} holds and $\bgamma < 2 / \Lt$ if \tup{\Cref{assum:potential_drift_2}} holds.
  Then 
  for any $\kappa \in \ccint{\ukappa, \bkappa}$, $\theta \in \Theta$ and $\gamma \in \ocint{0, \bgamma}$, we have
    \begin{equation}
      \max \parenthese{\Vnorm[W^{1/2}]{\pi_{\gamma, \theta} - \pi_{\theta}}, \Vnorm[W^{1/2}]{\bpi_{\gamma, \theta} - \bpi_{ \theta}}} \leq \Psibf(\gamma) \eqsp ,
    \end{equation}
    where for any $\gamma \in \ocint{0, \bgamma}$,
    $\pi_{\gamma, \theta}$ is the invariant probability measure of
    $\Sker_{\gamma, \theta}$ given by \eqref{eqII:def_S_ker}. In
    addition,
    $\Psibf(\gamma) = \tPsibf(\gamma) + B_3\gamma |\kappa - 1|$, where
    $\tPsibf$ is given in \Cref{thm:error_discret_pula_1} and $B_3$ in
    \Cref{prop:error_intra_pula}, and $W = W_{\pow}$ with
    $\pow \in \nsets$ if \tup{\Cref{assum:potential_drift_1}} holds
    and $W = W_{\alpha}$ with $\alpha < \min(\ukappa \eta /4, \eta)$
    if \tup{\Cref{assum:potential_drift_2}} holds, see
    \eqref{eqII:def_W_pow_alph}.
\end{corollary}

\begin{proof}
  We only show that for any $\theta\in \Theta$ and $\gamma \in \ocint{0, \bgamma}$ we have $\Vnorm[W^{1/2}]{\pi_{\gamma, \theta} - \pi_{\theta}} \leq \Psibf(\gamma)$ since the proof for $\bpi_{\gamma, \theta}$ and $\bpi_{\theta}$ are similar.
  Let $\kappa \in \ccint{\ukappa, \bkappa}$, $\theta \in \Theta$, $\gamma \in \ocint{0, \bgamma}$.
  The proof is a direct application of \Cref{thm:error_discret_pula_1} and \Cref{prop:error_intra_pula} upon noticing that
  \begin{equation}
    \Vnorm[W^{1/2}]{\pi_{\gamma, \theta} - \pi_{\theta}} \leq \Vnorm[W^{1/2}]{\pi_{\gamma, \theta} - \pi_{\gamma, \theta}^{\hash}} + \Vnorm[W^{1/2}]{\pi_{ \gamma,\theta}^{\hash} - \pi_{\theta}} \eqsp ,
  \end{equation}
where  $\pi_{\gamma, \theta}^{\hash}$ is the invariant probability measure of $\Sker_{\gamma, \theta}$ given by \eqref{eqII:def_S_ker} and associated with $\kappa = 1$.
\end{proof}

\begin{proposition}
  \label{prop:error_kernel_pula}
    Assume \tup{\Cref{assum:potential_regularity}} and \tup{\Cref{assum:potential_drift_1}} or \tup{\Cref{assum:potential_drift_2}}. Let $\bkappa \geq 1 \geq \ukappa > 1/2$.   Let $\bgamma <  2/(\mtt + \Lt)$ if \tup{\Cref{assum:potential_drift_1}} holds and $\bgamma < 2 / \Lt$ if \tup{\Cref{assum:potential_drift_2}} holds. Then there exists $A_4 \geq 0$ such that for any  $\kappa \in \ccint{\ukappa, \bkappa}$, $\theta_1, \theta_2 \in \Theta$, $\gamma_1, \gamma_2 \in \ocint{0, \bgamma}$ with $\gamma_2 < \gamma_1$, $a \in \ccint{1/4, 1/2}$ and $x \in \rset^{\dim}$
    \begin{multline}
      \max \parenthese{\Vnorm[W^a]{\updelta_x \Sker_{\gamma_1, \theta_1} - \updelta_x \Sker_{\gamma_2, \theta_2}}, \Vnorm[W^a]{\updelta_x \bSker_{\gamma_1, \theta_1} - \updelta_x \bSker_{\gamma_2, \theta_2}}} \\ \leq (\Lambdabf(\gamma_1, \gamma_2) + \Lambdabf(\gamma_1, \gamma_2) \norm{\theta_1 - \theta_2}) W^{2a}(x) \eqsp ,
    \end{multline}
    with
    \begin{equation}
        \Lambdabf_1(\gamma_1, \gamma_2) = A_4 (\gamma_1 / \gamma_2 - 1) \eqsp , \qquad 
        \Lambdabf_2(\gamma_1, \gamma_2) = A_4 \gamma_2^{1/2} \eqsp ,
    \end{equation}
    and where $W = W_{\pow}$ with $\pow \in \nset$ and $\pow \geq 2$
    if \tup{\Cref{assum:potential_drift_1}} is satisfied and
    $W = W_{\alpha}$ with $\alpha < \min(\ukappa \eta /4, \eta)$ if
    \tup{\Cref{assum:potential_drift_2}} is satisfied, see
    \eqref{eqII:def_W_pow_alph}.
\end{proposition}

\begin{proof}
  We only show that for any $\kappa \in \ccint{\ukappa, \bkappa}$,
  $\theta_1, \theta_2 \in \Theta$,
  $\gamma_1, \gamma_2 \in \ocint{0, \bgamma}$ with
  $\gamma_2 < \gamma_1$, $a \in \ccint{1/4, 1/2}$ and
  $x \in \rset^{\dim}$ we have
  $\Vnorm[W^a]{\updelta_x \Sker_{\gamma_1, \theta_1} - \updelta_x
    \Sker_{\gamma_2, \theta_2}} \leq (\Lambdabf(\gamma_1, \gamma_2) +
  \Lambdabf(\gamma_1, \gamma_2) \norm{\theta_1 - \theta_2}) W^{2a}(x)$
  since the proof for $\bSker_{\gamma_1, \theta_1}$ and
  $\bSker_{\gamma_2, \theta_2}$ is similar.  Let
  $a \in \ccint{1/4, 1/2}$, $\kappa \in \ccint{\ukappa, \bkappa}$,
  $\theta_1, \theta_2 \in \Theta$,
  $\gamma_1, \gamma_2 \in \ocint{0, \bgamma}$ with
  $\gamma_2 < \gamma_1$.  Using a generalized Pinsker inequality, see
  \cite[Lemma 24]{durmus2017nonasymptotic}, we have
  \begin{multline}
    \Vnorm[W^a]{\updelta_x \Sker_{\gamma_1, \theta_1} - \updelta_x \Sker_{\gamma_2,  \theta_2}} \\ \leq \sqrt{2} (\updelta_x \Sker_{\gamma_1, \theta_1}W^{2a}(x) + \updelta_x \Sker_{\gamma_2,  \theta_2}W^{2a}(x))^{1/2}\KL{\updelta_x \Sker_{\gamma_1, \theta_1}}{\updelta_x \Sker_{\gamma_2,  \theta_2}}^{1/2} \eqsp .
  \end{multline}
  Combining this result, Jensen's inequality and \Cref{lemma:drift_pula_1} if \Cref{assum:potential_drift_1} holds and \Cref{lemma:drift_pula_2} if \Cref{assum:potential_drift_2} holds, we obtain that 
  \begin{equation}
    \Vnorm[W^a]{ \Sker_{\gamma_1, \theta_1} - \Sker_{\gamma_2,  \theta_2}} \leq 2 (1 + b \bgamma)^{1/2} \defEns{\KL{\updelta_x \Sker_{\gamma_1, \theta_1}}{\updelta_x \Sker_{\gamma_2,  \theta_2}}}^{1/2} W^a(x)\eqsp .   
  \end{equation}
  Denote for $\upsilon \in \rset^{\dim}$ and $\upsigma >0$, $\Upsilon_{\upsilon, \upsigma}$ the $d$-dimensional Gaussian distribution with mean $\upsilon$ and covariance matrix $\upsigma^2 \Id$.
Using \Cref{lemma:kl_gauss} and the fact that $\gamma_1 \geq \gamma_2$ we have 
\begin{align}
\label{ineqII:kl_majo_kernel}
  &\KL{\updelta_x \Sker_{\gamma_1, \theta_1}}{\updelta_x \Sker_{\gamma_2,  \theta_2}} \\ & \qquad \qquad \leq d (\gamma_1 / \gamma_2 -1)^2/2 + \left . \norm{\cT_{\gamma_1, \theta_1}(\prox_{U_{\theta_1}}^{\gamma_1 \kappa}(x)) - \cT_{\gamma_2, \theta_2}(\prox_{U_{\theta_1}}^{\gamma_2\kappa}(x))}^2 \middle / (4 \gamma_2) \right . \eqsp , 
\end{align}
with $\cT_{\gamma, \theta}(z) = z - \gamma \nabla_x V_{\theta}(z)$ for any $\theta \in \Theta$, $\gamma \in \ocint{0, \bgamma}$ and $x \in \rset^{\dim}$.
We have
\begin{align}
  \label{eqII:ineq_numero_zero}       & \qquad (1/4) \norm{\cT_{\gamma_1, \theta_1}(\prox_{U_{\theta_1}}^{\gamma_1 \kappa}(x)) - \cT_{\gamma_2, \theta_2}(\prox_{U_{\theta_2}}^{\gamma_2 \kappa}(x))}^2 \\ &  \quad \leq  \norm{\cT_{\gamma_1, \theta_1}(\prox_{U_{\theta_1}}^{\gamma_1 \kappa}(x)) - \cT_{\gamma_1, \theta_1}(\prox_{U_{\theta_1}}^{\gamma_2 \kappa}(x))}^2
     +  \norm{\cT_{\gamma_1, \theta_1}(\prox_{U_{\theta_1}}^{\gamma_2 \kappa}(x)) - \cT_{\gamma_1, \theta_1}(\prox_{U_{\theta_2}}^{\gamma_2 \kappa}(x))}^2
  \\ & \quad +  \norm{\cT_{\gamma_1, \theta_1}(\prox_{U_{\theta_2}}^{\gamma_2 \kappa}(x)) - \cT_{\gamma_2, \theta_1}(\prox_{U_{\theta_2}}^{\gamma_2 \kappa}(x))}^2 +  \norm{\cT_{\gamma_2, \theta_1}(\prox_{U_{\theta_2}}^{\gamma_2 \kappa}(x)) - \cT_{\gamma_2, \theta_2}(\prox_{U_{\theta_2}}^{\gamma_2 \kappa}(x))}^2 \eqsp .
\end{align}
First using \Cref{assum:potential_regularity}, \cite[Theorem 2.1.5, Equation (2.1.8)]{nesterov2013introductory} and \Cref{lemma:control_prox_gamma} we have
\begin{align}
  \label{eqII:ineq_numero_uno}
  &\norm{\cT_{\gamma_1, \theta_1}(\prox_{U_{\theta_1}}^{\gamma_1 \kappa}(x)) - \cT_{\gamma_1, \theta_1}(\prox_{U_{\theta_1}}^{\gamma_2 \kappa}(x))} \\ & \qquad \qquad \qquad \leq \norm{\prox_{U_{\theta_1}}^{\gamma_1 \kappa}(x) - \prox_{U_{\theta_1}}^{\gamma_2 \kappa}(x)} \leq 2 \Mt \abs{\gamma_1 \kappa - \gamma_2 \kappa}\eqsp .
\end{align}
Second, we have using \eqref{eqII:part_ii_grad_prox}, \Cref{assum:potential_regularity}, \cite[Theorem 2.1.5, Equation (2.1.8)]{nesterov2013introductory} and \Cref{assum:theta_reg}
\begin{align}
  \label{eqII:ineq_numero_duo}
  &\qquad \norm{\cT_{\gamma_1, \theta_1}(\prox_{U_{\theta_1}}^{\gamma_2 \kappa}(x)) - \cT_{\gamma_1, \theta_1}(\prox_{U_{\theta_2}}^{\gamma_2 \kappa}(x))} \\ & \qquad \qquad \leq \gamma_2 \kappa \norm{\nabla_x U_{\theta_1}^{\gamma_2\kappa}(x) - \nabla_x U_{\theta_2}^{\gamma_2\kappa}(x)} \leq \sup_{t \in \ccint{0, \bgamma \kappa}} \defEnsLigne{\ftt_{\theta}(t)}  \gamma_2 \kappa \norm{\theta_1 - \theta_2}(1 + \norm{x}) \eqsp .
\end{align}
Third using \Cref{assum:potential_regularity} and \Cref{lemma:borne_prox} we have that
\begin{align}
  \label{eqII:ineq_numero_tertio}
  \norm{\cT_{\gamma_1, \theta_1}(\prox_{U_{\theta_2}}^{\gamma_2 \kappa}(x)) -  \cT_{\gamma_2, \theta_1}(\prox_{U_{\theta_2}}^{\gamma_2 \kappa}(x))} &\leq (\gamma_1 - \gamma_2) \norm{\nabla_x V_{\theta_1}(\prox_{U_{\theta_2}}^{\gamma_2 \kappa}(x))} \\
                                                                                                                                                                        &\leq (\gamma_1 - \gamma_2) \Lt \norm{\prox_{U_{\theta_2}}^{\gamma_2 \kappa}(x) - x_{\theta_1}^{\star}} \\
                                                                                                                                                                        &\leq (\gamma_1 - \gamma_2) \Lt (\Rvun + \bgamma \kappa \Mt + \norm{x}) \eqsp  .                                                            
\end{align}
Finally using \Cref{assum:potential_regularity}, \Cref{assum:theta_reg} and \Cref{lemma:borne_prox} we have that
\begin{align}
  \label{eqII:ineq_numero_quatro}
  &\norm{\cT_{\gamma_2, \theta_1}(\prox_{U_{\theta_2}}^{\gamma_2 \kappa}(x)) -  \cT_{\gamma_2, \theta_2}(\prox_{U_{\theta_2}}^{\gamma_2 \kappa}(x))} \\
  & \qquad \leq \gamma_2 \norm{\nabla_x V_{\theta_1}(\prox_{U_{\theta_2}}^{\gamma_2 \kappa}(x)) - \nabla_x V_{\theta_2}(\prox_{U_{\theta_2}}^{\gamma_2 \kappa}(x))}  \\
                                                                                                                                                                        & \qquad \leq \gamma_2 \Munu \norm{\theta_1 - \theta_2} (1 + \normLigne{\prox_{U_{\theta_2}}^{\gamma_2 \kappa}(x)}) \leq \gamma_2 \Munu \norm{\theta_1 - \theta_2} (1 + \bgamma \kappa \Mt + \norm{x}) \eqsp .  
\end{align}
Therefore, combining \eqref{eqII:ineq_numero_uno}, \eqref{eqII:ineq_numero_duo}, \eqref{eqII:ineq_numero_tertio} and \eqref{eqII:ineq_numero_quatro} in \eqref{eqII:ineq_numero_zero}, there exists $A_{4, 1} \geq 0$ such that for any $\gamma_1, \gamma_2 >0$ with $\gamma_2 < \gamma_1$ and $\theta_1, \theta_2 \in \Theta$
\begin{equation}
  \norm{\cT_{\gamma_1, \theta_1}(\prox_{U_{\theta_1}}^{\gamma_1 \kappa}(x)) - \cT_{\gamma_2, \theta_2}(\prox_{U_{\theta_2}}^{\gamma_2 \kappa}(x))}^2 \leq A_{4,1}\parentheseDeux{(\gamma_1 - \gamma_2)^2 + \gamma_2^2 \norm{\theta_1 - \theta_2}^2} W^{2a}(x) \eqsp .
\end{equation}
Using this result in \eqref{ineqII:kl_majo_kernel}, there exists $A_{4, 2} \geq 0$ such that
\begin{equation}
  \KL{\updelta_x \Sker_{\gamma_1, \theta_1}}{\updelta_x \Sker_{\gamma_2,  \theta_2}} \leq A_{4, 2} \parentheseDeux{(\gamma_1/\gamma_2-1)^2 + \gamma_2 \norm{\theta_1 - \theta_2}^2 } W^{2a}(x) \eqsp ,
\end{equation}
which implies the announced result upon setting $A_4 = 2\sqrt{A_{4,2}}(1 + b \bgamma)^{1/2}$ and using that for any $u,v \geq 0$, $\sqrt{u + v} \leq \sqrt{u} + \sqrt{v}$. 
\end{proof}

\subsection{Checking \cite[H1, H2]{de2019efficient} for MYULA}
\label{sec:check-citeh1d-myula}

In this section, similarly to \Cref{sec:check-citeh1d-myula} for
PULA, we show that \cite[H1, H2]{de2019efficient}
hold for MYULA.

\begin{lemma}
  \label{lemma:bornitude_myula}
  Assume \tup{\Cref{assum:potential_regularity}},
  \tup{\Cref{assum:potential_drift_1}} or
  \tup{\Cref{assum:potential_drift_2}}, and let
  $(X_{k}^n, \bar{X}_k^n)_{n \in \nset, k \in \{0, \dots, m_n\}}$ be
  given by \eqref{eqII:algo_SOUL} with
  $\{(\Kker_{\gamma,\theta}, \bKker_{\gamma,\theta}) \, : \, \gamma
  \in \ocint{0,\bgamma}, \theta \in \Theta\} =
  \{(\Rker_{\gamma,\theta}, \bRker_{\gamma,\theta}) \, : \, \gamma \in
  \ocint{0,\bgamma}, \theta \in \Theta\}$ and
  $\kappa \in \ccint{\ukappa, \bkappa}$ with
  $\bkappa \geq 1 \geq \ukappa > 1/2$.  Then there exists $\bA_{1} \geq 1$ such
  that for any $n, p \in \nset$ and $k \in \{0, \dots, m_n\}$
  \begin{equation}
    \begin{aligned}
      &\CPE{\Rker_{\gamma_n, \theta_n}^p W(X_k^n)}{X_0^0} \leq \bA_{1} W(X_0^0) \eqsp , \\
      &\CPE{\bRker_{\gamma_n, \theta_n}^p W(\bar{X}_k^n)}{\bar{X}_0^0} \leq \bA_{1} W(\bar{X}_0^0) \eqsp , \\
      &\expe{W(X_0^0)} < +\infty\eqsp , \qquad \expe{W(\bX_0^0)} < +\infty\eqsp .
      \end{aligned}
  \end{equation}
  with $W = W_{\pow}$ with $\pow \in \nsets$ and
  $\bgamma < 2 / (\mtt + \Lt)$ if \tup{\Cref{assum:potential_drift_1}}
  holds and $W = W_{\alpha}$ with
  $\alpha < \min(\ukappa \eta /4, \eta /8)$ and
  $\bgamma < \min \defEnsLigne{2 / \Lt, \eta / (2 \Mt \Lt)}$ if
  \tup{\Cref{assum:potential_drift_2}} holds, see
  \eqref{eqII:def_W_pow_alph}.
\end{lemma}

  \begin{proposition}
    \label{prop:error_intra_myula}
    Assume \tup{\Cref{assum:potential_regularity}} and
    \tup{\Cref{assum:potential_drift_1}} or
    \tup{\Cref{assum:potential_drift_2}}. Let
    $\bkappa \geq 1 \geq \ukappa > 1/2$. Let
    $\bgamma < \min\defEnsLigne{(2 - 1/\ukappa)/\Lt, 2 / (\mtt +
      \Lt)}$ if \tup{\Cref{assum:potential_drift_1}} holds and
    $\bgamma < \min \defEnsLigne{(2 - 1 / \ukappa) / \Lt, \eta / (2
      \Mt \Lt)}$ if \tup{\Cref{assum:potential_drift_2}} holds.  Then
    there exists $\bB_{3,1} \geq 0$ such that for any
    $\theta \in \Theta$, $\kappa_i \in \ccint{\ukappa, \bkappa}$,
    $\gamma\in \ocint{0, \bgamma}$
  \begin{equation}
    \max \parenthese{ \Vnorm[W^{1/2}]{\pi_{\gamma, \theta}^1 - \pi_{\gamma, \theta}^2}, \Vnorm[W^{1/2}]{\bpi_{\gamma, \theta}^1 - \bpi_{\gamma, \theta}^2} }\leq \bB_{3,1} \gamma \eqsp ,
  \end{equation}
  where for any $i \in \{1, 2\}$, $\theta \in \Theta$ and
  $\gamma \in \ocint{0, \bgamma}$, $\pi_{\gamma, \theta}^i$,
  respectively $\bpi_{\gamma, \theta}^i$, is the invariant probability
  measure of $\Rker_{i, \gamma, \theta}$, respectively
  $\bRker_{i, \gamma, \theta}$, given by \eqref{eqII:def_R_ker} and
  associated with $\kappa \leftarrow \kappa_i$. In addition,
  $W= W_{\pow}$ with $\pow \in \nsets$ if
  \tup{\Cref{assum:potential_drift_1}} holds and $W = W_{\alpha}$ with
  $\alpha < \min(\ukappa \eta /4, \eta /8)$ if
  \tup{\Cref{assum:potential_drift_2}} holds, see
  \eqref{eqII:def_W_pow_alph}.
     \end{proposition}

     \begin{proof}
       The proof is similar to the one of \Cref{prop:error_intra_pula} upon setting for any $ i\in \{1, 2\}$ and $(\omega_s)_{s \in \ccint{0,T}} \in \rmc(\ccint{0,T}, \rset^{\dim})$ with $T = \gamma \step$ 
       \begin{equation}
         b_i(t, (\omega_s)_{s \in \ccint{0,T}}) = \omega_{\floor{t/\gamma}\gamma} - \gamma \nabla_x V_{\theta}(\omega_{\floor{t / \gamma}\gamma}) - \gamma \nabla_x U_{\theta}^{\gamma \kappa_i(\gamma)}(\omega_{\floor{t/\gamma}\gamma}) \eqsp ,
       \end{equation}
       and replacing \eqref{eqII:diff_source} in \Cref{prop:kl_error_pula} by
\begin{multline}
  \norm{b_1(t, (\omega_s)_{s \in \ccint{0,T}}) - b_2(t, (\omega_s)_{s \in \ccint{0,T}}) }^2 \\= \norm{- \gamma \nabla_x U_{\theta}^{\gamma \kappa_1}(\omega_{\floor{t/\gamma}\gamma}) +   \gamma \nabla_x U_{\theta}^{\gamma \kappa_2}(\omega_{\floor{t/\gamma}\gamma})}^2  \leq 4 \gamma^2 \Mt^2 \eqsp .
\end{multline}       
     \end{proof}

\begin{proposition}
  \label{prop:error_inter_myula_pula}
  Assume \tup{\Cref{assum:potential_regularity}} and
  \tup{\Cref{assum:potential_drift_1}} or
  \tup{\Cref{assum:potential_drift_2}}. Let
  $\bkappa \geq 1 \geq \ukappa > 1/2$.  Let
  $\bgamma < \min\defEnsLigne{(2 - 1/\kappa)/\Lt, 2 / (\mtt + \Lt),
    \Lt^{-1}}$ if \tup{\Cref{assum:potential_drift_1}} holds and
  $\bgamma < \min \defEnsLigne{(2 - 1 / \kappa) / \Lt, \eta / (2 \Mt
    \Lt), \Lt^{-1}}$ if \tup{\Cref{assum:potential_drift_2}} holds.
  Then there exists $\bB_{3,2} \geq 0$ such that for any
  $\theta \in \Theta$, $\gamma\in \ocint{0, \bgamma}$ and
  $\kappa_i \in \ccint{\ukappa, \bkappa}$ with $i \in \{1, 2\}$ we have
  \begin{equation}
    \max \parenthese{ \Vnorm[W^{1/2}]{\pi_{\gamma, \theta}^{\flat} - \pi_{\gamma, \theta}^{\hash}}, \Vnorm[W^{1/2}]{\bpi_{\gamma, \theta}^{\flat} - \bpi_{\gamma, \theta}^{\hash}} }\leq \bB_{3,2} \gamma^2 \eqsp ,
  \end{equation}
  where for any $\theta \in \Theta$ and
  $\gamma \in \ocint{0, \bgamma}$, $\pi_{\gamma, \theta}^{\flat}$,
  respectively $\bpi_{\gamma, \theta}^{\flat}$, is the invariant
  probability measure of $\Rker_{\gamma, \theta}$, respectively
  $\bRker_{\gamma, \theta}$, given by \eqref{eqII:def_R_ker} and
  associated with $\kappa = 1$ and $\pi_{\gamma, \theta}^{\hash}$,
  respectively $\bpi_{\gamma, \theta}^{\hash}$, is the invariant
  probability measure of $\Sker_{\gamma, \theta}$, respectively
  $\bSker_{\gamma, \theta}$, given by \eqref{eqII:def_S_ker} and
  associated with $\kappa = 1$. In addition, $W= W_{\pow}$ with
  $\pow \in \nsets$ if \tup{\Cref{assum:potential_drift_1}} holds and
  $W = W_{\alpha}$ with $\alpha < \min(\ukappa \eta /4, \eta /8)$ if
  \tup{\Cref{assum:potential_drift_2}} holds, see
  \eqref{eqII:def_W_pow_alph}.
     \end{proposition}

     \begin{proof}
       The proof is similar to the one of \Cref{prop:error_intra_pula} upon setting for any $(\omega_s)_{s \in \ccint{0,T}} \in \rmc(\ccint{0,T}, \rset^{\dim})$ with $T = \gamma \step$ 
       \begin{equation}
         \begin{aligned}
         b_1(t, (\omega_s)_{s \in \ccint{0,T}}) &= \prox_{U_{\theta}}^{\gamma}(\omega_{\floor{t/\gamma}\gamma}) - \gamma \nabla_x V_{\theta}(\prox_{U_{\theta}}^{\gamma}(\omega_{\floor{t / \gamma}\gamma})) \eqsp , \\            
         b_2(t, (\omega_s)_{s \in \ccint{0,T}}) &= \omega_{\floor{t/\gamma}\gamma} - \gamma \nabla_x V_{\theta}(\omega_{\floor{t / \gamma}\gamma}) - \gamma \nabla_x U_{\theta}^{\gamma}(\omega_{\floor{t/\gamma}\gamma})\eqsp ,           
         \end{aligned}
       \end{equation}
       and replacing \eqref{eqII:diff_source} in \Cref{prop:kl_error_pula} and using \eqref{eqII:part_ii_grad_prox} and \Cref{lemma:borne_prox} we get 
\begin{align}
  &\norm{b_1(t, (\omega_s)_{s \in \ccint{0,T}}) - b_2(t, (\omega_s)_{s \in \ccint{0,T}}) }^2 \\ &= \| \prox_{U_{\theta}}^{\gamma}(\omega_{\floor{t / \gamma}\gamma})) -\gamma \nabla_x \V_{\theta}(\prox_{U_{\theta}}^{\gamma}(\omega_{\floor{t / \gamma}\gamma})) - \omega_{\floor{t / \gamma}\gamma} \\
  & \qquad + \gamma \nabla_x V_{\theta}(\omega_{\floor{t / \gamma}\gamma})) + \gamma (\omega_{\floor{t / \gamma}\gamma} - \prox_{U_{\theta}}^{\gamma}(\omega_{\floor{t / \gamma}\gamma}))/ \gamma \| ^2 \\
  &=\gamma^2\norm{\nabla_x \V_{\theta}(\prox_{U_{\theta}}^{\gamma}(\omega_{\floor{t / \gamma}\gamma}))) -  \nabla_x V_{\theta}(\omega_{\floor{t / \gamma}\gamma})) }^2 \leq \Lt^2 \Mt^2 \gamma^4 \eqsp .
\end{align}       
     \end{proof}
     
     \begin{proposition}
       \label{prop:error_disc_myula}
       Assume \tup{\Cref{assum:potential_regularity}} and
       \tup{\Cref{assum:potential_drift_1}} or
       \tup{\Cref{assum:potential_drift_2}}. Let
       $\bkappa \geq 1 \geq \ukappa > 1/2$.  Let
       $\bgamma < \min\defEnsLigne{(2-1/\ukappa)/\Lt, 2 / (\mtt +
         \Lt), \Lt^{-1}}$ if \tup{\Cref{assum:potential_drift_1}}
       holds and
       $\bgamma < \min \defEnsLigne{(2 - 1 / \ukappa) / \Lt, \eta / (2
         \Mt \Lt), \Lt^{-1}}$ if \tup{\Cref{assum:potential_drift_2}}
       holds.  Then for any $\theta \in \Theta$,
       $\kappa \in \ccint{\ukappa, \bkappa}$ and
       $\gamma \in \ocint{0, \bgamma}$, we have
    \begin{equation}
      \max \parenthese{\Vnorm[W^{1/2}]{\pi_{\gamma, \theta} - \pi_{\theta}}, \Vnorm[W^{1/2}]{\bpi_{\gamma, \theta} - \bpi_{ \theta}}} \leq \bPsibf(\gamma) \eqsp ,
    \end{equation}
      where for any $i \in \{1, 2\}$, $\theta \in \Theta$ and $\gamma \in \ocint{0, \bgamma}$, $\pi_{\gamma, \theta}^i$, respectively $\bpi_{\gamma, \theta}^i$, is the invariant probability measure of $\Rker_{i, \gamma, \theta}$, respectively $\bRker_{i, \gamma, \theta}$,  given by \eqref{eqII:def_R_ker} and associated with $\kappa \leftarrow  \kappa_i$. In addition,
    $\bPsibf(\gamma) = \tPsibf(\gamma) + \bB_{3,1} \gamma + \bB_{3,2} \gamma^2$, where $\tPsibf$ is given in \Cref{thm:error_discret_pula_1} and $B_3$ in \Cref{prop:error_intra_pula}, and  $W = W_{\pow}$ with $\pow \in \nsets$ if \tup{\Cref{assum:potential_drift_1}} holds and $W = W_{\alpha}$ with $\alpha < \min(\ukappa \eta /4, \eta /8)$ if \tup{\Cref{assum:potential_drift_2}} holds, see \eqref{eqII:def_W_pow_alph}.
     \end{proposition}
     \begin{proof}
       We only show that for any $\theta \in \Theta$ and
       $\gamma \in \ocint{0,\bgamma}$,
       $\Vnorm[W^{1/2}]{\pi_{\gamma, \theta} - \pi_{\theta}} \leq
       \bPsibf(\gamma)$ as the proof for $\bpi_{\gamma, \theta}$ and
       $\bpi_{\theta}$ is similar.  First note that for any
       $\theta \in \Theta$, $\kappa \in \ccint{\ukappa, \bkappa}$ and
       $\gamma\in \ocint{0,
         \bgamma}$ 
       we have
       \begin{equation}
         \Vnorm[W^{1/2}]{\pi_{\gamma, \theta}- \pi_{\theta}} \leq \Vnorm[W^{1/2}]{\pi_{\gamma, \theta}- \pi_{\gamma, \theta}^{\flat}} + \Vnorm[W^{1/2}]{\pi_{\gamma, \theta}^{\flat}- \pi_{\gamma, \theta}^{\hash}} + \Vnorm[W^{1/2}]{\pi_{\gamma, \theta}^{\sharp}- \pi_{\theta}} \eqsp ,
       \end{equation}
         where for any $\theta \in \Theta$ and $\gamma \in \ocint{0, \bgamma}$, $\pi_{\gamma, \theta}^{\flat}$ is the invariant probability measure of $\Rker_{\gamma, \theta}$  given by \eqref{eqII:def_R_ker} and associated with $\kappa = 1$ and $\pi_{\gamma, \theta}^{\hash}$ is the invariant probability measure of $\Sker_{\gamma, \theta}$ and associated with $\kappa = 1$.
       We conclude the proof upon combining \Cref{prop:error_intra_myula}, \Cref{prop:error_inter_myula_pula} and \Cref{thm:error_discret_pula_1}.
     \end{proof}

\begin{proposition}
  \label{prop:error_kernel_myula}
  Assume \tup{\Cref{assum:potential_regularity}} and
  \tup{\Cref{assum:potential_drift_1}} or
  \tup{\Cref{assum:potential_drift_2}}.  Let
  $\bkappa \geq 1 \geq \ukappa > 1/2$. Let
  $\bgamma < \min\defEnsLigne{(2-1/\ukappa)/\Lt, 2/(\mtt + \Lt)}$ if
  \tup{\Cref{assum:potential_drift_1}} holds and
  $\bgamma < \min \defEnsLigne{(2 - 1 / \ukappa) / \Lt, \eta / (2 \Mt
    \Lt)}$ if \tup{\Cref{assum:potential_drift_2}} holds. Then there
  exists $\bA_4 \geq 0$ such that for any
  $\theta_1, \theta_2 \in \Theta$,
  $\kappa \in \ccint{\ukappa, \bkappa}$,
  $\gamma_1, \gamma_2 \in \ocint{0, \bgamma}$ with
  $\gamma_2 < \gamma_1$, $a \in \ccint{1/4, 1/2}$ and
  $x \in \rset^{\dim}$
    \begin{multline}
      \max \parenthese{\Vnorm[W^a]{\updelta_x \Rker_{\gamma_1, \theta_1} - \updelta_x \Rker_{\gamma_2, \theta_2}}, \Vnorm[W^a]{\updelta_x \bRker_{\gamma_1, \theta_1} - \updelta_x \bRker_{\gamma_2, \theta_2}}} \\ \leq (\bLambdabf_1(\gamma_1, \gamma_2) + \bLambdabf_2(\gamma_1, \gamma_2) \norm{\theta_1 - \theta_2}) W^{2a}(x) \eqsp ,
    \end{multline}
    with
    \begin{equation}
      \begin{aligned}
        \bLambdabf_1(\gamma_1, \gamma_2) = \bA_4 (\gamma_1 / \gamma_2 - 1) 
        \eqsp , \qquad 
        \bLambdabf_2(\gamma_1, \gamma_2) = \bA_4 \gamma_2^{1/2} \eqsp ,
      \end{aligned}
    \end{equation}
    and where $W = W_{\pow}$ with $\pow \in \nset$ and $\pow \geq 2$
    if \tup{\Cref{assum:potential_drift_1}} is satisfied and
    $W = W_{\alpha}$ with $\alpha < \min(\ukappa \eta /4, \eta /8)$ if
    \tup{\Cref{assum:potential_drift_2}} is satisfied, see
    \eqref{eqII:def_W_pow_alph}.
\end{proposition}

\begin{proof}
  First, note that we only show that for any $\theta_1, \theta_2 \in \Theta$,
  $\kappa \in \ccint{\bkappa, \ukappa}$,
  $\gamma_1, \gamma_2 \in \ocint{0, \bgamma}$ with
  $\gamma_2 < \gamma_1$, $a \in \ccint{1/4, 1/2}$ and
  $x \in \rset^{\dim}$, we have
  $\Vnorm[W^a]{\updelta_x \Rker_{\gamma_1, \theta_1} - \updelta_x
    \Rker_{\gamma_2, \theta_2}} \leq (\bLambdabf(\gamma_1, \gamma_2) +
  \bLambdabf(\gamma_1, \gamma_2) \norm{\theta_1 - \theta_2})
  W^{2a}(x)$ since the proof for $\bRker_{\gamma_1, \theta_1}$ and
  $\bRker_{\gamma_2, \theta_2}$ is similar.  Let
  $a \in \ccint{1/4, 1/2}$, $\theta_1, \theta_2 \in \Theta$,
  $\kappa \in \ccint{\ukappa, \bkappa}$,
  $\gamma_1, \gamma_2 \in \ocint{0, \bgamma}$ with
  $\gamma_2 < \gamma_1$.  Using a generalized Pinsker inequality
  \cite[Lemma 24]{durmus2017nonasymptotic} we have
  \begin{multline}
    \Vnorm[W^a]{\updelta_x \Rker_{\gamma_1, \theta_1} - \updelta_x \Rker_{\gamma_2, \theta_2}} \\ \leq \sqrt{2} (\updelta_x \Rker_{\gamma_1, \theta_1}W^{2a}(x) + \updelta_x \Rker_{\gamma_2, \theta_2}W^{2a}(x))^{1/2}\KL{\updelta_x \Rker_{\gamma_1, \theta_1}}{\updelta_x \Rker_{\gamma_2, \theta_2}}^{1/2} \eqsp .
  \end{multline}
  Combining this result, Jensen's inequality and \Cref{lemma:drift_pula_1} if \Cref{assum:potential_drift_1} holds and \Cref{lemma:drift_pula_2} if \Cref{assum:potential_drift_2} holds, we obtain that
  \begin{equation}
    \Vnorm[W^a]{\updelta_x \Rker_{\gamma_1, \theta_1} - \updelta_x \Rker_{\gamma_2, \theta_2}} \leq 2 (1 + b \bgamma)^{1/2} \KL{\updelta_x \Rker_{\gamma_1, \theta_1}}{\updelta_x \Rker_{\gamma_2, \theta_2}}^{1/2} W^a(x)\eqsp .   
  \end{equation}
  Using \Cref{lemma:kl_gauss} and the fact that $\gamma_1 \geq \gamma_2$ we have 
\begin{multline}
  \label{eqII:ineq_kl_myula}
  \KL{\updelta_x \Rker_{\gamma_1, \theta_1}}{\updelta_x \Rker_{\gamma_2, \theta_2}} \\ \leq d (\gamma_1 / \gamma_2 -1)^2/2 + \| \gamma_2 \nabla_x V_{\theta_2}(x) - \gamma_1 \nabla_x V_{\theta_1}(x) + \gamma_2 \nabla_x U_{\theta_2}^{\gamma_2\kappa}(x) - \gamma_1 \nabla_x U_{\theta_1}^{\gamma_1 \kappa}(x) \|^2 / (4 \gamma_2) \eqsp , 
\end{multline}
We have 
\begin{align}
  \label{eqII:ineq_numero_zero_myula}
  &\| \gamma_2 \nabla_x V_{\theta_2}(x) - \gamma_1 \nabla_x V_{\theta_1}(x) + \gamma_2 \nabla_x U_{\theta_2}^{\gamma_2\kappa}(x) - \gamma_1 \nabla_x U_{\theta_1}^{\gamma_1 \kappa}(x) \|^2  \\ & \qquad \qquad \leq 4 \norm{\gamma_2 \nabla_x V_{\theta_2}(x) - \gamma_2 \nabla_x V_{\theta_1}(x)}^2 + 4 \norm{\gamma_2 \nabla_x V_{\theta_1}(x) - \gamma_1 \nabla_x V_{\theta_1}(x)}^2  \\
                                                                                                                                                                                                           & \qquad \qquad + 4 \norm{\gamma_1 \nabla_x U_{\theta_1}^{\gamma_1 \kappa }(x) - \gamma_2 \nabla_x U_{\theta_1}^{\gamma_2 \kappa }(x)}^2  + 4 \norm{\gamma_2 \nabla_x U_{\theta_1}^{\gamma_2 \kappa }(x) - \gamma_2 \nabla_x U_{\theta_2}^{\gamma_2 \kappa }(x)}^2 \eqsp .  
\end{align}
First using \Cref{assum:theta_reg} we have
\begin{equation}
  \label{eqII:numero_uno_myula}
  \norm{\gamma_2 \nabla_x V_{\theta_2}(x) - \gamma_2 \nabla_x V_{\theta_1}(x)} \leq \gamma_2 \Munu \norm{\theta_1 - \theta_2} (1 + \norm{x}) \eqsp .
\end{equation}
Second using \Cref{assum:potential_regularity} we have
\begin{align}
  \label{eqII:numero_duo_myula}
  &\norm{\gamma_2 \nabla_x V_{\theta_1}(x) - \gamma_1 \nabla_x V_{\theta_1}(x)} \leq (\gamma_1 - \gamma_2) \norm{\nabla_x V_{\theta_1}(x)} \\
 & \qquad \qquad \leq (\gamma_1 - \gamma_2) \Lt \norm{x - x_{\theta_1}^{\star}} \leq (\gamma_1 - \gamma_2) \Lt (\Rvun + \norm{x}) \eqsp .
\end{align}
Third using \Cref{assum:potential_regularity}, \Cref{assum:theta_reg}, \Cref{lemma:borne_prox} and \Cref{lemma:control_prox_gamma} we have
\begin{align}
  \label{eqII:numero_tertio_myula}
  \norm{\gamma_1 \nabla_x U_{\theta_1}^{\gamma_1 \kappa}(x) - \gamma_2 \nabla_x U_{\theta_1}^{\gamma_2 \kappa}(x)} & \leq \norm{(x - \prox_{U_{\theta_1}}^{\gamma_1 \kappa}(x)) / \kappa - (x - \prox_{U_{\theta_1}}^{\gamma_2 \kappa}(x)) / \kappa} \\
  & \leq 
    \left . \norm{\prox_{U_{\theta_1}}^{\gamma_2 \kappa}(x) - \prox_{U_{\theta_1}}^{\gamma_1 \kappa}(x)} \middle / \kappa \right .\\
   &\leq 2 \Mt  (\gamma_1 - \gamma_2) 
\end{align}
Finally using \Cref{assum:theta_reg} we have
\begin{equation}
  \label{eqII:numero_quatro_myula}
  \norm{\gamma_2 \nabla_x U_{\theta_1}^{\gamma_2 \kappa }(x) - \gamma_2 \nabla_x U_{\theta_2}^{\gamma_2 \kappa }(x)} \leq \gamma_2 \defEns{\sup_{\ccint{0, \bgamma \kappa}}\ftt_{\theta}(t)} \norm{\theta_1 - \theta_2} \eqsp .
\end{equation}
Combining \eqref{eqII:numero_uno_myula}, \eqref{eqII:numero_duo_myula}, \eqref{eqII:numero_tertio_myula} and \eqref{eqII:numero_quatro_myula} in \eqref{eqII:ineq_numero_zero_myula} 
we get that there exists $\bA_{4,1} \geq 0$ such that 
\begin{multline}
  \| \gamma_2 \nabla_x V_{\theta_2}(x) - \gamma_1 \nabla_x V_{\theta_1}(x) + \gamma_2 \nabla_x U_{\theta_2}^{\kappa}(x) - \gamma_1 \nabla_x U_{\theta_1}^{\kappa}(x) \|^2 \\ \leq \bA_{4,1} \parentheseDeux{(\gamma_1 - \gamma_2)^2  + \gamma_2^2 \norm{\theta_1 - \theta_2} } W^{2a}(x)\eqsp .
\end{multline}
Using this result in \eqref{eqII:ineq_kl_myula} we obtain that there exists $\bA_{4, 2} \geq 0$ such that
\begin{equation}
  \KL{\updelta_x \Rker_{\gamma_1, \theta_1}}{\updelta_x \Rker_{\gamma_2, \theta_2}} \leq \bA_{4, 2} \left[ (\gamma_1/\gamma_2-1)^2  + \gamma_2 \norm{\theta_1 - \theta_2}^2 \right] W^{2a}(x) \eqsp ,
\end{equation}
which implies the announced result upon setting $\bA_4 = 2 \sqrt{\bA_{4,2}} (1+ b\bgamma)^{1/2}$ and using that for any $u,v \geq 0$, $\sqrt{u + v} \leq \sqrt{u} + \sqrt{v}$. 
\end{proof}

\subsection{Proof of \Cref{thm:cv_pula}}
\label{sec:proof_thm_cv}

We divide the proof in two parts.
\begin{enumerate}[label=(\alph*), wide, labelwidth=!, labelindent=0pt]
\item First assume that
  $(X_{k}^n)_{n \in \nset, k \in \{0, \dots, m_n\}}$ and
  $(\bX_{k}^n)_{n \in \nset, k \in \{0, \dots, m_n\}}$ are given by
  \eqref{eqII:algo_SOUL} and we have
  $\{(\Kker_{\gamma,\theta}, \bKker_{\gamma,\theta}) \, : \, \gamma
  \in \ocint{0,\bgamma}, \theta \in \Theta\} =
  \{(\Sker_{\gamma,\theta}, \bSker_{\gamma,\theta}) \, : \, \gamma \in
  \ocint{0,\bgamma}, \theta \in \Theta\}$. Then
  \Cref{lemma:bornitude_pula} implies that
  \cite[H1a]{de2019efficient} is satisfied with
  $A_1 \leftarrow A_1$, \Cref{thm:ergo_cv_pula} implies that
  \cite[H1b]{de2019efficient} holds with
  $A_2 \leftarrow A_2$ and $\rho \leftarrow \rho$. Finally, using
  \Cref{coro:control_disc_pula} we get that
  \cite[H1c]{de2019efficient} holds with
  $\Psibf \leftarrow \Psibf$. Therefore, we can apply \cite[Theorem
  1]{de2019efficient} and we obtain that the
  sequence $(\theta_n)_{n \in \nset}$ converges \as \ if
  \begin{equation}
    \sum_{n=0}^{+\infty} \delta_n = +\infty \eqsp , \qquad \sum_{n=0}^{+\infty} \delta_{n+1} \Psibf(\gamma_n)  < +\infty \eqsp , \qquad \sum_{n=0}^{+\infty} \delta_{n+1} / (m_n \gamma_n) < +\infty \eqsp .
  \end{equation}
  Since $\Psibf(\gamma_n) = \mathcal{O}(\gamma_n^{1/2})$ by \Cref{coro:control_disc_pula}, these summability conditions are satisfied under the summability assumptions of \Cref{thm:cv_pula}-\ref{item:item_1}. \Cref{prop:error_kernel_pula} implies that \cite[H2]{de2019efficient} holds with $\Lambdabf_1 \leftarrow \Lambdabf_1$ and $\Lambdabf_2 \leftarrow \Lambdabf_2$. Therefore if $m_n = m_0$ for all $n \in \nset$, we can apply \cite[Theorem 3]{de2019efficient} and we obtain that the sequence $(\theta_n)_{n \in \nset}$ converges \as \ if
  \begin{align}
    & \sum_{n=0}^{+\infty} \delta_n = +\infty \eqsp , \qquad \sum_{n=0}^{+\infty} \delta_{n+1} \Psibf(\gamma_n)  < +\infty \eqsp , \qquad \sum_{n=0}^{+\infty} \delta_{n+1} \gamma_n^{-2} < +\infty \\
      & \sum_{n=0}^{+\infty} \delta_{n+1} / \gamma_n^{2} (\Lambdabf_1(\gamma_n, \gamma_{n+1}) + \delta_{n+1} \Lambdabf_2(\gamma_n, \gamma_{n+1})) < +\infty \eqsp .
  \end{align}
These summability conditions are satisfied under the summability assumptions of \Cref{thm:cv_pula}~-\ref{item:item_2}.
\item Second assume that  $(X_{k}^n)_{n \in \nset, k \in \{0, \dots, m_n\}}$ and $(\bX_{k}^n)_{n \in \nset, k \in \{0, \dots, m_n\}}$ are given by \eqref{eqII:algo_SOUL} with $\{(\Kker_{\gamma,\theta}, \bKker_{\gamma,\theta}) \, : \, \gamma \in \ocint{0,\bgamma}, \theta \in \Theta\} = \{(\Rker_{\gamma,\theta}, \bRker_{\gamma,\theta}) \, : \, \gamma \in \ocint{0,\bgamma}, \theta \in \Theta\}$. Then \Cref{lemma:bornitude_myula} implies that \cite[H1a]{de2019efficient} is satisfied with $A_1 \leftarrow \bA_1$, \Cref{thm:ergo_cv_myula} implies that \cite[H1b]{de2019efficient} holds with $A_2 \leftarrow \bA_2$ and $\rho \leftarrow \brho$. Finally, using \Cref{prop:error_disc_myula} we get that \cite[H1c]{de2019efficient} holds with $\Psibf \leftarrow \bPsibf$. Therefore, we can apply \cite[Theorem 1]{de2019efficient} and we obtain that the sequence $(\theta_n)_{n \in \nset}$ converges \as \ if
  \begin{equation}
    \sum_{n=0}^{+\infty} \delta_n = +\infty \eqsp , \qquad \sum_{n=0}^{+\infty} \delta_{n+1} \bPsibf(\gamma_n)  < +\infty \eqsp , \qquad \sum_{n=0}^{+\infty} \delta_{n+1} / (m_n \gamma_n) < +\infty \eqsp .
  \end{equation}  
  Since $\Psibf(\gamma_n) = \mathcal{O}(\gamma_n^{1/2})$ by \Cref{prop:error_disc_myula}, these summability conditions are satisfied under the summability assumptions of \Cref{thm:cv_pula}-\ref{item:item_1}. \Cref{prop:error_kernel_myula} implies that \cite[H2]{de2019efficient} holds with $\Lambdabf_1 \leftarrow \bLambdabf_1$ and $\Lambdabf_2 \leftarrow \bLambdabf_2$. Therefore if $m_n = m_0$ for all $n \in \nset$, we can apply \cite[Theorem 3]{de2019efficient} and we obtain that the sequence $(\theta_n)_{n \in \nset}$ converges \as \ if
  \begin{equation}
    \begin{aligned}
      &\sum_{n=0}^{+\infty} \delta_n = +\infty \eqsp , \quad \sum_{n=0}^{+\infty} \delta_{n+1} \bPsibf(\gamma_n)  < +\infty \eqsp , \quad \sum_{n=0}^{+\infty} \delta_{n+1}^2 \gamma_n^{-2} \eqsp , \\
      & \sum_{n=0}^{+\infty} \delta_{n+1} / \gamma_n^2 (\bLambdabf_1(\gamma_n, \gamma_{n+1}) + \delta_{n+1} \bLambdabf_2(\gamma_n, \gamma_{n+1})) < +\infty \eqsp .
      \end{aligned}
  \end{equation}
These summability conditions are satisfied under the summability assumptions of \Cref{thm:cv_pula}-\ref{item:item_2}.
\end{enumerate}
\end{proof}

\subsection{Proof of \Cref{thm:error_pula}}
\label{thm:error_pula_proof}

The proof is similar to the one of \Cref{thm:cv_pula} using \cite[Theorem 2, Theorem 4]{debortoli2019convergence} instead of \cite[Theorem 1, Theorem 3]{debortoli2019convergence}.

\section{Acknowledgements}
AD acknowledges financial support from Polish National Science Center grant: NCN UMO-2018/31/B/ST1/00253. MP acknowledges financial support from EPSRC under grant EP/T007346/1.
\bibliographystyle{plain}
\bibliography{SIIMS_M133984_arxiv.bbl}
	
 \end{document}